\documentclass[11pt,reqno]{amsart}
\usepackage[T1]{fontenc}
\usepackage{amssymb,amsmath}
\usepackage[english]{babel}
\usepackage{enumerate}
\usepackage{enumitem,kantlipsum}
\usepackage[colorlinks=true,linkcolor=blue]{hyperref}
\usepackage{dsfont}
\usepackage{cite}
\usepackage[foot]{amsaddr}
\usepackage{nicefrac}
\usepackage{esint}
\usepackage{mathrsfs}

\newtheorem{theorem}{Theorem}[section]
\newtheorem{lemma}[theorem]{Lemma}
\newtheorem{corollary}[theorem]{Corollary}
\theoremstyle{definition}
\newtheorem{definition}[theorem]{Definition}
\newtheorem{example}[theorem]{Example}
\newtheorem{remark}[theorem]{Remark}
\newtheorem{assumption}[theorem]{Assumption}

\setlist[enumerate]{leftmargin=*, label=(\roman*)}
\numberwithin{equation}{section}

\allowdisplaybreaks

\newcommand{\Newline}{{\rule{0mm}{1mm}\\[-3.25ex]\rule{0mm}{1mm}}}
\newcommand{\Rinf}{[-\infty,\infty)}
\newcommand{\Rbar}{\Rinf}

\def\Bk{{\rm B}_\kappa}

\def\Cb{{\rm C}_{\rm b}}
\def\Cbi{{\rm C}_{\rm b}^\infty}

\def\Cci{{\rm C}_{\rm c}^\infty}
\def\Ck{{\rm C}_\kappa}

\def\Fk{{\rm F}_\kappa}

\def\Uk{{\rm U}_\kappa}

\def\Lipb{{\rm Lip}_{\rm b}}

\def\C{\mathcal{C}}

\DeclareMathOperator*{\Glim}{\Gamma-\lim}
\DeclareMathOperator*{\Gliminf}{\Gamma-\liminf}
\DeclareMathOperator*{\Glimsup}{\Gamma-\limsup}

\DeclareMathOperator{\id}{id}
\DeclareMathOperator{\iter}{iter}

\DeclareMathOperator{\loc}{loc}
\DeclareMathOperator{\supp}{supp}
\DeclareMathOperator{\Span}{span}
\DeclareMathOperator{\sym}{sym}
\DeclareMathOperator{\tr}{tr}

\def\A{\mathcal{A}}
\def\B{\mathcal{B}}
\def\d{{\rm d}}
\def\D{\mathcal{D}}

\def\E{\mathbb{E}}

\def\L{\mathcal{L}}
\def\N{\mathbb{N}}
\def\n{\mathcal{N}}
\def\P{\mathbb{P}}
\def\p{\mathcal{P}}
\def\Q{\mathcal{Q}}
\def\R{\mathbb{R}}
\def\S{\mathbb{S}}
\def\s{\mathscr{S}}
\def\T{\mathcal{T}}
\def\W{\mathcal{W}}
\def\X{\mathcal{X}}

\def\epsilon{\varepsilon}

\def\AG{A_\Gamma}
\def\BG{B_\Gamma}
\def\fe{\overline{f}^\epsilon}

\def\one{\mathds{1}}

\def\AI{\mathcal{A}_I}
\def\AS{\mathcal{A}_S}
\def\AIiter{\mathcal{A}_I^{\iter}}
\def\ASiter{\mathcal{A}_S^{\iter}}
\def\LI{\mathcal{L}_I}
\def\LS{\mathcal{L}_S}
\def\LT{\mathcal{L}_T}
\def\LIplus{\mathcal{L}_I^+}
\def\LSplus{\mathcal{L}_S^+}
\def\LTplus{\mathcal{L}_T^+}
\def\LIsym{\mathcal{L}_I^{\sym}}
\def\LSsym{\mathcal{L}_S^{\sym}}

\def\ca{{\rm ca}_\kappa^+}

\def\Rd{\mathbb{R}^d}
\def\Rq{\mathbb{R}^q}

\def\Aad{\mathcal{A}_{\rm ad}}

\begin{document}

\title[Convex monotone semigroups]
{Convex monotone semigroups and their generators with respect to $\Gamma$-convergence}

\author{Jonas Blessing$^{*,1}$}
\address{{\rm $^*$Department of Mathematics, ETH Zurich, 8092 Zurich,Switzerland}}
\email{$^1$jonas.blessing@math.ethz.ch}

\author{Robert Denk$^{**,2}$}
\email{$^2$robert.denk@uni-konstanz.de}
\address{{\rm $^{**}$Department of Mathematics and Statistics,  University of Konstanz,
 78457 Konstanz, Germany}}
 
\author{Michael Kupper$^{**,3}$}
\email{$^3$ kupper@uni-konstanz.de}

\author{Max Nendel$^{***,4}$}
\address{{\rm $^{***}$Center for Mathematical Economics, Bielefeld University,
 33615 Bielefeld, Germany}}
\email{$^4$max.nendel@uni-bielefeld.de}

\thanks{Financial support through the Deutsche Forschungsgemeinschaft 
(DFG, German Research Foundation) -- SFB 1283/2 2021 – 317210226 is 
gratefully acknowledged.}

\date{\today}

\begin{abstract}
 We study semigroups of convex monotone operators on spaces 
 of continuous functions and their behaviour with respect to 
 $\Gamma$-convergence. In contrast to the linear theory, the domain 
 of the generator is, in general, not invariant under the semigroup. 
 To overcome this issue, we consider different versions of invariant 
 Lipschitz sets which turn out to be  suitable domains for weaker 
 notions of the generator. The so-called $\Gamma$-generator is 
 defined as the time derivative with respect to $\Gamma$-convergence 
 in the space of upper semicontinuous functions. Under suitable 
 assumptions, we show that the $\Gamma$-generator uniquely 
 characterizes the semigroup and is determined by its evaluation
 at smooth functions. Furthermore, we provide Chernoff approximation 
 results for convex monotone semigroups and show that approximation 
 schemes based on the same infinitesimal behaviour lead to the same 
 semigroup. Our results are applied to semigroups related to stochastic 
 optimal control problems in finite and infinite-dimensional settings as 
 well as Wasserstein perturbations of transition semigroups. 
 
\end{abstract}

\vspace*{-1cm}
\maketitle
\vspace*{-1cm}

\thispagestyle{empty}
\tableofcontents
\vspace*{-1cm}

\section{Introduction}

In this article, we address the question whether strongly continuous convex
monotone semigroups are uniquely determined via their infinitesimal generators. 
It is a classical result that a strongly continuous linear semigroup $(S(t))_{t\geq 0}$ on a Banach
space $\X$ satisfies $S(t)\colon D(A)\to D(A)$ for all $t\geq 0$
and that the unique solution of the abstract Cauchy problem $\partial_t u(t)=Au(t)$ 
with $u(0)=x\in D(A)$ is given by $u(t)=S(t)x$ for all $t\geq 0$. Here, the domain $D(A)$ consists of all 
$x\in\X$ such that the limit $Ax:=\lim_{h\downarrow 0}\frac{S(h)x-x}{h}\in\X$
exists. For more details, we refer to Pazy~\cite{Pazy83} and Engel and Nagel~\cite{EN00}. 
These results can be extended to nonlinear semigroups which are generated by m-accretive 
or maximal monotone operators, see Barbu~\cite{Barbu10}, B\'enilan and Crandall~\cite{BC91}, 
Br\'ezis~\cite{Brezis71}, Crandall and Liggett~\cite{CL71} and Kato~\cite{Kato67}. 
While this approach closely resembles the theory of linear semigroups, the definition of 
the nonlinear resolvent typically requires the existence of a unique classical solution of 
a corresponding fully nonlinear elliptic PDE. As pointed out in Evans~\cite{Evans87} and 
Feng and Kurtz~\cite{FK06}, it is, in general, a delicate issue to verify the necessary regularity of classical solutions. This observation was, among others, one of the motivations for the introduction 
of viscosity solutions, see Crandall et al.~\cite{CIL92}, Crandall and Lions~\cite{CL83}
and Lions~\cite{Lions82}. The key ideas in order to obtain uniqueness of viscosity solutions 
are local comparisons with a sufficiently large class of smooth test functions and regularizations 
by introducing additional viscosity terms.\ For the relation between viscosity solutions and semigroups, we refer to Alvarez et al.~\cite{AGL93} and Biton~\cite{Biton01}, where the authors provide axiomatic foundations of viscosity solutions to fully nonlinear second-order PDEs based on suitable regularity and locality assumptions for monotone semigroups on spaces of continuous functions. While these works mainly focus on the existence and axiomatization 
of second-order differential operators through semigroups, the uniqueness of the associated 
semigroups in terms of their generator is not yet fully clarified, cf.\ the discussion 
in~\cite[Section 5]{Biton01}. We refer to Fleming and Soner~\cite[Chapter~II.3]{FS06} 
for a broad discussion on the relation between semigroups and viscosity solutions and 
to Yong and Zhou \cite[Chapter~4]{YZ99} for an illustration of the interplay between 
the dynamic programming principle and viscosity solutions in a stochastic optimal 
control setting. 

In the spirit of traditional semigroup theory, this article is concerned with comparison principles for strongly continuous convex monotone semigroups resembling the classical analogue from the linear case.\ 
In order to uniquely characterize semigroups via their infinitesimal generator, 
it is crucial to show that the domain is invariant under the semigroup but 
for nonlinear semigroups this statement might be wrong, see, e.g., Crandall and 
Liggett~\cite[Section~4]{CL71} and Denk et.\ al.~\cite[Example~5.2]{DKN21+}.\ 
On the one hand, for strongly continuous convex monotone semigroups defined 
on spaces with order continuous norm such as $L^p$-spaces and Orlicz hearts, 
the domain is invariant and the unique solution of the of 
the abstract Cauchy problem \[\partial_t u(t)=Au(t)\quad\text{with}\quad u(0)=x\in D(A)\]
is given by $u(t)=S(t)x$ for all $t\geq 0$, see Denk et. al.~\cite{DKN21}. However, 
this approach uses rather restrictive conditions on the nonlinear generator, 
see Blessing and Kupper~\cite{BK22}. On the other hand, 
for spaces of continuous functions, the domain is typically not invariant
under the semigroup and we have to find an invariant set on which we can define
a weaker notion of the generator that uniquely determines the semigroup. 
In Blessing and Kupper, see~\cite{BK23, BK22}, invariant Lipschitz sets are used 
in order to construct nonlinear semigroups and invariant symmetric Lipschitz sets 
provide regularity results in Sobolev spaces.\ In this article, we introduce the invariant 
upper Lipschitz set $\LSplus$ on which we can define the upper $\Gamma$-generator by
\[ \AG^+ f:=\Glimsup_{h\downarrow 0}\frac{S(h)f-f}{h} \quad\mbox{for all } f\in\LSplus. \]
Since $\AG^+ f$ is, in general, only upper semicontinuous, we have to extend $S(t)$ from continuous 
to upper semicontinuous functions in order to define the term $S(t)\AG^+ f$, see Subsection~\ref{sec:convex}.
The equivalence between continuity from above, continuity w.r.t. the mixed 
topology and upper semi-continuity w.r.t. $\Gamma$-convergence for families of uniformly 
bounded convex monotone operators is also the key to understand why 
$\Gamma$-convergence is a suitable choice for the definition of the generator. 

The first main result of this article is a comparison principle, which allows to understand 
strongly continuous convex monotone semigroups as minimal $\Gamma$-supersolutions
to the abstract Cauchy problem $\partial_t u(t)=\AG^+ u(t)$ with
$u(0)=f\in\LSplus$, see Theorem~\ref{thm:comp}. Moreover, strongly 
continuous convex monotone semigroups are uniquely determined by their upper 
$\Gamma$-generators defined on their upper Lipschitz sets, see Theorem~\ref{thm:comp2}.
Under additional assumptions we further show that $\AG f=\Glim_{n\to\infty}\AG f_n$ 
for suitable approximating sequences $f_n\to f$ such that $(\AG^+ f_n)_{n\in\N}$ is bounded above, 
see Theorem~\ref{thm:lb}. Using regularization by convolution and truncation, we then 
establish comparison principles for convex monotone semigroups by comparing their generators
only on smooth test functions or even smooth test functions with compact support,
see Theorem~\ref{thm:conv3} and Theorem~\ref{thm:trunc}.
This leads to an explicit description of the upper $\Gamma$-generator if, for smooth functions, the 
$\Gamma$-generator is given as a convex functional of certain partial derivatives,
see Theorem~\ref{thm:dist}. It was shown in Alvarez et al.~\cite{AGL93} and Biton~\cite{Biton01} 
that this is the case for typical fully nonlinear second-order PDEs. 
We would also like to remark that our notion of a strongly continuous convex monotone semigroup 
coincides with the one in Goldys et al.\ \cite{GNR22}, where it is shown that the function 
$u(t):=S(t)f$ is a viscosity solution of the abstract Cauchy problem $\partial_t u=Au$ 
with initial condition $u(0)=f$. 
The idea of weakening topological properties of the semigroup is already present in the literature.\ Goldys and 
Kocan~\cite{GK01}, van Casteren~\cite{vanCasteren11}, Kunze~\cite{Kunze09} and 
Kraaij~\cite{Kraaij19b} study linear semigroups in strict topologies, see also 
Kraaij~\cite{Kraaij16} and Yosida~\cite{Yoshida80} for semigroups in locally convex spaces.
We further discuss this in Remark~\ref{rem:topsemigroup}.
In addition, equicontinuity in the strict topology is suitable for stability results. Kraaij~\cite{Kraaij22} 
provides convergence results for nonlinear semigroups based on the link between viscosity 
solutions to HJB equations and pseudo-resolvents and in~\cite{Kraaij19} the author establishes
$\Gamma$-convergence of functionals on path-spaces. 

Second, we study Chernoff-type approximation schemes for strongly continuous convex monotone semigroups
which resemble Chernoff's original work, see~\cite{Chernoff68, Chernoff74}, that 
generalizes the Trotter--Kato product formula for linear semigroups, see~\cite{Trotter59,Kato78}.
Starting with a family $(I(t))_{t\geq 0}$ of operators $I(t)\colon\Ck\to\Ck$, we construct a 
corresponding semigroup as the limit 
\begin{equation} \label{eq:intro}
 S(t)f:=\lim_{n\to\infty}I(h_n)^{k_n^t}f\in\Ck,
\end{equation}
where $(h_n)_{n\in\N}\subset (0,\infty)$ converges to zero and $k_n^t:=\max\{k\in\N_0\colon kh_n\leq t\}$.  
Denoting by $A$ the generator of $(S(t))_{t\geq 0}$, it holds
\[ Af=I'(0)f:=\lim_{h\downarrow 0}\frac{I(h)f-f}{h}\in\Ck \]
for all $f\in\Ck$ such that the previous limit exists, see Theorem~\ref{thm:I}.
For instance, in case of the Trotter--Kato formula, the choice $I(t):=e^{tA_1}e^{tA_2}$ 
yields a semigroup with generator $A=A_1+A_2$. Since all limits in the present framework
are taken w.r.t. the mixed topology, this work extends previous results that are either 
based on norm convergence, see Blessing and Kupper~\cite{BK23}, or monotone convergence, 
see~\cite{Nisio76, NR21, DKN20,BEK21, FKN23}. A priori, the limit in equation~\eqref{eq:intro}
only exists for a subsequence and the derivative $I'(0)f$ can only be computed for smooth 
functions. Hence, in order to apply the approximation results from Section~\ref{sec:approx}, 
we introduce the so-called approximation set, see Theorem~\ref{thm:approx}. We further 
identify explicit conditions such that the semigroup $(S(t))_{t\geq 0}$ is uniquely determined
by $I'(0)f$, see Theorem~\ref{thm:I3}. In particular, the limit in equation~\eqref{eq:intro} 
does not depend on the choice the sequence $(h_n)_{n\in\N}$ and different approximation 
schemes with the same infinitesimal behaviour lead to the same semigroup. 

In Section~\ref{sec:examples}, we explain how the abstract results can be applied
in different examples.\ In Subsection~\ref{sec:control}, we show that stochastic optimal control 
problems can be approximated by using piecewise constant controls and a discrete noise. Moreover, 
an explicit computation of the so-called symmetric Lipschitz set yields a regularity result 
for the corresponding fully nonlinear PDE. In Subsection~\ref{sec:infdim}, we extend the 
previous approximation result to an infinite-dimensional setting.\ In Subsection~\ref{sec:wasserstein}, 
we show that non-parametric Wasserstein perturbations of transition semigroups asymptotically 
coincide with perturbations which have a finite-dimensional parameter space.\ As a byproduct, 
we recover the Talagrand $T_2$ inequality for the normal distribution.

\section{Setup and notation}

Throughout, let $(X,d)$ be a complete separable metric space and $\kappa\colon X\to (0,\infty)$
be a bounded continuous function.\ Functions $f,g\colon X\to\Rbar$ are ordered pointwise and 
we define $f\vee g:=\max\{f,g\}$, $f\wedge g:=\min\{f,g\}$, $f^+:=f\vee 0$ and $f^-:=-(f\wedge 0)$ 
with the convention $-(-\infty):=\infty$. Defining 
\[ \|f\|_\kappa:=\sup_{x\in X}|f(x)|\kappa(x)\in [0,\infty]
	\quad\mbox{for all } f\colon X\to\Rbar, \]
we denote by $\Ck$ the space of all continuous functions $f\colon X\to\R$ with $\|f\|_\kappa<\infty$, 
by $\Uk$ the set of all upper semicontinuous functions $f\colon X\to\Rinf$ with $\|f^+\|_\kappa<\infty$ 
and by $\Bk$ the space of all Borel measurable functions $f\colon X\to\Rbar$ with $\|f^+\|_\kappa<\infty$.
The space $\Ck$ is endowed with the mixed topology between $\|\cdot\|_\kappa$ and the topology 
of uniform convergence on compacts sets, i.e., the strongest locally convex topology on $\Ck$ 
that coincides on $\|\cdot\|_\kappa$-bounded sets with the topology of uniform convergence on 
compact subsets. It is well-known, see~\cite[Proposition~B.2]{GNR22}, that $f_n\to f$ if and only if 
\[ \sup_{n\in\N}\|f_n\|_\kappa<\infty \quad\mbox{and}\quad
\lim_{n\to\infty}\|f-f_n\|_{\infty,K}=0 \]
for all compact subsets $K\subset X$, where $\|f\|_{\infty,K}:=\sup_{x\in K}|f(x)|$. 
Subsequently, if not stated otherwise, all limits in $\Ck$ are understood w.r.t.\ the mixed 
topology and we write $K\Subset X$ if $K$ is a compact subset of $X$. An operator 
$\Phi\colon \Ck\to \Ck$ is called monotone if $\Phi f\leq \Phi g$ for all $f,g\in\Ck$ with 
$f\leq g$ and convex if $\Phi(\lambda f+(1-\lambda)g)\leq\lambda\Phi f+(1-\lambda)\Phi g$
for all $f,g\in\Ck$ and $\lambda\in [0,1]$.\ Although the mixed topology is not metrizable, 
for monotone operators $\Phi \colon \Ck\to \Ck$, sequential continuity is still equivalent 
to continuity and continuity is equivalent to continuity on $\|\cdot\|_\kappa$-bounded subsets, 
see~\cite{Nendel22}.\ The mixed topology also belongs to the class of strict topologies, 
see~\cite{FGH72,Kunze09,Sentilles72}. For more details, we refer to~\cite{Wheeler83, Wiweger61} and Appendix~B in \cite{GNR22}.
Note that the set $\Uk$ consists of all functions $f\colon X\to\Rbar$ such that there exists 
a sequence $(f_n)_{n\in\N}\subset\Ck$ that decreases pointwise to $f$. Indeed,
\[ f(x)=\lim_{n\to\infty}\sup_{y\in X}\frac{1}{\kappa(x)}\big(\max\{(f\kappa)(y),-n\}-n^2 d(x,y)\big), \]
showing that $\Uk=(\Ck)_\delta:=\{\inf_{n\in\N}f_n\colon (f_n)_{n\in \N}\subset\Ck\}$. 
Subsequently, we write $f_n\downarrow f$ if $(f_n)_{n\in\N}\subset\Uk$ decreases pointwise 
to $f\in\Uk$.\ A set $F\subset\Bk$ is called bounded if $\sup_{f\in F}\|f\|_\kappa<\infty$ 
and bounded above if $\sup_{f\in F}\|f^+\|_\kappa<\infty$. For every $r\ge 0$, let
\[ B_{\Ck}(r):=\{f\in \Ck\colon\|f\|_\kappa\leq r\}. \]
Likewise, we define $B_{\Uk}(r):=\{f\in\Uk\colon\|f\|_\kappa\leq r\}$ and 
$B_{\Bk}(r):=\{f\in\Bk\colon\|f\|_\kappa\leq r\}$.

\begin{definition}
 For every sequence $(f_n)_{n\in\N}\subset \Uk$, which is bounded above, we define
 \[ \Big(\Glimsup_{n\to\infty}f_n\Big)(x)
 	:=\sup\Big\{\limsup_{n\to\infty}f_n(x_n)\colon
 	(x_n)_{n\in\N}\subset X\mbox{ with } x_n\to x\Big\}\in\Rbar \] 
 for all $x\in X$.  Moreover, we say that $f=\Glim_{n\to\infty}f_n$ with $f\in\Uk$ if, for every $x\in X$,
 \begin{itemize}
  \item $f(x)\geq\limsup_{n\to\infty}f_n(x_n)$ for every sequence $(x_n)_{n\in\N}\subset X$ with $x_n\to x$,
  \item $f(x)=\lim_{n\to\infty}f_n(x_n)$ for some sequence $(x_n)_{n\in\N}\subset X$ with $x_n\to x$.
 \end{itemize}
 For every $t\geq 0$ and $(f_s)_{s\geq 0}\subset\Uk$ being bounded above, we define 
 \[ \Glimsup_{s\to t}f_s
 	:=\sup\Big\{\Glimsup_{n\to\infty}f_{s_n}\colon 0\leq s_n\to t\Big\}\in\Uk. \]
 In addition, we write $f=\Glim_{s\to t}f_s$ with $f\in\Uk$ if $f=\Glim_{n\to\infty}f_{s_n}$ 
 for all sequences $(s_n)_{n\in\N}\subset [0,\infty)$ with $s_n\to t$. 
\end{definition}

For further details on $\Gamma$-convergence, we refer to Appendix~\ref{app:gamma}.
The different convergence concepts in this article are related as follows. For every sequence 
$(f_n)_{n\in\N}\subset\Ck$ and $f\in\Ck$, $\|f_n-f\|_\kappa\to 0$ implies $f_n\to f$, 
$f_n\to f$ implies $f=\Glim_{n\to\infty}f_n$ and Dini's theorem guarantees that $f_n\downarrow f$ 
implies $f_n\to f$. In addition, if the sequence $(f_n)_{n\in\N}$ is bounded, then the following 
statements are equivalent:
\begin{enumerate}
 \item $f_n\to f$,
 \item $f_n\to f$ uniformly on compacts,
 \item $f_n(x_n)\to f(x)$ for all $x\in X$ and sequences $(x_n)_{n\in \N}\subset X$ with $x_n\to x$,
 \item $f=\Glimsup_{n\to\infty}f_n=-(\Glimsup_{n\to\infty} (-f_n))$.
\end{enumerate}
The equivalence between~(i) and~(iv) follows from Lemma~\ref{lem:fe}(ii) and Dini's theorem.
Indeed, since $f\in\Ck$, it holds $\Glimsup_{n\to\infty}f_n\leq f$ if and only if $\|(f_n-f)^+\|_{\infty,K}\to 0$ 
for all $K\Subset X$ and $\Glimsup_{n\to\infty}(-f_n)\leq -f$ if and only if $\|(f-f_n)^+\|_{\infty,K}\to 0$ 
for all $K\Subset X$. While~(iii) is equivalent to $f=\Glim_{n\to\infty}f_n=-\Glim_{n\to\infty}(-f_n)$,
in general, only the inequality $\Glim_{n\to\infty}f_n\geq-\Glim_{n\to\infty}(-f_n)$ is valid.

\subsection{Lipschitz sets, $\Gamma$-generators and basic continuity properties}

The (infinitesimal) generator of a family $(I(t))_{t\geq 0}$ of operators $I(t)\colon\Ck\to\Ck$ 
is defined by
\[ A\colon D(A)\to\Ck,\; f\mapsto\lim_{h\downarrow 0}\frac{I(h)f-f}{h}, \]
where the domain $D(A)$ consists of all $f\in\Ck$ such that the previous limit exists w.r.t.\ 
the mixed topology.\ While the domain of any strongly continuous linear semigroup is invariant, 
the same does not necessarily hold for strongly continuous convex monotone semigroups. We therefore
introduce the upper Lipschitz set which is invariant and will serve as domain for the upper 
$\Gamma$-generator.

\begin{definition} \label{def:Lset}
 Let $(I(t))_{t\geq 0}$ be a family of operators $I(t)\colon\Ck\to\Ck$. The upper Lipschitz set $\LIplus$ 
 consists of all $f\in\Ck$ such that
 \[ \limsup_{h\downarrow 0}\left\|\frac{(I(h)f-f)^+}{h}\right\|_\kappa<\infty \]
 which is equivalent to the existence of $c\geq 0$ and $h_0>0$ with $\|(I(h)f-f)^+\|_\kappa\leq ch$
 for all $h\in [0,h_0]$. Furthermore, the Lipschitz set $\LI$ consists of all $f\in\Ck$ such that
 \[ \limsup_{h\downarrow 0}\left\|\frac{I(h)f-f}{h}\right\|_\kappa<\infty \]
 which is equivalent to the existence of $c\geq 0$ and $h_0>0$ with $\|I(h)f-f\|_\kappa\leq ch$
 for all $h\in [0,h_0]$.\ The symmetric Lipschitz set is defined by $\LIsym:=\{f\in\LI\colon -f\in\LI\}$.
\end{definition}

From the convexity estimates in Appendix~\ref{sec:convestim}, we obtain the following invariance result.

\begin{lemma} \label{lem:inv}
 Let $(I(t))_{t\geq 0}$ be a family of convex monotone operators $I(t)\colon\Ck\to\Ck$ with
 $I(s)I(t)f=I(t)I(s)f$ for all $f\in\Ck$ and $s,t\geq 0$. Then, 
 \[ I(t)\colon\LI\to\LI \quad\mbox{and}\quad I(t)\colon\LIplus\to\LIplus \quad\mbox{for all } t\geq 0. \]
\end{lemma} 
\begin{proof} 
 For every $t\geq 0$, $f\in\Ck$ and $h\in (0,1]$, it follows from $I(h)I(t)f=I(t)I(h)f$, 
 Lemma~\ref{lem:lambda} and the monotonicity of $I(t)$ that
 \begin{align*}
  I(t)I(h)f-I(t)\left(\frac{f-I(h)f}{h}+I(h)f\right)
  &\leq \frac{I(t)I(h)f-I(t)f}{h}=\frac{I(h)I(t)f-I(t)f}{h} \\
  &\leq I(t)\left(\frac{I(h)f-f}{h}+f\right)-I(t)f \\
  &\leq I(t)\left(\frac{(I(h)f-f)^+}{h}+f\right)-I(t)f. 
 \end{align*}
 If $f\in\LI$, we can choose $c\geq 0$ and $h_0\in (0,1]$ with $\|I(h)f-f\|_\kappa\leq ch$ 
 for all $h\in [0,h_0]$. Due to the previous estimate and Lemma~\ref{lem:lip}(ii), there 
 exists $c'\geq 0$ with 
 \[ \left\|\frac{I(t)I(h)f-I(t)f}{h}\right\|_\kappa\leq cc' \quad\mbox{for all } h\in (0,h_0]. \]
 If $f\in\LIplus$, one can argue similarly. 
\end{proof}

The invariance of $\LIsym$ can not always be guaranteed. However, in several examples, the 
symmetric Lipschitz set is invariant and can be determined explicitly which implies that certain 
regularity properties of the initial function $f$ are transferred to $I(t)f$ for all $t\geq 0$, 
see~\cite{BK22, BK23} and Subsection~\ref{sec:control}.

\begin{remark}\label{rem.favard}
 We briefly discuss the relation between Lipschitz sets and Favard spaces. Let $(T(t))_{t\geq 0}$ 
 be a be a strongly continuous semigroup of bounded linear operators on a Banach space $\X$. 
 For simplicity, we assume that there exist $c\geq 0$ and $\omega<0$ with $\|T(t)x\|\leq c e^{\omega t}\|x\|$ 
 for all $x\in\X$. Then, the set
 \[ F_1:=\Big\{x\in\mathcal{X}\colon\sup_{h>0}\Big\|\frac{T(h)x-x}h\Big\|<\infty\Big\} \]
 is called Favard space or saturation class, see~\cite[Section~2.1]{BN71} and~\cite[Section~II.5.b]{EN00}. 
 Denoting by $B$ the norm generator of $(T(t))_{t\geq 0}$, it is known that $F_1=D(B)$ holds 
 if $\X$ is reflexive, see~\cite[Theorem~2.1.2]{BN71}. If $(T(t))_{t\geq 0}$ is holomorphic, 
 then $F_1=(\X,D(B))_{1,\infty}$, see~\cite[Proposition~2.2.2]{Lunardi95}, where
 $(\cdot,\cdot)_{1,\infty}$ stands for the real interpolation functor. For non-reflexive 
 $\mathcal{X}$, an explicit description of $F_1$ seems to be unknown in many cases.
\end{remark}

On the upper Lipschitz set, we now define the so-called upper $\Gamma$-generator 
which is usually not a continuous function anymore but still upper semicontinuous. 
In Section~\ref{sec:approx}, we give conditions under which the upper $\Gamma$-generator 
coincides with the $\Gamma$-generator.

\begin{definition} \label{def:GammaGen}
Let $(I(t))_{t\geq 0}$ be a family of operators $I(t)\colon\Ck\to\Ck$.\ The upper $\Gamma$-generator 
is defined by
 \[ \AG^+\colon\LIplus\to\Uk,\; f\mapsto\Glimsup_{h\downarrow 0}\frac{I(h)f-f}{h}. \]
 Furthermore, the $\Gamma$-generator is defined by
 \[ \AG\colon D(\AG)\to\Uk,\; f\mapsto\Glim_{h\downarrow 0}\frac{I(h)f-f}{h}, \]
 where the domain $D(\AG)$ consists of all $f\in\LIplus$ such that the previous limit exists.
\end{definition}

We conclude this section with the discussion of basic continuity properties for families 
of convex monotone operators.\ The results transfer to single operators by considering 
families that are constant in time, see also~\cite{Nendel22} for a more detailed discussion.\ 
Note that, by Lemma~\ref{lem:lip}, for a family $(I(t))_{t\geq 0}$ of convex monotone 
operators $I(t)\colon\Ck\to\Ck$, the following statements are equivalent:
\begin{enumerate}
 \item $\sup_{t\in [0,T]}\|I(t)\tfrac{r}{\kappa}\|_\kappa<\infty$ for all $r,T\geq0$.
 \item $\sup_{t\in [0,T]}\sup_{f\in B_{\Ck}(r)}\|I(t)f\|_\kappa<\infty$ for all $r,T\geq0$.
 \item For every $r,T\geq 0$, there exists $c\ge 0$ with 
  \[ \|I(t)f-I(t)g\|_\kappa\le c\|f-g\|_\kappa \quad\mbox{for all } t\in[0,T] \mbox{ and } f,g\in B_{\Ck}(r). \]
\end{enumerate}
Moreover, $I(t)f_n\downarrow 0$ for all $(f_n)_{n\in \N}\subset \Ck$ with $f_n\downarrow 0$ 
implies $I(t)0=0$ for all $t\geq 0$.

\begin{lemma}  \label{lem:uniform}
Let $(I(t))_{t\geq 0}$ be a family of convex monotone operators $I(t)\colon\Ck\to\Ck$ 
satisfying $I(t)0=0$, $\sup_{s\in [0,T]}\|I(s)\frac{r}{\kappa}\|_\kappa<\infty$ and
$\Glimsup_{s\to t}I(s)f\leq I(t)f$ for all $r,t,T\geq 0$ and $f\in\Ck$. 
Then, the following statements are equivalent:
\begin{enumerate}
\item It holds $I(t)f_n\downarrow 0$ for all $t\geq 0$ and $(f_n)_{n\in\N}\subset \Ck$ with $f_n\downarrow 0$.
\item For every $T\geq 0$, $K\Subset X$ and $(f_n)_{n\in\N}\subset \Ck$ with $f_n\downarrow 0$, it holds
 \[ \sup_{(t,x)\in [0,T]\times K}(I(t)f_n)(x)\downarrow 0 \quad\mbox{as } n\to\infty. \]
\item For every $\epsilon>0$, $r,T\geq 0$ and $K\Subset X$, there exist $K'\Subset X$ and $c\geq 0$ with
 \[ \|I(t)f-I(t)g\|_{\infty, K}\leq c\|f-g\|_{\infty, K'}+\epsilon \]
 for all $t\in [0,T]$ and $f,g\in B_{\Ck}(r)$. 	
\end{enumerate}
\end{lemma}
\begin{proof}
 Since the mapping $[0,\infty)\times X\to\R,\; (t,x)\mapsto (I(t)f)(x)$ is upper semicontinuous
 for all $f\in\Ck$, Dini's theorem implies that~(i) and~(ii) are equivalent. Let $T\geq 0$ and 
 $K\Subset X$. For every $(t,x)\in [0,T]\times K$, we define 
 \[ \phi_{t,x}\colon\Ck\to\R,\; f\mapsto (I(t)f)(x). \]
 Then, the equivalence between~(ii) and~(iii) follows from Lemma~\ref{lem:cont.app}.
\end{proof}

\begin{remark}\label{rem:equicont}
 A family $(I(t))_{t\geq 0}$ satisfying the assumptions and the equivalent conditions (i)-(iii) 
 of Lemma~\ref{lem:uniform} is locally uniformly equicontinuous on bounded sets in the mixed topology. 
 Indeed, following Appendix~B in~\cite{GNR22}, the mixed topology is generated by the seminorms 
 $p_{\kappa,(K_n),(a_n)}(f):=\sup_{n\in\N}\sup_{x\in K_n} a_n|f(x)|\kappa(x)$, where $K_n\Subset X$ 
 and $(a_n)_{n\in\N}\subset (0,\infty)$ satisfies $a_n\to 0$. By considering the functionals 
 $\tilde{\phi}_{t,x}\colon\Ck\to\R$ given by $\tilde{\phi}_{t,x}(f):=(I(t)f)(x)\kappa(x)$, we obtain 
 that the conditions (i)-(iii) are equivalent to the following statement: for every $\epsilon>0$ and 
 $r,T\geq 0$, there exists $\tilde{c}\geq 0$ such that, for every $K\Subset X$, there exists $K'\Subset X$ with
 \[ \|I(t)f-I(t)g\|_{\kappa, K}\leq \tilde c\|f-g\|_{\kappa, K'}+\epsilon \]
 for all $t\in [0,T]$ and $f,g\in B_{\Ck}(r)$, where $\|f\|_{\kappa,K}:=\sup_{x\in K}|f(x)|\kappa(x)$. 
 Since $\tilde{c}$ does not depend on $K$, for every $\epsilon>0$, $r,T\geq 0$ and seminorm $p_{\kappa,(K_n),(a_n)}$, 
 there exist $\delta>0$ and another seminorm $p_{\kappa,(C_n),(b_n)}$ such that
 \[ p_{\kappa,(K_n),(a_n)}(I(t)f-I(t)g)\leq\epsilon \]
 for all $t\in [0,T]$ and $f,g\in B_{\Ck}(r)$ with $p_{\kappa,(C_n),(b_n)}(f-g)\leq\delta$.
\end{remark}

\section{Convex monotone semigroups, comparison and uniqueness}\label{sec:convex.comparison}

\subsection{Convex monotone semigroups}
\label{sec:convex}

The following definition characterizes the main object of this article.\ We emphasize that
requiring norm continuity of the mappings $[0,\infty)\to\Ck$, $t\mapsto S(t)f$ for all 
$f\in\Ck$ would be too strong even in the linear case. Hence, we define strong continuity 
w.r.t. the mixed topology and, for most of the results, we only require the weaker conditions~\ref{S3} 
and~\ref{S4}.

\begin{definition} \label{def:semigroup}
 A family $(S(t))_{t\geq 0}$ of operators $S(t)\colon\Ck\to\Ck$ is called convex monotone 
 semigroup on $\Ck$ if the following conditions are satisfied:
 \begin{enumerate}[label=(S\arabic*)]
  \item \label{S1} $S(t)$ is convex and monotone with $S(t)f_n\downarrow 0$ for all $t\geq 0$ and
   $f_n\downarrow 0$.
  \item \label{S2} $S(0)f=f$ and $S(s+t)f=S(s)S(t)f$ for all $s,t\geq 0$ and $f\in \Ck$.
  \item \label{S3} $\sup_{t\in [0,T]}\|S(t)\frac{r}{\kappa}\|_\kappa<\infty$
   for all $r,T\geq 0$.
  \item \label{S4} $\Glimsup_{s\to t}S(s)f\leq S(t)f$ for all 
   $t\geq 0$ and $f\in\Ck$.
 \end{enumerate}
 Furthermore, a convex monotone semigroup $(S(t))_{t\geq 0}$ on $\Ck$ is strongly continuous 
 if $f=\lim_{t\to 0}S(t)f$ for all $f\in \Ck$.
\end{definition}

The following lemma shows that strong continuity transfers to all positive times.

\begin{lemma} \label{lem:time.cont}
Let $(S(t))_{t\geq0}$ be a family of operators $S(t)\colon \Ck\to \Ck$ that satisfies \ref{S1} and \ref{S2}.\ Then, the following statements are equivalent:
\begin{enumerate}
  \item $(S(t))_{t\geq0}$ is a stongly continuous convex monotone semigroup on $\Ck$. 
  \item $S(t)f=\lim_{s\to t}S(s)f$ for all $t\geq 0$ and $f\in\Ck$.
  \item $S(t_n)f_n\to S(t)f$ for all sequences $(t_n)_{n\in\N}\subset [0,\infty)$ and 
   $(f_n)_{n\in\N}\subset\Ck$ such that $t:=\lim_{n\to\infty}t_n$ and $f:=\lim_{n\to\infty}f_n$ 
   exist. 
 \end{enumerate}
\end{lemma}
\begin{proof}
 First, we show that~(i) implies~(ii). Let $t\geq 0$ and $f\in\Ck$. Condition~\ref{S2} implies 
 $S(s)f-S(t)f=S(s)f-S(s)S(t-s)f$ for all $s\in [0,t]$. Moreover, for every $\epsilon>0$ and 
 $K\Subset X$, Lemma~\ref{lem:uniform} guarantees the existence of $c\geq 0$ and $K'\Subset X$ with  
 \[ \|S(s)f-S(s)S(t-s)f\|_{\infty,K}\leq c\|S(t-s)f-f\|_{\infty,K'}+\epsilon \quad\mbox{for all } s\in [0,t]. \]
 It follows from~(i) that $\|S(t-s)f-f\|_{\infty,K'}\to 0$ showing that $S(s)f\to S(t)f$ as $s\to t$. 
 For $s>t$, one can argue similarly.
 
 Second, we show that~(ii) implies~(i).\ It is clear that~(ii) implies $\lim_{t\downarrow 0} S(t)f=f$ for all $f\in \Ck$. It remains to prove that~(ii) implies \ref{S3} and \ref{S4}. To that end, let $r,T\geq0$ and assume towards a contradiction that there existed a sequence $(t_n)_{n\in \N}\subset [0,T]$ with $\|S(t_n)\frac{r}{\kappa}\big\|_\kappa\geq n$ for all $n\in \N$. Since $[0,T]$ is compact, by potentially passing to a subsequence, we may w.l.o.g.\ assume that $t_n\to t\in [0,T]$, so that $S(t_n)\frac{r}{\kappa}\to S(t)\frac{r}{\kappa}$. The latter implies that $\sup_{n\in \N} \|S(t_n)\frac{r}{\kappa}\big\|_\kappa<\infty$, which leads to the desired contradiction.\ We have therefore shown that \ref{S3} is satisfied and \ref{S4} now follows as a direct consequence of~(ii). 
 
 Third, it is clear that~(iii) implies~(ii).\ Since the equivalence between~(i) and~(ii) is already established, it remains to show that~(i) together with~(ii) implies~(iii). Let $(t_n)_{n\in\N}\subset [0,\infty)$ and 
 $(f_n)_{n\in\N}\subset\Ck$ such that $t:=\lim_{n\to\infty}t_n$ and $f:=\lim_{n\to\infty}f_n$ 
 exist. Let $\epsilon>0$ and $K\Subset X$. By Lemma~\ref{lem:uniform}, there exist $c\geq 0$ 
 and $K'\Subset X$ with 
 \begin{align*} 
  \|S(t_n)f_n-S(t)f\|_{\infty,K}
  &\leq\|S(t_n)f_n-S(t_n)f\|_{\infty,K}+\|S(t_n)f-S(t)f\|_{\infty,K} \\
  &\leq c\|f-f_n\|_{\infty,K'}+\epsilon+\|S(t_n)f-S(t)f\|_{\infty,K}
  \quad\mbox{for all } n\in\N.
 \end{align*}
 It follows from~(ii) that $\|S(t_n)f-S(t)f\|_{\infty,K}\to 0$ showing that $S(t_n)f_n\to S(t)f$.
\end{proof}

Let $(S(t))_{t\geq 0}$ by a convex monotone semigroup on $\Ck$ with generator $A$ and 
upper $\Gamma$-generator $\AG^+$.\ While $Af\in\Ck$ holds by definition, the function 
$\AG^+ f$ is not necessarily continuous and may take the value $-\infty$.\ Hence, in 
order to define the term $S(t)\AG^+ f$, an extension of $S(t)$ to $\Uk$ is necessary.\ 
Moreover, the argumentation in this article relies heavily on the fact that the extension 
is $\Gamma$-upper semicontinuous in time and continuous from above.

\begin{theorem} \label{thm:ext}
 Let $(S(t))_{t\geq 0}$ be a convex monotone semigroup on $\Ck$.\ Then, for every $t\geq 0$, 
 there exists a unique extension $S(t)\colon\Uk\to\Uk$ which is continuous from above\footnote{
 It holds $S(t)f_n\downarrow S(t)f$ for all $(f_n)_{n\in\N}\subset\Uk$ and $f\in\Uk$ with $f_n\downarrow f$.}.
 The family $(S(t))_{t\geq 0}$ of extended operators is a convex monotone semigroup on $\Uk$, 
 i.e., it satisfies the conditions~\ref{S1}-\ref{S4} with $\Uk$ instead of $\Ck$. In addition,
 for every $t\ge 0$ and $x\in X$, the functional $\Uk\to\Rbar,\; f\mapsto (S(t)f)(x)$ can be 
 further extended to the space $\Bk$ such that, for every $r,T\geq 0$, $\epsilon>0$ and $K\Subset X$, 
 there exists $K_1\Subset X$ with
 \[ \sup_{(t,x)\in [0,T]\times K}\big(S(t)\big(\tfrac{r}{\kappa}\one_{K_1^c}\big)\big)(x)\leq\epsilon. \]
\end{theorem}
\begin{proof}
 First, we extend the family $(S(t))_{t\geq 0}$ from $\Ck$ to $\Uk$.\ Let $t\geq 0$ and $x\in X$.\ 
 By condition~\ref{S1}, the functional
 \[ \phi_{t,x}\colon\Ck\to\R,\; f\mapsto (S(t)f)(x) \] 
 satisfies the assumptions of Theorem~\ref{thm:dual}.\ Define
 $(S(t)f)(x):=\overline{\phi_{t,x}}(f)$ for all $f\in\Uk$, where $\overline{\phi_{t,x}}$ denotes 
 the unique extension of $\phi_{t,x}$ from Theorem~\ref{thm:dual}(ii).\ The family of extended 
 operators satisfies the conditions~\ref{S1},~\ref{S2} and~\ref{S3} with $\Uk$ instead of $\Ck$ 
 and is the unique extension which is continuous from above.\ These properties follow from the fact
 that, for every $f\in\Uk$, there exists a sequence $(f_n)_{n\in\N}\subset\Ck$ with $f_n\downarrow f$.
 In order to verify condition~\ref{S4}, let $t\geq 0$, $f\in\Uk$ and $x\in X$. Let 
 $(t_n)_{n\in\N}\subset [0,\infty)$ and $(x_n)_{n\in\N}\subset X$ be sequences with $t_n\to t$ and 
 $x_n\to x$. We use Theorem~\ref{thm:dual}(ii) and condition~\ref{S4} on $\Ck$ to estimate
 \begin{align*}
  \limsup_{n\to\infty}\,(S(t_n)f)(x_n)
  &=\inf_{n\in\N}\sup_{k\geq n}\inf\big\{\big(S(t_k)g\big)(x_k)\colon g\in\Ck,\, g\geq f\big\} \\
  &\leq\inf\Big\{\limsup_{n\to\infty}\big(S(t_n)g\big)(x_n)\colon g\in\Ck,\, g\geq f\Big\} \\
  &\leq\inf\big\{\big(S(t)g\big)(x)\colon g\in\Ck,\, g\geq f\big\}=(S(t)f)(x).
 \end{align*}
 
 Second, we define $(S(t)f)(x):=\widehat{\phi_{t,x}}(f)$ for all $t\geq 0$, $f\in\Bk$ and
 $x\in X$, where $\widehat{\phi_{t,x}}$ denotes the extension from Theorem~\ref{thm:dual}(iii). 
 Let $r,T\geq 0$, $\epsilon>0$ and $K\Subset X$. It follows from Theorem~\ref{thm:dual}(iii) that
 \[ 0\leq\widehat{\phi_{t,x}}\big(\tfrac{r}{\kappa}\one_{K_1^c}\big)
 	\leq\sup_{\mu\in M}\mu\big(\tfrac{r}{\kappa}\one_{K_1^c}\big) \]
 for all $(t,x)\in [0,T]\times K$ and $K_1\Subset X$, where 
 \[ M:=\bigcup_{(t,x)\in [0,T]\times K}\Big\{\mu\in\ca\colon\phi_{t,x}^*(\mu)
    \leq\phi_{t,x}\big(\tfrac{2r}{\kappa}\big)-2\phi_{t,x}\big(-\tfrac{r}{\kappa}\big)\Big\} \] 
 for the convex conjugate $\phi_{t,x}^*$ as defined in equation~\eqref{def:convconjugate}.\
 Furthermore, due to Dini's theorem, the functional
 \[ \phi\colon\Ck\to\R,\; f\mapsto\sup_{(t,x)\in [0,T]\times K}\phi_{t,x}(f) \]
 is well-defined, convex, monotone and continuous from above. 
 Hence,~\cite[Theorem~2.2]{BCK19} and condition~\ref{S3} imply that
 the set
 \[ M\subset\Big\{\mu\in\ca\colon\phi^*(\mu)
 	\leq\sup_{(t,x)\in [0,T]\times K}\big|\phi_{t,x}\big(\tfrac{2r}{\kappa}\big)
 	-2\phi_{t,x}\big(-\tfrac{r}{\kappa}\big)\big|\Big\} \]
 is $\sigma(\ca,\Ck)$-relatively compact.\ By Prokhorov's theorem, the set 
 $\{\mu_\kappa\colon\mu\in M\}$ is tight, where $\mu_\kappa(A):=\int_A\frac{1}{\kappa}\,\d\mu$ 
 for all $A\in\B(X)$. Hence, there exists $K_1\Subset X$ such that
 \[ (S(t)f)(x)=\widehat{\phi_{t,x}}\big(\tfrac{r}{\kappa}\one_{K_1^c}\big)\leq\epsilon 
    \quad\mbox{for all } (t,x)\in [0,T]\times K. \qedhere \]
\end{proof}

The extended convex monotone semigroup $(S(t))_{t\ge 0}$ on $\Uk$ satisfies the following 
upper semicontinuity property.

\begin{lemma} \label{lem:usc}
 Let $(S(t))_{t\geq 0}$ be an extended convex monotone semigroup on $\Uk$. Then,
  \[ \Glimsup_{n\to \infty}S(t_n)f_n\leq S(t)f \]
 for all sequences $(t_n)_{n\in\N}\subset [0,\infty)$ such that $t:=\lim_{n\to\infty}f_n$
 exists and $(f_n)_{n\in\N}\subset\Uk$ bounded above with $f:=\Glimsup_{n\to\infty}f_n$.
\end{lemma}
\begin{proof}
 In the sequel, we denote by $(S(t))_{t\geq 0}$ the extension to $\Bk$ from Theorem~\ref{thm:ext}.
 Let $\epsilon>0$ and $K\Subset X$. By Lemma~\ref{lem:fe}(ii), there exists $n_0\in\N$ with 
 \[ f_n(x)\leq\fe(x) \quad\mbox{for all } x\in K \mbox{ and } n\geq n_0. \]
 We use the inequality $\fe\geq-\frac{1}{\epsilon\kappa}$ and the monotonicity of $S(t_n)$ 
 to estimate
 \[ S(t_n)f_n\leq S(t_n)\big(\fe+\tfrac{r_\epsilon}{\kappa}\one_{K^c}\big)
 	\quad\mbox{for all } n\geq n_0, \]
 where $r_\epsilon:=\sup_{n\in\N}\|f_n^+\|_\kappa+\frac{1}{\epsilon}<\infty$. 
 Let $\lambda\in (0,1)$. The convexity of $S(t_n)$ implies
 \[ S(t_n)\big(\fe+\tfrac{r_\epsilon}{\kappa}\one_{K^c}\big)
 	\leq \lambda S(t_n)\big(\tfrac{1}{\lambda}\fe\big)
   	+(1-\lambda)S(t_n)\Big(\tfrac{r_\epsilon}{\kappa (1-\lambda)}\one_{K^c}\Big). \]
 Hence, it follows from Lemma~\ref{lem:gamma}(iv) and Theorem~\ref{thm:ext} that
 \[ \Glimsup_{n\to\infty}S(t_n)f_n\leq\lambda S(t)\big(\tfrac{1}{\lambda}\fe\big)
   	+(1-\lambda)\sup_{n\in \N}S(t_n)\Big(\tfrac{r_\epsilon}{\kappa (1-\lambda)}\one_{K^c}\Big). \]
 Since $K\Subset X$ has been arbitrary, we can use Theorem~\ref{thm:ext} to obtain
 \[ \Glimsup_{n\to\infty}S(t_n)f_n\leq\lambda S(t)\big(\tfrac{1}{\lambda}\fe\big). \]
 Furthermore, the inequality $\fe\kappa\leq\|f^+\|_\kappa+\epsilon\leq r_\epsilon$ and the 
 monotonicity of $S(t)$ imply
 \[ \Glimsup_{n\to\infty}S(t_n)f_n\leq\lambda S(t)\big(\tfrac{1}{\lambda}\fe\big)
 	\leq\lambda S(t)\Big(\fe+\big(\tfrac{1}{\lambda}-1\big)\tfrac{r_\epsilon}{\kappa}\Big). \]
 Since $S(t)$ is continuous from above, the right-hand side converges to $S(t)\fe$ as 
 $\lambda\uparrow 1$. It follows from Lemma~\ref{lem:fe}(i) that
 \[ \Glimsup_{n\to\infty}S(t_n)f_n\leq S(t)\fe\downarrow S(t)f \quad\mbox{as } \epsilon\downarrow 0. \qedhere \]
\end{proof}

\begin{remark}\label{rem:topsemigroup}
 Continuity properties for semigroups on spaces of continuous functions w.r.t. the uniform topology 
 are often too restrictive and are therefore relaxed to the mixed topology. For transition semigroups 
 of Markov processes this observation has been discussed in detail in the introduction of \cite{GNR22}.
 We point out that a convex monotone semigroup on $\Ck$ is not necessarily strongly continuous but 
 rather $\Gamma$-upper semicontinuous in time.\ Furthermore, it follows from Remark~\ref{rem:equicont} 
 that a convex monotone semigroup on $\Ck$ is locally uniformly equicontinuous on bounded sets in the 
 sense of~\cite[Definition 3.1(iii)]{GNR22}. A strongly continuous convex monotone semigroup is 
 consequently a $C_0$-semigroup in the sense of \cite[Definition 3.1]{GNR22}. In the linear case, 
 our notion of strongly continuous semigroups on $\Ck$ falls into the class of $C_0$-semigroups on 
 locally convex spaces, see \cite{Yoshida80,kraaij2016strongly}, and into the class of bi-continuous 
 semigroups, see \cite{kuhnemund2003hille}. They are also closely related to semigroups on norming dual 
 pairs together with the strict topology~\cite{Kunze09, kunze2011pettis}. In this context, the mixed 
 topology has first been used in~\cite{GK01} to study transition semigroups of solutions to SDEs with 
 additive noise and Lipschitz drift. Similar but rather pointwise convergence concepts are used in the 
 context of $\pi$-semigroups~\cite{priola1999class} and weakly continuous semigroups~\cite{Cerrai94}.\
 We also want to mention the use of strict topologies in the context of linear semigroup 
 in~\cite{Kraaij19b,vanCasteren11} and the use of $\Gamma$-convergence in~\cite{Kraaij19}. 
\end{remark}

\subsection{Comparison and uniqueness}
\label{sec:comp}

The following theorem is the main result of this section and states that the minimal 
$\Gamma$-supersolution of the abstract Cauchy problem $\partial_t u(t)=\AG^+ u(t)$ 
with $u(0)=f$ is given by $u(t)=S(t)f$ for all $t\geq 0$ and $f\in\L_S^+$. We can not 
prove uniqueness for $\Gamma$-solutions, since $\Gamma$-converge is not symmetric, 
i.e., $f=\Glim_{n\to \infty} f_n$ does not imply $-f=\Glim_{n\to \infty}(-f_n)$. 
However, if $u(t):=T(t)f$ is given by another semigroup, we can reverse the roles of 
$(S(t))_{t\geq 0}$ and $(T(t))_{t\geq 0}$ to obtain uniqueness for strongly continuous 
convex monotone semigroups.\ In Section~\ref{sec:approx}, we investigate how the 
$\Gamma$-generator can be approximated.\ In particular, for $X=\Rd$ and under additional
conditions, we obtain that strongly continuous convex monotone semigroups are uniquely 
determined by their generators evaluated at smooth functions.

\begin{theorem} \label{thm:comp}
 Let $(S(t))_{t\geq0}$ be a convex monotone semigroup on $\Ck$ and $f\in\LSplus$.
 In addition, let $0\leq T_1\leq T_2$ and $u\colon [T_1,T_2]\to\LSplus$ be a function 
 with $S(T_1)f\leq u(T_1)$, $\sup_{t\in [T_1, T_2]}\|u(t)\|_\kappa<\infty$
 and $\Glimsup_{s\to t} u(s)\leq u(t)$ for all $t\in [T_1, T_2]$. Suppose that, for every 
 $t\in [T_1,T_2)$,
 \begin{align}
  \limsup_{h\downarrow 0}\Big\|\Big(\frac{u(t+h)-u(t)}{h}\Big)^-\Big\|_\kappa &<\infty
    \quad\text{and} \label{eq:comp} \\
  \Glimsup_{h\downarrow 0}\Big(\AG^+ u(t)-\frac{u(t+h)-u(t)}{h}\Big) &\leq 0. \label{eq:comp2}
 \end{align}
 Then, it holds $S(t)f\leq u(t)$ for all $t\in [T_1,T_2]$. 
\end{theorem}
\begin{proof}
 It is sufficient to prove the result in the case $T_1=0$ and $S(0)f=u(0)$ since the general 
 case follows immediately by considering $\tilde{u}\colon [0,T_2-T_1]\to\L^S_+,\; t\mapsto u(t+T_1)$.
 Indeed, suppose that the result holds for $T_1=0$ and $S(0)f=u(0)$.\ Since $\tilde{u}$  
 satisfies the conditions imposed on $u$, we obtain $S(t)g\leq\tilde{u}(t)$ for all 
 $t\in [0,T_2-T_1]$, where $g:=u(T_1)$. The semigroup property and the monotonicity of 
 $S(t-T_1)$ imply
 \[ S(t)f=S(t-T_1)S(T_1)f\leq S(t-T_1)g\leq\tilde{u}(t-T_1)=u(t) 
 	\quad\mbox{for all } t\in [T_1, T_2]. \]
 	
 Now, let $T_1=0$ and $S(0)f=u(0)$. For fixed $t\in [0,T]$, we show that
 \[ v\colon [0,t]\to\LSplus,\; s\mapsto S(t-s)u(s) \]
 satisfies $v(0)\leq v(s)$ for all $s\in [0,t]$. First, we show that the mapping
 \begin{equation} \label{eq:v}
  \liminf_{h\downarrow 0}\frac{v(s+h)-v(s)}{h}\geq 0 
  \quad\mbox{for all } s\in [0,t).
 \end{equation}
 Let $s\in [0,t)$. Due to $u(s)\in\LSplus$ and condition~\eqref{eq:comp}, 
 there exists $h_0>0$ with
 \begin{equation} \label{eq:c}
  c:=\sup_{h\in (0,h_0]}\max\left\{\Big\|\Big(\frac{S(h)u(s)-u(s)}{h}\Big)^+\Big\|_\kappa,
  \Big\|\Big(\frac{u(s+h)-u(s)}{h}\Big)^-\Big\|_\kappa\right\}<\infty. 
 \end{equation}
 For every $h\in (0,h_0]$, we define 
 \[ f_h:=\max\left\{\frac{S(h)u(s)-u(s)}{h},-\frac{c}{\kappa}\right\}
 	\quad\mbox{and}\quad 
 	g_h:=\max\left\{-\frac{u(s+h)-u(s)}{h},-\frac{c}{\kappa}\right\}. \]
 It follows from Lemma~\ref{lem:gamma}(vi) that
 \[ f:=\Glimsup_{h\downarrow 0} f_h=\max\left\{\AG ^+ u(s),-\frac{c}{\kappa}\right\}. \]
 Moreover, inequality~\eqref{eq:comp2}, equation~\eqref{eq:c} and Lemma~\ref{lem:gamma}(vi) 
 yield
 \begin{equation} \label{eq:0}
  \Glimsup_{h\downarrow 0}\, (f+g_h)\leq 0.
 \end{equation}  
 Let $\epsilon>0$. Since $u$ is bounded and by condition~\ref{S3}, there exists 
 $\lambda\in (0,1]$ with
 \begin{equation} \label{eq:epsilon}
  \sup_{a,b\in [0,t]}\lambda\|S(a)u(b)\|_\kappa<\epsilon. 
 \end{equation}
 Lemma~\ref{lem:lambda} and inequality~\eqref{eq:epsilon} imply
 \begin{align*}
  &-\liminf_{h\downarrow 0}\frac{v(s+h)-v(s)}{h}
  =\limsup_{h\downarrow 0}\frac{v(s)-v(s+h)}{h} \\
  &=\limsup_{h\downarrow 0}\frac{S(t-s-h)S(h)u(s)-S(t-s-h)u(s+h)}{h} \\
  &\leq\limsup_{h\downarrow 0}\lambda\left(S(t-s-h)\left(\frac{S(h)u(s)-u(s+h)}{\lambda h}
  	+u(s+h)\right)-S(t-s-h)u(s+h)\right) \\
  &\leq\limsup_{h\downarrow 0}\lambda S(t-s-h)\left(\frac{f_h+g_h}{\lambda}+u(s+h)\right)
  	+\frac{\epsilon}{\kappa}. 
 \end{align*}
 Furthermore, the boundedness of $(g_h)_{h\in (0,h_0]}$, $(f_h)_{h\in (0,h_0]}$ and $u$, 
 the conditions~\ref{S1} and~\ref{S3} and Lemma~\ref{lem:uniform} guarantee that we can
 apply Theorem~\ref{thm:limsup} to obtain
 \[ \limsup_{h\downarrow 0} \lambda S(t-s-h)\left(\frac{f_h+g_h}{\lambda}+u(s+h)\right)
 	\leq\limsup_{h\downarrow 0} \lambda S(t-s-h)\bigg(\frac{f+g_h}{\lambda}+u(s+h)\bigg). \]
 Lemma~\ref{lem:usc}, inequality~\eqref{eq:0}, Lemma~\ref{lem:gamma}(iv), the inequality
 $\Glimsup_{h\downarrow 0}u(s+h)\leq u(s)$ and inequality~\eqref{eq:epsilon} yield
 \[ \limsup_{h\downarrow 0} \lambda S(t-s-h)\left(\frac{f+g_h}{\lambda}+u(s+h)\right)
 	\leq\lambda S(t-s)u(s)\leq\frac{\epsilon}{\kappa}. \]
 We combine the previous estimates and let $\epsilon\downarrow 0$ to obtain inequality~\eqref{eq:v}. 
 
 Second, we adapt the proof of~\cite[Lemma~1.1 in Chapter~2]{Pazy83} to show that
 $v(0)\leq v(s)$ for all $s\in [0,t]$. Let $x\in X$ and $\epsilon >0$. Define
 $v(s,x):=(v(s))(x)$ and
 \[ v_\epsilon(\cdot,x)\colon [0,t]\to\R,\; s\mapsto v(s,x)+\epsilon s.\]
 Let $s_0:=\sup\{s\in [0,t]\colon v_\epsilon(0,x)\leq v_\epsilon(s,x)\}$.\ It follows from 
 $\Glimsup_{r\to s} u(r)\leq u(s)$ and Lemma~\ref{lem:usc} that $v_\epsilon(\cdot,x)$ is upper 
 semicontinuous and thus $v_\epsilon(0,x)\leq v_\epsilon(s_0,x)$.\ By contradiction, we assume 
 that $s_0<t$.\ Let $(s_n)_{n\in\N}\subset (s_0,t]$ be a sequence with $s_n\downarrow s_0$. 
 It follows from $v_\epsilon(s_n,x)<v_\epsilon(0,x)\leq v_\epsilon(s_0,x)$ and inequality~\eqref{eq:v} 
 that  
 \[ 0\geq\limsup_{n\to\infty}\frac{v_\epsilon(s_n,x)-v_\epsilon(s_0,x)}{s_n-s_0}
 	=\limsup_{n\to\infty}\frac{v(s_n,x)-v(s_0,x)}{s_n-s_0}+\epsilon\geq\epsilon>0. \]
 This implies $v_\epsilon(0,x)\leq v_\epsilon(t,x)$ and therefore $v(0,x)\leq v(t,x)$ as 
 $\epsilon\downarrow 0$. In particular, we obtain $S(t)f=v(0)\leq v(t)=u(t)$.
\end{proof}

Inspecting the proof of the previous theorem, it seems natural to replace the 
conditions~\eqref{eq:comp} and~\eqref{eq:comp2} by the assumption that
\[ \limsup_{h\downarrow 0}\Big\|\Big(\frac{S(h)u(t)-u(t+h)}{h}\Big)^+\Big\|_\kappa<\infty
	\quad\mbox{and}\quad 
	\Glimsup_{h\downarrow 0}\frac{S(h)u(t)-u(t+h)}{h}\leq 0. \]
Indeed, the previous theorem remains valid and the proof simplifies. In particular, 
we do not need Theorem~\ref{thm:limsup}.\ However, in applications, this assumption 
is not verifiable. Lemma~\ref{lem:gamma}(iv) implies that condition~\eqref{eq:comp2} 
is satisfied if 
\[ \AG^+ u(t)\le{\Gliminf}_{h\downarrow 0}\frac{u(t+h)-u(t)}{h}
	:=-\Big(\Glimsup_{h\downarrow 0}-\frac{u(t+h)-u(t)}{h}\Big).\]
Furthermore, it follows from Lemma~\ref{lem:fe}(ii) that condition~\eqref{eq:comp2} 
is satisfied if
\[ \lim_{h\downarrow 0}\Big(\AG^+ u(t)-\frac{u(t+h)-u(t)}{h}\Big)^+=0 \]
w.r.t.\ the mixed topology.

\begin{theorem} \label{thm:comp2} 
 Let $(S(t))_{t\geq 0}$ and $(T(t))_{t\geq 0}$ be two convex monotone semigroups on $\Ck$ 
 and $\D\subset\LT\cap\LSplus$.\ We denote by $\AG^+$ the upper $\Gamma$-generator of 
 $(S(t))_{t\geq 0}$ and by $B$ and $\BG^+$ the generator and the upper $\Gamma$-generator
 of $(T(t))_{t\geq 0}$, respectively.\ Assume that $T(t)\colon\D\to\D$ for all $t\geq 0$ and
 \begin{equation} \label{eq:AB}
  \AG^+ f\leq B_\Gamma^+ f \quad\mbox{for all } f\in\D.
 \end{equation}
 Then, it holds $S(t)f\leq T(t)f$ for all $t\geq 0$ and $f\in\D\cap D(B)$. 
\end{theorem}	
\begin{proof}
 Let $f\in\D\cap D(B)$ and $u(t):=T(t)f$ for all $t\geq 0$. The conditions~\ref{S3} and~\ref{S4} 
 and the invariance of $\D\subset\LSplus$ imply that $u\colon[0,\infty)\to\LSplus$
 is a well-defined mapping satisfying $u(0)=f$, $\sup_{s\in [0,t]}\|u(s)\|_\kappa<\infty$
 and $\Glimsup_{s\to t} u(s)\leq u(t)$ for all $t\geq 0$. Condition~\eqref{eq:comp} follows 
 from the invariance of $\D\subset\LT$.\ It remains to verify condition~\eqref{eq:comp2}. 
 For every $t\geq 0$ and $h>0$, we use $u(t)\in\D$ and inequality~\eqref{eq:AB} to estimate
 \[ \AG^+ u(t)-\frac{u(t+h)-u(t)}{h}\leq\BG^+ u(t)-\frac{u(t+h)-u(t)}{h}. \]
 Let $t\geq 0$ and $(h_n)_{n\in\N}\subset (0,1]$ be a sequence with $h_n\downarrow 0$. 
 For every $n\in\N$, 
 \begin{equation} \label{eq:BG}
 \BG^+ u(t)-\frac{u(t+h_n)-u(t)}{h_n}
 =\BG^+ u(t)-g_n+g_n-\frac{u(t+h_n)-u(t)}{h_n},
 \end{equation}
 where $g_n:=\frac{1}{h_n}\big(T(t)(f+h_n Bf)-T(t)f\big)$. It follows from Lemma~\ref{lem:lambda} 
 that
 \begin{align*}
  &T(t)(f+h_n Bf)-T(t)\left(\frac{T(h_n)f-f}{h_n}-Bf+f+h_n Bf\right) \\
  &\leq g_n-\frac{u(t+h_n)-u(t)}{h_n} =\frac{T(t)(f+h_n Bf)-T(t)T(h_n)f}{h_n} \\
  &\leq T(t)\left(-\left(\frac{T(h_n)f-f}{h_n}-Bf\right)+T(h_n)f\right)-T(t)T(h_n)f.
 \end{align*}
 Combining the previous estimate with Lemma~\ref{lem:uniform} yields
 \[ g_n-\frac{u(t+h_n)-u(t)}{h_n}\to 0. \] 
 Furthermore, since $T(t)$ is convex, the sequence $(g_n)_{n\in\N}$ is non-increasing. 
 Hence, there exists a function $g\in\Uk$ with $g_n\downarrow g$ and Lemma~\ref{lem:gamma}(iii) 
 and~(v) imply that $g=\BG^+ u(t)$. It follows from inequality~\eqref{eq:BG}, the inequality
 $\BG^+ u(t)-g_n\leq 0$ and Lemma~\ref{lem:gamma}(iv) that condition~\eqref{eq:comp2} is 
 satisfied.\ Theorem~\ref{thm:comp} yields $S(t)f\leq T(t)f$ for all $t\geq 0$. 
\end{proof}

A function $u\colon [0,T]\to \LSplus$ satisfying the conditions from Theorem~\ref{thm:comp}
can be seen as a $\Gamma$-supersolution of the equation
\begin{equation} \label{eq:cauchy}
 \partial_t u(t) = A_\Gamma^+ u(t) \quad\mbox{for all } t\in [0,T], \quad u(0)=f.
\end{equation}
Let $f\in D(A)$ and $v(t):=S(t)f$ for all $t\in[0,T]$. Similar to the proof of Theorem~\ref{thm:comp2}, 
one can show that $v$ a $\Gamma$-supersolution of equation~\eqref{eq:cauchy}.\ Furthermore,
Theorem~\ref{thm:comp} guarantees that $v$ is the smallest $\Gamma$-supersolution.

\begin{remark}
 In view of Theorem~\ref{thm:comp} and Theorem~\ref{thm:comp2}, it becomes apparent that
 our solution concept is based on a weak form of differentiability of the solution 
 while viscosity solutions are defined by comparison with smooth functions.\ In general,
 the domain of the generator containing all functions such that the trajectories $t\mapsto S(t)f$
 are differentiable might not be invariant under the semigroup, see, e.g.,~\cite[Section~4]{CL71} 
 and~\cite[Example~5.2]{DKN21+}. Fortunately, the Lipschitz set $\LS$ containing all functions 
 such the trajectories $t\mapsto S(t)f$ are Lipschitz continuous is invariant.\ Since the 
 existence of the $\Gamma$-limit can a priori not be guaranteed, we define the upper
 $\Gamma$-generator as a limit superior which is naturally defined on the upper Lipschitz 
 set $\LSplus$. The latter is also invariant and, under additional conditions, the upper 
 $\Gamma$-generator coincides with the $\Gamma$-generator, see Section~\ref{sec:approx}.
 Furthermore, the proof of Theorem~\ref{thm:comp} crucially relies on the close relation between 
 upper semicontinuity w.r.t.\ $\Gamma$-convergence, continuity from above and continuity w.r.t. 
 the mixed topology.\ Previously to the introduction of the
 $\Gamma$-generator in this article, the notion of a monotone generator has been studied 
 in~\cite{DKN21+}. Furthermore, in spaces with order continuous norm such as $L^p$-spaces 
 and Orlicz hearts, the domain of the norm generator is invariant and the semigroup defines 
 the unique classical solution of the abstract Cauchy problem, see~\cite{BK22,DKN21}. 
 However, these two approaches are rather limited in applications.\ Finally, in order to
 apply the classical theory of nonlinear semigroup theory based on m-accretive or maximal 
 monotone operators, see~\cite{Barbu10,BC91,Brezis71,CL71,Kato67}, one typically has to 
 solve a fully nonlinear elliptic PDE. The lack of classical solutions for this type of
 equations has been one of the motivations for the introduction of viscosity solutions,
 see, e.g., the discussion in~\cite{Evans87} and~\cite[Chapter~6]{FK06}.\ Since our approach
 does not involve the resolvent of the nonlinear generator, we do not have to address these
 issues. For that reason, we also use Chernoff-type approximations to construct semigroups, 
 see Section~\ref{sec:cher}.\ Finally, we want to emphasize that our approach is self-contained 
 and does, in particular, not rely on the theory of viscosity solutions. 
\end{remark}

\section{Approximation of the $\Gamma$-generator}
\label{sec:approx}

Throughout this section, let $(S(t))_{t\geq 0}$ be a convex monotone semigroup on $\Ck$.
In many applications, it is possible to compute the generator for smooth functions whereas 
the upper $\Gamma$-generator cannot be determined explicitly. Moreover, the semigroup does 
not map smooth functions to smooth functions, see, e.g.,~\cite[Example~5.2]{DKN21+}. Hence, 
in order to apply Theorem~\ref{thm:comp2}, we must show that the upper $\Gamma$-generator 
can be approximated by means of smooth functions. In contrast to the theory of strongly 
continuous linear semigroups, where the norm generator is always a closed operator, we do 
not claim the same to be valid for the upper $\Gamma$-generator.\ Indeed, while the upper 
bound $\AG^+ f\leq\Glimsup_{n\to\infty}\AG^+ f_n$ is satisfied for any approximating sequence such that $(\AG^+ f_n)_{n\in\N}$ is bounded above, 
we prove the equality $\AG^+ f=\Glim_{n\to\infty}\AG^+ f_n$ only under additional conditions 
on the semigroup and for particular choices of the sequence $(f_n)_{n\in\N}$. In the sequel, 
we denote by $(S(t))_{t\geq 0}$ the extended convex monotone semigroup on $\Uk$ from
Theorem~\ref{thm:ext}. For every $t\geq 0$, $f\in\Uk$ and $x\in X$, we define
\[ \left(\int_0^t S(s)f\,\d s\right)(x):=\int_0^t \big(S(s)f\big)(x)\,\d s. \]

\subsection{General approximation results}

The proof of the upper bound is based on the following auxiliary estimate.

\begin{lemma} \label{lem:ub}
 For every $r,T\geq 0$ and $\epsilon>0$, there exists $\lambda_0\in (0,1]$
 with
 \[ S(t)f-f\leq\lambda\int_0^t  S(s)\Big(\frac{1}{\lambda}\AG^+ f+f\Big)\d s
 	+\frac{\epsilon t}{\kappa} \]
 for all $t\in [0,T]$, $f\in B_{\Ck}(r)\cap\LSplus$ and $\lambda\in (0,\lambda_0]$.
\end{lemma}
\begin{proof}
 Let $r,T\geq 0$ and $\epsilon>0$. By condition~\ref{S3}, there exists $\lambda_0\in (0,1]$ with
 \begin{equation} \label{eq:ub}
  \sup_{t\in [0,T]}\sup_{f\in B_{\Ck}(r)}\lambda_0\|S(t)f\|_\kappa\leq\epsilon
  \quad\mbox{and}\quad \lambda_0 T\leq 1.
 \end{equation}
 In the sequel, we fix $t\in [0,T]$, $f\in B_{\Ck}(r)\cap\LSplus$ and 
 $\lambda\in (0,\lambda_0]$. Define $h_n:=2^{-n}t$ and $t_n^k:=k2^{-n}t$ 
 for all $k,n\in\N_0$. For every $n\in\N$, it follows from the semigroup property, 
 inequality~\eqref{eq:ub} and Lemma~\ref{lem:lambda} that
 \begin{align*}
  S(t)f-f
  &=\sum_{k=1}^{2^n} \big(S(t_n^k)f-S(t_n^{k-1})f\big)
  =\sum_{k=1}^{2^n} \big(S(t_n^{k-1})S(h_n)f-S(t_n^{k-1})f\big) \\
  &\leq\lambda h_n\sum_{k=1}^{2^n} \left(S(t_n^{k-1})\left(\frac{S(h_n)f-f}{\lambda h_n}+f\right)
  	-S(t_n^{k-1})f\right) \\
  &\leq\lambda h_n\sum_{k=1}^{2^n} S(t_n^{k-1})\left(\frac{S(h_n)f-f}{\lambda h_n}+f\right)
  	+\frac{\epsilon t}{\kappa} \\
  &=\lambda\int_0^t \sum_{k=1}^{2^n} S(t_n^{k-1})\left(\frac{S(h_n)f-f}{\lambda h_n}+f\right)
  	\one_{[t_n^{k-1},t_n^k)}(s)\,\d s+\frac{\epsilon t}{\kappa}.
 \end{align*}
 Due to $f\in\LSplus$ and condition~\ref{S3}, the sequence inside the integral is bounded 
 above. Hence, we can apply Fatou's lemma and Lemma~\ref{lem:usc} to obtain
 \begin{align*}
  S(t)f-f
  &\leq\lambda\int_0^t \limsup_{n\to\infty}\sum_{k=1}^n S(t_n^{k-1})
  	\left(\frac{S(h_n)f-f}{\lambda h_n}+f\right)
  	\one_{[t_n^{k-1},t_n^k)}(s)\,\d s+\frac{\epsilon t}{\kappa} \\
  &\leq\lambda\int_0^t S(s)\Big(\frac{1}{\lambda}\AG^+ f+f\Big)\d s
  	+\frac{\epsilon t}{\kappa}.\qedhere
 \end{align*}
\end{proof}

\begin{theorem} \label{thm:ub}
 Let $f\in\Ck$ and $(f_n)_{n\in\N}\subset\LSplus$ be a sequence with $f_n\to f$ 
 such that $(\AG^+ f_n)_{n\in\N}$ is bounded above. Then, it holds $f\in\LSplus$
 and $\AG^+ f\leq\Glimsup_{n\to\infty}\AG^+ f_n$. 
\end{theorem}
\begin{proof}
 Choose $\epsilon>0$ and $\lambda_0\in (0,1]$ such that the statement of Lemma~\ref{lem:ub}
 is valid with $r:=\sup_{n\in\N}\|f_n\|_\kappa$ and $t_0:=1$.\ Then, for every $h\in (0,1]$ 
 and $\lambda\in (0,\lambda_0]$, it follows from Lemma~\ref{lem:uniform}, Fatou's lemma, 
 Lemma~\ref{lem:usc} and Lemma~\ref{lem:gamma}(iv) that
 \begin{align*}
  \frac{S(h)f-f}{h}
  &=\lim_{n\to\infty}\frac{S(h)f_n-f_n}{h} 
  \leq\limsup_{n\to\infty}\frac{1}{h}\int_0^h 
  	\lambda S(s)\Big(\frac{1}{\lambda}\AG^+ f_n+f_n\Big)\d s +\frac{\epsilon}{\kappa} \\
  &\leq\frac{1}{h}\int_0^h \limsup_{n\to\infty}
  	\lambda S(s)\Big(\frac{1}{\lambda}\AG^+ f_n+f_n\Big)\d s +\frac{\epsilon}{\kappa} 
  \leq\frac{1}{h}\int_0^h \lambda S(s)\Big(\frac{1}{\lambda}g+f\Big)\d s
  	+\frac{\epsilon}{\kappa},
 \end{align*}
 where $g:=\Glimsup_{n\to\infty}\AG^+ f_n$. Condition~\ref{S3} implies $f\in\LSplus$. 
 Moreover, it follows from Lemma~\ref{lem:gamma}(iv), Lemma~\ref{lem:usc} and 
 $S(0)=\id_{\Uk}$ that
 \[ \AG^+ f\leq\Glimsup_{h\downarrow 0}\frac{1}{h}\int_0^h 
 	\lambda S(s)\Big(\frac{1}{\lambda}g+f\Big)\d s +\frac{\epsilon}{\kappa}
  	\leq g+\lambda f+\frac{\epsilon}{\kappa} \quad\mbox{for all } \lambda\in (0,\lambda_0]. \]
 Letting $\epsilon, \lambda\downarrow 0$, we obtain $\AG^+ f\leq\Glimsup_{n\to\infty}\AG^+ f_n$.
\end{proof}

The previous result relies on the fact that the semigroup is upper semicontinuous w.r.t. 
$\Gamma$-convergence and that the $\Gamma$-limit superior exists for any sequence which
is bounded above. Since the semigroup is not continuous w.r.t. $\Gamma$-convergence,
the reverse estimate is only valid for particular choices of the approximating sequence. 
Condition~\eqref{eq:lb} incorporates the fact that the $\Gamma$-limit is not determined by
evaluating a sequence of functions at a fixed point but also depends on the values of
these functions in an arbitrarily small neighbourhood of that point. In Subsection~\ref{sec:conv},
we study the construction of such sequences by means of convolution and truncation.
These results are particular interesting for applications with $X:=\Rd$, where one can
typically verify that the domain of the generator contains smooth functions with bounded
derivatives or at least smooth function with compact support.

\begin{theorem} \label{thm:lb}
 Let  $f\in\Ck$ and $(f_n)_{n\in\N}\subset D(\AG)$ be a sequence with $f_n\to f$ 
 such that $(\AG f_n)_{n\in\N}$ is bounded above. Suppose that, for every $\epsilon>0$ 
 and $K\Subset X$, there exist $n_0\in\N$, $h_0>0$ and a sequence 
 $(r_n)_{n\in\N}\subset (0,\infty)$ with $r_n\to 0$ such that 
 \begin{equation} \label{eq:lb}
  \left(\frac{S(h)f_n-f_n}{h}\right)(x)
  \leq\sup_{y\in B(x,r_n)}\left(\frac{S(h)f-f}{h}\right)(y)+\frac{\epsilon}{\kappa(x)}
 \end{equation}
 for all $h\in (0,h_0]$, $n\geq n_0$ and $x\in K$. Then, $f\in D(\AG)$ and
 $\AG f=\Glim_{n\to\infty}\AG f_n$. 
\end{theorem}
\begin{proof} 
 First, let $(h_m)_{m\in\N}\subset (0,\infty)$ be a sequence with $h_m\to 0$. Since 
 Theorem~\ref{thm:ub} guarantees that $f\in\LSplus$, we can apply Lemma~\ref{lem:gamma}(i)
 to choose a subsequence, still denoted by $(h_m)_{m\in\N}$, such that
 \begin{equation} \label{eq:g1}
  g_1:=\Glim_{m\to\infty}\frac{S(h_m)f-f}{h_m}\in\Uk
 \end{equation}
 exists. Moreover, since $(\AG f_n)_{n\in\N}$ is bounded above and due to Lemma~\ref{lem:gamma}(i),
 every subsequence $(f_{n_k})_{k\in\N}$ has a further subsequence $(f_{n_{k_l}})_{l\in\N}$ such that 
 \begin{equation} \label{eq:g2}
  g_2:=\Glim_{l\to\infty}\AG f_{n_{k_l}}\in\Uk 
 \end{equation}
 exists. To simplify the notation, we write $f_l:=f_{n_{k_l}}$ for all $l\in\N$. 
 Theorem~\ref{thm:ub} implies
 \[ g_1\leq\AG^+ f\leq\Glim_{l\to\infty}\AG f_l=g_2. \]
	
 Second, we show that $g_1\geq g_2$. To do so, let $x\in X$ and $\epsilon>0$. 
 By definition of the $\Gamma$-limit, we can choose a sequence $(x_l)_{l\in\N}\subset X$
 with $x_l\to x$ such that 
 \[ \Big(\Glim_{l\to\infty}\AG f_l\Big)(x)=\lim_{l\to\infty}\AG f_l(x_l).\]
 In addition, there exist sequences $(m_l)_{l\in\N}$ and $(y_l)_{l\in\N}$ with
 $d(x_l,y_l)\to 0$ such that 
 \[ \AG f_l(x_l)
 	\leq\left(\frac{S(h_{m_l})f_l-f_l}{h_{m_l}}\right)(y_l)+\epsilon
 	\quad\mbox{for all } l\in\N. \]
 Since $K:=\{y_l\colon l\in \N\}\cup \{x\}$ is compact, there exist $h_0>0$ and
 $(r_l)_{l\in\N}\subset (0,\infty)$ satisfying condition~\eqref{eq:lb}. 
 For every $l\in\N$ with $n_{k_l}\geq n_0$ and $h_{m_l}\leq h_0$, 
 \begin{align*}
  \AG f_l(x_l)
  &\leq\sup_{z\in B(y_l,r_l)}\left(\frac{S(h_{m_l})f-f}{h_{m_l}}\right)(z)
  +\frac{\varepsilon}{\kappa(y_l)}+\epsilon \\
  &\leq\sup_{z\in B(x,\delta_l)}\left(\frac{S(h_{m_l})f-f}{h_{m_l}}\right)(z)
  +\sup_{l\in\N}\frac{\epsilon}{\kappa(y_l)}+\epsilon,
 \end{align*}
 where $\delta_l:=d(x,y_l)+r_l\to 0$. Since $\inf_{l\in\N}\kappa(y_l)>0$,
 Lemma~\ref{lem:gamma}(vii) implies
 \[ g_2(x)=\Big(\Glim_{l\to\infty}\AG f_l\Big)(x)=\lim_{l\to\infty}\AG f_l(x_l)
 	\leq\left(\Glim_{l\to\infty}\frac{S(h_{m_l})f-f}{h_{m_l}}\right)(x)=g_1(x).\]

 Third, we show that $f\in D(\AG)$ with $\AG f=\Glim_{n\to\infty}\AG f_n$. From 
 the first part, we know that every sequence $(h_m)_{m\in\N}\subset(0,\infty)$ 
 with $h_m\to 0$ has a subsequence which satisfies equation~\eqref{eq:g1}. A 
 priori the choice of the subsequence and the limit $g_1$ depend on the choice 
 of the sequence $(h_m)_{m\in\N}$. However, it holds $g_1=g_2$ and the function 
 $g_2$ is independent of $(h_m)_{m\in\N}$. Hence, Lemma~\ref{lem:gamma}(ii) 
 implies
 \[ g_1=\Glim_{h\downarrow 0}\frac{S(h)f-f}{h} \]
 which means that $f\in D(\AG)$ with $\AG f=g_1$. Since the limit in equation~\eqref{eq:g2}
 is also independent of the choice of the subsequence, we obtain
 $\AG f=\lim_{n\to\infty}\AG f_n$.
\end{proof}

\subsection{Regularization by convolution and truncation}
\label{sec:conv}

We study two particular constructions for approximating sequences satisfying 
condition~\eqref{eq:lb}.

\subsubsection{Convolution}
\label{sec:conv2}

Let $X$ be a separable Banach space and 
\begin{equation} \label{eq:kappa}
 c_\kappa:=\sup_{x\in X}\sup_{y\in B_X(1)}\frac{\kappa(x)}{\kappa(x-y)}<\infty,
\end{equation}
where $B_X(1):=\{y\in X\colon\|x\|_X\leq 1\}$. Let $(\mu_n)_{n\in\N}$ be a sequence of 
probability measures on the Borel $\sigma$-algebra $\B(X)$ satisfying
\begin{equation} \label{eq:conv2} 
 \mu_n\big(B_X\big(\tfrac{1}{n}\big)^c\big)=0\quad\mbox{for all }n\in\N.
\end{equation}
For every $n\in\N$, $f\in\Uk$ and $x\in X$, we define 
\[ (f*\mu_n)(x):=\int_X f(x-y)\mu_n(\d y)\in\Rbar, \]
where condition~\eqref{eq:kappa} guarantees that the integral is well-defined.

\begin{lemma} \label{lem:conv}
 Let $n\in\N$. Then, it holds $f*\mu_n\in\Uk$ for all $f\in\Uk$ and  
 \[ \Glimsup_{m\to\infty}\, (f_m*\mu_n)
 	\leq\Big(\Glimsup_{m\to\infty}f_m\Big)*\mu_n \]
 for every sequence $(f_m)_{m\in\N}\subset\Uk$ which is bounded above. Furthermore,
 \[ \Glimsup_{n\to\infty}\,(f*\mu_n)\leq f \quad\mbox{for all } f\in\Uk. \]  
\end{lemma}
\begin{proof}
 First, let $n\in\N$, $f\in\Uk$ and $x\in X$. Condition~\eqref{eq:kappa} and 
 condition~\eqref{eq:conv2} imply
 \[ (f*\mu_n)(x)\kappa(x)
 	=\int_{B_X(1)} f(x-y)\kappa(x-y)\frac{\kappa(x)}{\kappa(x-y)}\,\mu_n(\d y)
	\leq c_\kappa\|f^+\|_\kappa.\]
 Moreover, for every sequence $(x_m)_{m\in\N}\subset X$ with $x_m\to x$, Fatou's lemma yields
 \[ \limsup_{m\to\infty}\,(f*\mu_n)(x_m) 
    \leq\int_X \limsup_{m\to\infty}f(x_m-y)\,\mu_n(\d y)
	\leq\int_X f(x-y)\,\mu_n(\d y). \]
 This shows that $f*\mu_n$ is upper semicontinuous. 

 Second, let $(f_m)_{m\in\N}\subset\Uk$ be bounded above, $x\in X$ and $(x_m)_{m\in\N}\subset X$ 
 with $x_m\to x$. It follows from Fatou's lemma that
 \begin{align*}
  &\limsup_{m\to \infty}\,(f*\mu_n)(x_m)
  =\limsup_{m\to \infty}\int_X f_m(x_m-y)\,\mu_n(\d y) \\
  &\leq\int_X \limsup_{m\to \infty}f_m(x_m-y)\,\mu_n(\d y)
  \leq \int_X \Big(\Glimsup_{m\to\infty}f_m\Big)(x-y)\,\mu_n(\d y).
 \end{align*}
 This shows that $\Glimsup_{m\to\infty}(f_m*\mu_n)\leq (\Glimsup_{m\to\infty}f_m)*\mu_n$.
 
 Third, let $x\in X$ and $(x_n)_{n\in\N}\subset X$ with $x_n\to x$. Since $f$ is upper 
 semicontinuous, for every $\epsilon>0$, there exists $n_0\in\N$ such that 
 $f(x_n-y)\leq f(x)+\epsilon$ for all $n\geq n_0$ and $y\in B_X(\nicefrac{1}{n})$. 
 We obtain 
 \[ (f*\mu_n)(x_n)=\int_{B_X(\nicefrac{1}{n})}f(x_n-y)\,\mu_n(\d y)
 	\leq f(x)+\epsilon \quad\mbox{for all } n\geq n_0. \]
 Letting $\epsilon\downarrow 0$ yields $\limsup_{n\to\infty}(f*\mu_n)(x_n)\leq f(x)$.
\end{proof}

Condition~\eqref{eq:trans} in the following lemma is crucial in order to control 
the difference between $S(t)(f*\mu_n)$ and $(S(t)f)*\mu_n$.\ We will encounter this 
condition again in Section~\ref{sec:cher}, where we study the construction of semigroups 
via Chernoff-type approximation schemes. For every $f\colon X\to\R$ and $x,y\in X$, 
we define $(\tau_x f)(y):=f(x+y)$.

\begin{theorem} \label{thm:conv}
 Let $f\in\LSplus$ such that $f_n:=f*\mu_n\in D(\AG)$ for all $n\in\N$. 
 Assume that, for every $\epsilon>0$, there exist $\delta, t_0>0$ with
 \begin{equation} \label{eq:trans}
  S(t)(\tau_x f)\leq\tau_x S(t)f+\tfrac{\epsilon t}{\kappa}
  \quad\mbox{for all } t\in[0,t_0] \mbox{ and } x\in B_X(\delta). 
 \end{equation}
 Then, it holds $f\in D(\AG)$ and $\AG f=\Glim_{n\to\infty}\AG f_n$. 
\end{theorem}
\begin{proof}
 We verify the assumptions of Theorem~\ref{thm:lb}. First, we show that $f_n\to f$. 
 Condition~\eqref{eq:kappa} guarantees $\|f_n\|_\kappa\leq c_\kappa\|f\|_\kappa$ 
 for all $n\in\N$. Moreover, for every $K\Subset X$, 
 \[ \sup_{x\in K}|f_n(x)-f(x)|
	\leq\sup_{x\in K}\int_{B_X(\nicefrac{1}{n})} |f(x-y)-f(x)|\,\mu_n(\d y)
	\to 0\quad\mbox{as } n\to\infty.\]
	
 Second, we show that condition~\eqref{eq:lb} is satisfied. Let $\epsilon>0$ and $K\Subset X$.
 Due to condition~\eqref{eq:trans}, there exist $h_0>0$ and $n_0\in\N$ with
 \[ S(h)(\tau_{-y}f)\leq\tau_{-y}S(h)f+\tfrac{\epsilon h}{\kappa} \quad\mbox{for all } h\in [0,h_0]
 \mbox{ and } y\in B_X(\nicefrac{1}{n_0}). \]
 Hence, for every $h\in [0,h_0]$, $n\geq n_0$ and $x\in X$, it follows from Jensen's inequality 
 and the monotonicity of $S(h)$ that
 \begin{align*}
  &(S(h)f_n)(x)
  =\bigg(S(h)\bigg(\int_{B_X(\nicefrac{1}{n})}(\tau_{-y}f)(\,\cdot\,)\,\mu_n(\d y)\bigg)\bigg)(x) \\
  &\leq\int_{B_X(\nicefrac{1}{n})}\big(S(h)(\tau_{-y}f)\big)(x)\,\mu_n(\d y)
  \leq\int_{B_X(\nicefrac{1}{n})}\Big(\big(\tau_{-y}S(h)f\big)(x)
  	+\tfrac{\epsilon h}{\kappa(x)}\Big)\,\mu_n(\d y) \\
  &=\Big((S(h)f)*\mu_n+\tfrac{\epsilon h}{\kappa}\Big)(x).
 \end{align*}
 For every $h\in (0,h_0]$ and $n\geq n_0$, condition~\eqref{eq:conv2} implies
 \begin{equation*} 
  \frac{S(h)f_n-f_n}{h}
  \leq\left(\frac{S(h)f-f}{h}\right)*\mu_n+\frac{\epsilon}{\kappa}
  \leq\sup_{y\in B_X(\,\cdot\,,\nicefrac{1}{n})}\Big(\frac{S(h)f-f}{h}\Big)(y)+\frac{\epsilon}{\kappa}.
 \end{equation*}

 Third, we show that $(\AG f_n)_{n\in\N}$ is bounded above. The previous inequality applied with
 with $\epsilon:=1$ and condition~\eqref{eq:kappa} yield $h_0>0$ and $n_0\in\N$ with 
 \[ \left\|\frac{(S(h)f_n-f_n)^+}{h}\right\|_\kappa
    \leq c_\kappa\left\|\frac{(S(h)f-f)^+}{h}\right\|_\kappa+1 \]
 for all $h\in [0,h_0]$ and $n\geq n_0$. Hence, it follows from $f\in\LSplus$ that  
 $(\AG^+ f_n)_{n\in\N}$ is bounded above and Theorem~\ref{thm:lb} yields the claim.
\end{proof}

Using the previous approximation results, we obtain the following comparison result which forms the basis 
for comparison in terms of $\Cbi$ and $\Cci$, see Theorem~\ref{thm:conv3} and Theorem~\ref{thm:trunc}  below.
\begin{theorem}\label{thm:conv2}
 Let $(S(t))_{t\geq 0}$ and $(T(t))_{t\geq 0}$ be two convex monotone semigroups on $\Ck$.
 We denote by $\AG^+$ the upper $\Gamma$-generator of $(S(t))_{t\geq 0}$ and by $B$ and $\BG$
 the generator and the $\Gamma$-generator of $(T(t))_{t\geq 0}$, respectively. Let $\C\subset\Ck$
 and $\D\subset\LT$ be two subsets satisfying the following conditions:
 \begin{enumerate}
  \item For every $f\in\D$ and $\epsilon>0$, there exist $\delta,t_0>0$ with 
   \[ T(t)(\tau_x f)\leq \tau_x T(t)f+\tfrac{\epsilon t}{\kappa} 
    \quad\mbox{for all } t\in [0,t_0] \mbox{ and } x\in B_X(\delta). \]
   Furthermore, it holds $f\ast\mu_n\in\C$ for all $n\in\N$ and $f\in\D$.
  \item It holds $T(t)\colon\D\to\D$ for all $t\geq0$.
  \item It holds $\C\cap\LTplus\subset\LSplus\cap D(\BG)$ and $\AG^+f\leq\BG f$ for all $f\in\C\cap\LTplus$.
 \end{enumerate}
 Then, it holds $S(t)f\leq T(t)f$ for all $t\geq 0$ and $f\in\D\cap D(B)$. 
\end{theorem}
\begin{proof}
 The arguments are very similar to the proof of~\cite[Theorem 2.5]{BK22b} but for the sake 
 of a self-contained exposition, we outline the details. In order to apply Theorem \ref{thm:comp2}, 
 we have to show that $\D\subset\LSplus$ and $\AG^+f\leq\BG^+f$ for all $f\in\D$. To do so, 
 let $f\in\D$, $\epsilon>0$ and $f_n:=f\ast\mu_n$ for all $n\in\N$. Choose $\delta,t_0>0$ such 
 that condition~(i) is satisfied. Then, following the lines along the proof of~\cite[Theorem 2.5]{BK22b}, 
 we use Jensen's inequality, condition~(i) and condition~\eqref{eq:kappa} to obtain
\begin{align*}
 \big(T(t)f_n-f_n\big)(x)\kappa(x)
 &\leq\int_{B_X(\delta)}\big(T(t)(\tau_{-y}f)-\tau_{-y}f\big)(x)\kappa(x)\,\mu_n(\d y) \\
 &\leq\int_{B_X(\delta)}\big(\tau_{-y}T(t)f-\tau_{-y}f\big)(x)\kappa(x)\,\mu_n(\d y)+\epsilon t \\
 &\leq c_\kappa\int_{B_X(\delta)}\big(T(t)f-f\big)(x-y)\kappa(x-y)\,\mu_n(\d y)+\epsilon t \\
 &\leq c_\kappa\Big(\big((T(t)f-f)^+\kappa\big)*\mu_n\Big)(x)+\epsilon t
 \end{align*}
 for all $t\in [0,t_0]$, $x\in X$ and $n\in\N$ with $\frac{1}{n}\leq\delta$. Hence,
 \[ \big\|(T(t)f_n-f_n)^+\big\|_\kappa\leq c_\kappa\big\|(T(t)f-f)^+\big\|_\kappa+\epsilon t \]
 for all $t\in [0,t_0]$ and $n\in\N$ with $\frac{1}{n}\leq\delta$. Combining the previous 
 inequality with the conditions~(i) and~(iii) yields $f_n\in\C\cap\LTplus\subset\LSplus\cap D(\BG)$ and
 \[ \|(\AG^+ f_n)^+\|_\kappa\leq\|(\BG f_n)^+\|_\kappa\leq c_\kappa\|(\BG^+ f)^+\|_\kappa+\epsilon
    \quad\mbox{for all } n\in\N \mbox{ with } \tfrac{1}{n}\leq\delta. \]
 It follows from Theorem \ref{thm:ub} and Theorem \ref{thm:conv} that $f\in\LSplus$ and, invoking Lemma~\ref{lem:gamma}(iv),
 \[ \AG^+ f\leq\Glimsup_{n\to\infty}\AG^+ f_n\leq\Glim_{n\to\infty}\BG f_n=\BG f=\BG^+ f. \]
 Hence, Theorem~\ref{thm:comp2} implies $S(t)f\leq T(t)f$ for all $t\geq 0$ and $f\in\D\cap D(B)$.
\end{proof}

For the rest of this subsection, let $X:=\Rd$ and $\eta\colon\Rd\to\R_+$ be an infinitely 
differentiable function with $\supp(\eta)\subset B_{\Rd}(1)$ and $\int_{\Rd}\eta(x)\,\d x=1$. 
For every $n\in\N$, $x\in\Rd$ and locally integrable function $f\colon\Rd\to\R$, we define 
$\eta_n(x):=n^d\eta(nx)$ and
\[ (f*\eta_n)(x):=\int_{\Rd} f(x-y)\eta_n(y)\,\d y.\]
Denote by $\Lipb$ the space of all bounded Lipschitz continuous functions $f\colon\Rd\to\R$ 
and by $\Cbi$ the space of all bounded infinitely differentiable functions $f\colon\Rd\to\R$ 
such that the partial derivatives of any order are bounded.

The following theorem uniquely characterizes semigroups via their generators evaluated 
at smooth functions.\ In particular, while a thorough understanding of the $\Gamma$-generator
and related topological properties is crucial in order to prove the abstract comparison
principle in Subsection~\ref{sec:comp}, in applications one only has to verify several explicit
conditions.

\begin{theorem} \label{thm:conv3}
 Let $(S(t))_{t\geq 0}$ and $(T(t))_{t\geq 0}$ be two convex monotone semigroups on $\Ck$ 
 with generators $A$ and $B$, respectively. Assume that $(T(t))_{t\geq 0}$ satisfies 
 condition~\eqref{eq:trans} for all $f\in\LT\cap\Lipb$ and that $T(t)\colon\Lipb\to\Lipb$
 for all $t\geq 0$. Furthermore, suppose that $\Cbi\subset D(A)\cap D(B)$ and
 \[ Af\leq Bf\quad\text{for all }f\in \Cbi. \]
 Then, it holds $S(t)f\leq T(t)f$ for all $t\geq 0$ and $f\in \Ck$. 
\end{theorem}
\begin{proof}
 Applying Theorem~\ref{thm:conv2} with $\D:=\LT\cap\Lipb$ and $\C:=\Cbi$ yields $S(t)f\leq T(t)f$ 
 for all $t\geq 0$ and $f\in D(B)\cap\Lipb$. Since the set $\Cbi\subset D(B)\cap\Lipb$ is dense 
 in $\Ck$, see Remark~\ref{rem:dense} below, Lemma~\ref{lem:uniform} implies $S(t)f\leq T(t)f$ 
 for all $t\geq 0$ and $f\in\Ck$.
\end{proof}

\subsubsection{Truncation}

Let $X:=\Rd$ and denote by $\Cci$ the space of all infinitely differentiable functions 
$f\colon\Rd\to\R$ with compact support. Let $(\phi_n)_{n\in\N}\subset\Cci$ be a sequence 
with $0\leq\phi_n\leq 1$ for all $n\in\N$ and $\phi_n(x)=1$ for all $n\in\N$ and $x\in B_{\Rd}(n)$.

\begin{lemma}\label{lem:trunc}
 Suppose that $\Cci\subset D(A)$.\ Let $f\in\Ck$ be infinitely differentiable with $f\geq 0$ 
 and define $f_n:=f\phi_n$ for all $n\in\N$.
 \begin{enumerate}
  \item If $f\in\LSplus$, then $(Af_n)(x)\leq (\AG^+f)(x)$ for all $n\in\N$ and $x\in B_{\Rd}(n)$. 
  \item If $(A f_n)_{n\in\N}$ is bounded above, then $f\in D(A_\Gamma)$ and 
	$A_\Gamma f=\Glim_{n\to\infty}Af_n$. 
 \end{enumerate}
\end{lemma}	
\begin{proof}
 First, for every $h>0$, $n\in\N$ and $x\in B_{\Rd}(n)$, the monotonicity of $S(h)$ yields
 \begin{equation}\label{eq:trunc}
  \frac{(S(h)f_n)(x)-f_n(x)}{h}=\frac{(S(h)f_n)(x)-f(x)}{h}\leq\frac{(S(h)f)(x)-f(x)}{h}.
 \end{equation}
 Since $\Cci\subset D(A)$, it follows from $f\in\LSplus$ and the previous inequality that
 \[ (Af_n)(x)\leq (\AG^+f)(x) \quad\mbox{for all } n\in\N \mbox{ and } x\in B_{\Rd}(n). \] 

 Second, inequality~\eqref{eq:trunc} guarantees that condition~\eqref{eq:lb} is satisfied. 
 Furthermore, it holds $f_n\in D(A)$ for all $n\in\N$ and $f_n\to f$. Hence, if $(A f_n)_{n\in\N}$ 
 is bounded above, Theorem~\ref{thm:lb} implies $f\in D(A_\Gamma)$ and $A_\Gamma f=\Glim_{n\to\infty}Af_n$. 
\end{proof}

\begin{theorem} \label{thm:trunc}
 Let $(S(t))_{t\geq 0}$ and $(T(t))_{t\geq 0}$ be two convex monotone semigroups on $\Ck$ 
 with generators $A$ and $B$, respectively. Assume that $(T(t))_{t\geq 0}$ satisfies 
 condition~\eqref{eq:trans} for all $f\in\LT\cap\Lipb$ with $f\geq 0$ and that
 $T(t)\colon\Lipb\to\Lipb$ for all $t\geq 0$. Furthermore, let $\Cci\subset D(A)\cap D(B)$
 and let $(Bf_n)_{n\in\N}$ be bounded above for all $f\in\Cbi\cap\LTplus$ with $f\geq 0$, 
 where $f_n:=f\phi_n$ for all $n\in\N$. Assume that 
 \[ Af\leq Bf \quad\mbox{for all } f\in \Cci. \]
 Then, it holds $S(t)f\leq T(t)f$ for all $t\geq 0$ and $f\in\Ck$ with $f\geq 0$. 
\end{theorem}
\begin{proof}
 Let $\C:=\{f\in\Cbi\colon f\geq0\}$ and $\D:=\{f\in\LT\cap\Lipb\colon f\geq0\}$. 
 It follows from Theorem~\ref{thm:ub} and Lemma~\ref{lem:trunc} that 
 $\C\cap\LTplus\subset\LSplus\cap D(\BG)$ and 
 \[ \AG^+ f\leq\Glimsup_{n\to\infty}Af_n\leq\Glim_{n\to\infty}Bf_n=\BG f
    \quad\mbox{for all } f\in\C\cap\LTplus. \]
 Hence, Theorem~\ref{thm:conv2} yields $S(t)f\leq T(t)f$ for all $t\geq 0$ and 
 $f\in D(B)\cap\Lipb$ with $f\geq 0$. Since the set
 $\{f\in\Cci\colon f\geq 0\}\subset\{D(B)\cap\Lipb\colon f\geq 0\}$ is dense in
 $\{f\in\Ck\colon f\geq 0\}$, Lemma~\ref{lem:uniform} implies $S(t)f\leq T(t)f$ for all 
 $t\geq 0$ and $f\in\Ck$ with $f\geq 0$. 
\end{proof}

Since inequality~\eqref{eq:trunc} is only valid for non-negative functions, 
the previous theorem has a priori the same restriction.\ However, in many
applications one can show that $S(t)(f+c)=S(t)f+c$ and $T(t)(f+c)=T(t)f+c$ for
all $t\geq 0$, $f\in\Ck$ and $c\in\R$. In this case, we obtain $S(t)f\leq T(t)f$
for all $t\geq 0$ and $f\in\Ck$.

\subsection{Link to distributional derivatives}
\label{sec:sobolev}

Let $X:=\Rd$. In many applications, one can show that $\Cbi\subset D(A)$ and 
$Af=H((D^\alpha f)_{\alpha\in I})$ for all $f\in\Cbi$, where $I\subset\N_0^d$ 
is a finite index set and $H\colon\R^I\to\R$ is a convex function. Furthermore, 
one can show that 
\[ S(t)\colon\LSsym\cap\Lipb\to\LSsym\cap\Lipb \quad\mbox{for all } t\geq 0 \]
and the set $\LSsym\cap\Lipb$ often admits an explicit representation by means of 
Sobolev spaces. The latter allows to define $H((D^\alpha f)_{\alpha\in I})$, where 
the partial derivatives $D^\alpha f$ are regular distributions. In Subsection~\ref{sec:conv2}, 
we also identified conditions such that
\[ \LSsym\cap\Lipb\subset\LSplus\cap\Lipb\subset D(\AG). \]
Hence, the question arises whether $u(t):=S(t)f$ solves the equation
\[ \AG u(t)=H\big((D^\alpha u(t))_{\alpha\in I}\big) \quad\mbox{for all } t\geq 0. \]
The left-hand side of this equation is an upper semicontinuous function while the right-hand 
side is an equivalence class of functions coinciding almost everywhere and the upper 
semicontinuous hull strongly depends on the choice of the representative. However, 
a canonical choice is given by Lebesgue's differentiation theorem. For a locally 
integrable function $f\colon\Rd\to\R$, we define\footnote{Denoting by $\lambda$ the 
Lebesgue measure, the normalized integral is given by
\[ \fint_{B(x,r)}f(y)\,\d y:=\frac{1}{\lambda(B(x,r))}\int_{B(x,r)}f(y)\,\d y.\]}
\[ X_f:=\left\{x\in\Rd\colon\lim_{r\downarrow 0}\fint_{B(x,r)}f(y)\,\d y\in\R \mbox{ exists}\right\}. \]
The set $X_f\subset\Rd$ is dense and it holds $X_f=X_g$ for any other function 
$g\colon\Rd\to\R$ such that $f=g$ almost everywhere. Let $\eta\colon\Rd\to\R_+$ be an 
infinitely differentiable function satisfying $\supp(\eta)\subset B_{\Rd}(1)$ and 
$\int_{\Rd}\eta(x)\,\d x=1$. For every $n\in\N$, $x\in\Rd$ and locally integrable 
function $f\colon\Rd\to\R$, we define $\eta_n(x):=n^d\eta(nx)$ and
\[ (f*\eta_n)(x):=\int_{\Rd} f(x-y)\eta_n(y)\,\d y.\]
The functions $f*\eta_n$ are infinitely differentiable. We also want to emphasize that 
the construction of the function $\overline{g}$ in the following theorem does not depend 
on the choice of the representative for the distributional derivatives $D^\alpha f$.

\begin{theorem} \label{thm:dist}
 Let $I\subset\N_0^d$ be a finite index set and let $H\colon\R^I\to\R$ be a convex 
 function. Let $f\in\LSplus$ such that, for every $\epsilon>0$, there exists $\delta, t_0>0$ with
 \[ S(t)(\tau_x f)\leq\tau_x S(t)f+\tfrac{\epsilon t}{\kappa}
    \quad\mbox{for all } t\in[0,t_0] \mbox{ and } x\in B_{\Rd}(\delta). \]
 For every $n\in\N$, we define $f_n:=f*\eta_n$ and assume that
 \begin{enumerate}
  \item $D^\alpha f$ is a regular distribution for all $\alpha\in I$, 
  \item $H((D^\alpha f)_{\alpha\in I})$ is locally integrable, 
  \item $f_n\in D(\AG)$ and $\AG f_n=H((D^\alpha f_n)_{n\in\N})$.
 \end{enumerate}
 Furthermore, we define the functions
 \begin{align*}
  g(x) &:=H\big((D^\alpha f(x))_{\alpha\in I}\big) \quad\mbox{for all } x\in\Rd,  \\
  \tilde{g}(x) &:=\lim_{r\downarrow 0}\fint_{B(x,r)}g(y)\,\d y \quad\mbox{for all } x\in X_g, \\
  \overline{g}(x) &:=\limsup_{y\in X_g,y\to x}\widetilde g(y) \quad\mbox{for all } x\in\Rd.
 \end{align*}
 Then, it holds $f\in D(\AG)$ and $(\AG f)(x)=\overline{g}(x)$ for all $x\in\Rd$. 
\end{theorem}
\begin{proof}
 It follows from Theorem~\ref{thm:conv} that $f\in D(\AG)$ and $\AG f=\Glim_{n\to\infty}Af_n$. 
 First, we show that $\overline{g}\leq\AG f$. Define $F:=(D^\alpha f)_{\alpha\in I}$ and 
 $F_n:=(D^\alpha f_n)_{\alpha\in I}$ for all $n\in\N$. Since $F_n=F*\eta_n$ for all $n\in\N$, 
 condition~(i) implies $F_n\to F$ almost everywhere. We use the continuity of $H$ and 
 condition~(iii) to obtain
 \[ g=H(F)=\lim_{n\to\infty}H(F_n)=\lim_{n\to\infty}\AG f_n 
 	\leq\Glimsup_{n\to\infty}\AG f_n=\AG f \]
 almost everywhere and therefore
 \[ \fint_{B(x,r)} g(y)\,\d y\leq\fint_{B(x,r)}(\AG f)(y)\,\d y
 	\quad\mbox{for all } x\in\Rd \mbox{ and } r>0. \]
 For every $x\in X_g$, it follows from the upper semicontinuity of $\AG f$
 that
 \[ \tilde{g}(x)=\lim_{r\downarrow 0}\fint_{B(x,r)} g(y)\,\d y
 	\leq\limsup_{r\downarrow 0}\fint_{B(x,r)}(\AG f)(y)\,\d y
 	\leq (\AG f)(x). \]
 By condition~(ii), the set $X_g\subset\Rd$ is dense and the upper semicontinuity of $\AG f$ yields
 \[ \overline{g}(x)=\limsup_{y\in X_g,y\to x}\tilde{g}(y)
 	\leq\limsup_{y\in X_g,y\to x}(\AG f)(y) \leq (\AG f)(x)
 	\quad\mbox{for all } x\in\Rd. \]
 
 Second, we show that $\AG f\leq\overline{g}$. Let $x\in\Rd$ and $\epsilon>0$. Since 
 $\AG f=\Glim_{n\to\infty}\AG f_n$, there exists a sequence $(x_n)_{n\in\N}\subset\Rd$ 
 with $x_n\to x$ and $(\AG f_n)(x_n)\to (\AG f)(x)$. Moreover, there exists $\delta>0$ 
 with $\overline{g}(y)<\overline{g}(x)+\epsilon$ for all $y\in B(x,\delta)$ because 
 $\overline{g}$ is upper semicontinuous. Choose $n_0\in\N$ with $B(x_n,\nicefrac{1}{n})\subset B(x,\delta)$ 
 for all $n\geq n_0$. Since condition~(ii) implies $H(F)=g=\tilde{g}\leq\overline{g}$,
 for every $n\geq n_0$, it follows from Jensen's inequality that 
 \begin{align*}
  \big(\AG f_n\big)(x_n)=H\big(F_n(x_n)\big)
  &=H\left(\int_{B(\nicefrac{1}{n})} F(x_n-y)\eta_n(y)\,\d y\right) \\
  &\leq\int_{B(\nicefrac{1}{n})}H\big(F(x_n-y)\big)\eta_n(y)\,\d y \\
  &\leq\int_{B(\nicefrac{1}{n})}\overline{g}(x_n-y)\eta_n(y)\,\d y 
  \leq\overline{g}(x)+\epsilon.
 \end{align*}
 We obtain $(\AG f)(x)=\lim_{n\to\infty}\AG f_n(x_n)\leq\overline{g}(x)$ 
 for all $x\in\Rd$.
\end{proof}

\section{Chernoff-type approximation schemes}
\label{sec:cher}

Let $(I(t))_{t\geq 0}$ be a family of operators $I(t)\colon\Ck\to\Ck$ and $(h_n)_{n\in\N}\subset (0,\infty)$ 
be a sequence with $h_n\to 0$. For every $t\geq 0$, $f\in\Ck$ and $n\in\N$, we define 
\[ I(\pi^t_n)f:=I(h_n)^{k_n^t}f=\underbrace{\big(I(h_n)\circ\ldots\circ I(h_n)\big)}_{k_n^t \text{ times}}f, \]
where $k_n^t:=\max\{k\in\N_0\colon kh_n\leq t\}$ and $\pi^t_n:=\{0,h_n,\ldots, k_n^t h_n\}$.
Recall that the Lipschitz set $\LI$ consists of all $f\in\Ck$ such that there exist $c\geq 0$ and $t_0>0$ with
\[ \|I(t)f-f\|_\kappa\leq ct \quad\mbox{for all } t\in [0,t_0].\]
Moreover, for every $f\in\Ck$ such that the following limit exists, we define 
 \[ I'(0)f:=\lim_{h\downarrow 0}\frac{I(h)f-f}{h}\in\Ck. \]
By definition of the mixed topology, the existence of $I'(0)f$ implies $f\in\LI$.

\begin{assumption} \label{ass:I}
 Suppose that the following conditions are valid:
 \begin{enumerate}
  \item $I(0)=\id_{\Ck}$.
  \item $I(h_n)$ is convex and monotone with $I(h_n)0=0$ for all $n\in\N$.
  \item There exists a function $\alpha\colon\R_+\times\R_+\to\R_+$, which is 
   non-decreasing in the second argument, such that, for every $r\geq 0$
   and $k,l,n\in\N$,
   \[ I(h_n)\colon B_{\Ck}(r)\to B_{\Ck}(\alpha(r,h_n)) \quad\mbox{and}\quad
   	\alpha(\alpha(r,kh_n),lh_n)\leq\alpha(r,(k+l)h_n). \]
  \item For every $r\geq 0$, there exists $\omega_r\geq 0$ with
   \[ \|I(h_n)f-I(h_n)g\|_\kappa\leq e^{h_n\omega_r}\|f-g\|_\kappa
   	\quad\mbox{for all } n\in\N \mbox{ and } f,g\in B_{\Ck}(r). \]
   Moreover, the mapping $r\mapsto\omega_r$ is non-decreasing.
  \item There exists a countable set $\D\subset\LI$ such that $(I(\pi_n^t)f)_{n\in\N}$ is uniformly 
   equicontinuous for all $(f,t)\in\D\times\T$ and, for every $f\in\Ck$, there exists $(f_n)_{n\in\N}\subset\D$ 
   with $\sup_{n\in\N}\|f_n\|_\kappa\leq\|f\|_\kappa$ and $f_n\to f$.
  \item For every $\epsilon>0$, $r,T\geq 0$ and $K\Subset X$, there exist
   $K'\Subset X$ and $c\geq 0$ with
   \[ \|I(\pi_n^t)f-I(\pi_n^t)g\|_{\infty, K}\leq c\|f-g\|_{\infty, K'}+\epsilon \]
   for all $t\in [0,T]$, $n\in\N$ and $f,g\in B_{\Ck}(r)$.
 \end{enumerate}
\end{assumption}

The conditions~(i)-(iv) only involve the one-step operators $I(h_n)$ and are then
transferred to the iterated operators $I(\pi_n^t)$. Furthermore, in many applications, 
condition~(v) can be verified as described in the following remark.

\begin{remark} \label{rem:dense}
 Let $X:=\Rd$ and assume that $\Cci\subset\LI$, where $\Cci$ denotes the space of all infinitely 
 differentiable functions $f\colon\Rd\to\R$ with compact support.\ Subsequently, we construct a 
 countable set $\D\subset\Cci$ such that, for every $f\in\Ck$, there exists $(f_n)_{n\in\N}\subset\D$ 
   with $\sup_{n\in\N}\|f_n\|_\kappa\leq\|f\|_\kappa$ and $f_n\to f$. Let $n\in\N$ and denote 
 by ${\rm C}(B(n))$ the space of all continuous functions $f\colon B(n)\to\R$, where $B(n):=B_{\Rd}(n)$. 
 Due to the Stone--Weierstra\ss{} theorem, the set $\D_n\subset {\rm C}(B(n))$ of all polynomials 
 $f\colon B(n)\to\R$ with rational coefficients is dense w.r.t. the norm $\|f\|_{\kappa,B(n)}:=\sup_{x\in B(n)}|f(x)|\kappa(x)$. Let 
 $\zeta\in\Cci$ with $0\leq\zeta\leq 1$, $\zeta\equiv 1$ on $B(1)$ and $\zeta\equiv 0$ on 
 $B(2)^c$. Define $\zeta_n(x):=\zeta(\frac{x}{n})$ for all $n\in\N$ and $x\in\Rd$. Furthermore, 
 by setting $(f\zeta_n)(x):=0$ for all $x\in B(2n)^c$, we define  
 \[ \D:=\bigcup_{n\in\N}\{f\zeta_n\colon f\in\D_{2n}\}\subset\Cci. \]
 Let $f\in\Ck$. For every $n\in\N$, there exists $\tilde{f}_n\in\D_{2n}$ with 
 $\|f|_{B(2n)}-\tilde{f}_n\|_{\kappa,B(2n)}\leq\frac{1}{n}$ and
 $\|\tilde{f}_n\|_{\kappa,B(2n)}\leq\|f|_{B(2n)}\|_{\kappa,B(2n)}\leq\|f\|_{\kappa,\Rd}$. 
 Hence, the sequence $(f_n)_{n\in\N}\subset\D$ given by $f_n:=\tilde{f}_n\zeta_n$ satisfies 
$\|f_n\|_{\kappa,\Rd}\leq\|f\|_{\kappa,\Rd}$ and $f_n\to f$. 
 Hence, Assumption~\ref{ass:I}(v) is, for example, satisfied if there exists $c\geq 0$ with $I(t)\colon\Lipb(r)\to\Lipb(e^{ct}r)$ 
 for all $r,t\geq 0$, where $\Lipb(r)$ contains all $r$-Lipschitz functions $f\colon\Rd\to\R$ 
 with $\|f\|_\infty\leq r$.
\end{remark}

It remains to discuss how condition~(vi) can be verified. For the examples presented
in Section~\ref{sec:examples}, we rely on the fact that, due to Lemma~\ref{lem:cont.app}, 
condition~(vi) is equivalent to 
\[ \sup_{(t,x)\in [0,T]\times K}\sup_{n\in\N}\big(I(\pi_n^t)f_k\big)(x)\downarrow 0 
	\quad\mbox{as } k\to\infty \]
for all $T\geq 0$, $K\Subset X$ and $(f_k)_{k\in\N}\subset\Ck$ with $f_k\downarrow 0$.
There also exist several verifiable sufficient conditions on the one-step operators $I(t)$,
one of which is presented in the next remark. In applications, this condition can be satisfied
by requiring finiteness of certain moments which appears very naturally, see, e.g.,~\cite{BK22b}.
A detailed discussion of several further sufficient conditions can be found in~\cite[Subsection~2.5]{BKN23}.

\begin{remark} \label{rem:I}
 Let $\tilde{\kappa}\colon X\to (0,\infty)$ be a bounded continuous function such that, 
 for every $\epsilon>0$, there exists $K\Subset X$ with $\sup_{x\in K^c}\frac{\tilde{\kappa}(x)}{\kappa(x)}\leq\epsilon$. 
 Let $\tilde{\alpha}\colon\R_+\times\R_+\to\R_+$ be a function, which is non-decreasing in 
 the second argument, such that
 \begin{equation} \label{eq:tilde}
  \|I(h_n)f\|_{\tilde{\kappa}}\leq\tilde{\alpha}(r,h_n) \quad\mbox{and}\quad
  \tilde{\alpha}(\tilde{\alpha}(r,kh_n),lh_n)\leq\tilde{\alpha}(r,(k+l)h_n) 
 \end{equation}
 for all $r\geq 0$, $k,l,n\in\N$ and $f\in\Ck$ with $\|f\|_{\tilde{\kappa}}\leq r$. Then, for every 
 $(f_n)_{n\in\N}\subset\Ck$ with $f_n\downarrow 0$, Dini's theorem implies $\|f_n\|_{\tilde{\kappa}}\to 0$. 
 It follows from Assumption~\ref{ass:I}(ii), equation~\eqref{eq:tilde},~\cite[Lemma~2.7]{BK23} 
 and Lemma~\ref{lem:lip}(ii) that
   \[ \sup_{t\in [0,T]}\sup_{n\in\N}\|I(\pi_n^t)f_k\|_{\tilde{\kappa}}
   	\leq\tilde{\alpha}(3\|f_1\|_{\tilde{\kappa}},t)\|f_k\|_{\tilde{\kappa}}\to 0
 	\quad\mbox{as } k\to\infty. \]
 Since $\inf_{x\in K}\kappa(x)>0$ for every $K\Subset X$, Assumption~\ref{ass:I}(vi) is satisfied.
\end{remark}

The following theorem is based on the work in~\cite{BK23}, where Chernoff-type approximation
schemes for nonlinear semigroups have been systematically investigated.\ The results in~\cite{BK23} 
require relative compactness of the sequence $(I(\pi_n^t)f)_{n\in\N}$ w.r.t.\ a given metric (here the one induced by $\|\cdot\|_\kappa$). A close 
inspection of the proofs given in~\cite{BK23} reveals that relative compactness of the sequence 
$(I(\pi_n^t)f)_{n\in\N}$ w.r.t.\ the mixed topology is sufficient as long as Assumption~\ref{ass:I}(vi) 
can be verified. This condition also guarantees that the corresponding semigroup $(S(t))_{t\geq 0}$ 
is continuous w.r.t. the mixed topology. For the reader's convenience, we provide a detailed proof 
in Appendix~\ref{app:I}, where we thoroughly explain how the arguments from~\cite{BK23} have to be 
modified. Subsequently, we define
\[ \|f\|_{\infty,Y}:=\sup_{x\in Y}|f(x)|
	\quad\mbox{for all } f\colon X\to\R \mbox{ and } Y\subset X. \]

\begin{theorem} \label{thm:I}
 Suppose that Assumption~\ref{ass:I} is satisfied and let $\T\subset\R_+$ be a countable 
 dense set including zero.\ Then, there exist a strongly continuous convex monotone semigroup 
 $(S(t))_{t\geq 0}$ on $\Ck$ with $S(t)0=0$ and a subsequence $(n_l)_{l\in\N}\subset\N$ such that
 \begin{equation} \label{eq:cher2.IS}
  S(t)f=\lim_{l\to\infty}I(\pi_{n_l}^t)f \quad\mbox{for all } (f,t)\in\Ck \times\T.
 \end{equation}
 Furthermore, the following statements are valid:
 \begin{enumerate} 
  \item It holds $f\in D(A)$ and $Af=I'(0)f$ for all $f\in\Ck$ such that $I'(0)f\in\Ck$ exists. 
  \item For every $r,t\geq 0$ and $f,g\in B_{\Ck}(r)$, 
   \[ \|S(t)f\|_\kappa\leq\alpha(r,t) \quad\mbox{and}\quad 
   	\|S(t)f-S(t)g\|_\kappa\leq e^{t\omega_{\alpha(r,t)}}\|f-g\|_\kappa. \]
  \item For every $\epsilon>0$, $r,T\geq 0$ and $K\Subset X$, there exist $K'\Subset X$ 
   and $c\geq 0$ with
   \[ \|S(t)f-S(t)g\|_{\infty,K}\leq c\|f-g\|_{\infty,K'}+\epsilon \]
   for all $t\in [0,T]$ and $f,g\in B_{\Ck}(r)$.
  \item It holds $\LI\subset\LS$ and $S(t)\colon\LS\to\LS$ for all $t\geq 0$. 
 \end{enumerate}
\end{theorem}

The previous theorem provides a general approach for the construction of strongly continuous
convex monotone semigroups. In principle, the limit in equation~\eqref{eq:cher2.IS} could depend on the 
choice of the convergent subsequence and convergence of the whole sequence might fail.\ The latter
is crucial in order to understand the previous result as an approximation scheme and as a possibility to determine the continuous-time limit of 
an appropriately scaled discrete-time dynamics. Furthermore, in view of equation~\eqref{eq:cher2.IS}, 
the semigroup $(S(t))_{t\geq 0}$ should only depend on the infinitesimal behaviour of the family 
$(I(t))_{t\geq 0}$ close to zero. 

Classical existence results for strongly continuous linear semigroups are based on the resolvent of $A$ and the semigroup is defined on $\overline{D(A)}$
which is supposed to coincide with $\Ck$. Here, for the existence result, we do not a priori 
assume that $I'(0)$ and thus $A$ are densely defined but in order to apply the comparison
principle from Section~\ref{sec:comp} it is crucial determine $Af$ for sufficiently many functions. In many applications, the limit $I'(0)f$ can be computed explicitly for sufficiently 
smooth functions. Then, Theorem~\ref{thm:I} guarantees that $Af=I'(0)f$ and that $(S(t))_{t\geq 0}$ is a strongly continuous convex monotone semigroup.

Since the set of smooth functions is not invariant under the 
semigroup and the upper $\Gamma$-generator cannot be computed explicitly, we rely on the 
approximation results in Section~\ref{sec:approx}, which heavily depend on condition~\eqref{eq:trans}. 
In order to characterize functions satisfying this condition, we introduce the so-called approximation 
set $\AS$ which is invariant under the semigroup $(S(t))_{t\geq 0}$. Furthermore, the iterative 
approximation set $\AIiter$ contains functions for which this condition can be verified by means 
the operators $(I(t))_{t\geq 0}$ and transferred to $(S(t))_{t\geq 0}$.
In the sequel, let $X$ be a Banach space and assume that
\begin{equation} \label{eq:kappa2}
 c_\kappa:=\sup_{x\in X}\sup_{y\in B_X(1)}\frac{\kappa(x)}{\kappa(x-y)}\leq 1.
\end{equation}
Recall that the shift operators are given by $(\tau_x f)(y):=f(x+y)$ for all functions 
$f\colon X\to\R$ and $x,y\in X$.

\begin{definition} \label{def:Aset}
 Let $(I(t))_{t\geq 0}$ be a family of operators $I(t)\colon\Ck\to\Ck$. The approximation 
 set $\AI$ consists of all $f\in\Ck$ such that, for every $\epsilon>0$ and $t\geq 0$, 
 there exist $\delta, s_0>0$ with 
 \[ I(s)\tau_x I(t)f\leq\tau_x I(s)I(t)f+\frac{\epsilon s}{\kappa} \]
 for all $s\in [0,s_0]$ and $x\in B_X(\delta)$. Furthermore, the iterative approximation
 set $\AIiter$ consists of all $f\in\Ck$ such that, for every $\epsilon>0$, there exist 
 $c\geq 0$ and $\delta, h_0>0$ with
 \[ I(h)\tau_x I(h)^l f\leq\tau_x I(h)^{l+1}f+\frac{e^{clh}\epsilon h}{\kappa} \]
 for all $h\in [0,h_0]$, $l\in\N_0$ and $x\in B_X(\delta)$. 
\end{definition}

\begin{theorem} \label{thm:approx}
 Suppose that Assumption~\ref{ass:I} is satisfied and denote by $(S(t))_{t\geq 0}$ the 
 corresponding semigroup from Theorem~\ref{thm:I}. In addition, we assume that there exists 
 $\omega\geq 0$ with $\|I(h_n)f\|_\kappa\leq e^{\omega h_n}\|f\|_\kappa$ for all $f\in\Ck$ 
 and $n\in\N$. Then, 
 \[ \AIiter\subset\ASiter\subset\AS \quad\mbox{and}\quad 
    S(t)\colon\AS\to\AS \quad\mbox{for all } t\geq 0. \]
\end{theorem}
\begin{proof}
 Let $f\in\AIiter$ and $\epsilon>0$. Then, there exist $c\geq 0$ and $\delta, h_0\in (0,1]$ with
 \begin{equation} \label{eq:approx1}
  I(h)\tau_x I(h)^l f\leq\tau_x I(h)^{l+1}f+\frac{e^{clh}\epsilon h}{\kappa}
 \end{equation}
 for all $h\in [0,h_0]$, $l\in\N_0$ and $x\in B_X(\delta)$. By induction, we show that
 \begin{equation} \label{eq:approx2}
  I(h)^k\tau_x I(h)^l f\leq\tau_x I(h)^{k+l}f+\frac{e^{(c\vee\omega)(k+l)h}\epsilon kh}{\kappa}
 \end{equation}
 for all $h\in [0,h_0]$, $k,l\in\N_0$ and $x\in B_X(\delta)$. For $k=0,1$, this follows 
 from inequality~\eqref{eq:approx1}. For the induction step, we use Lemma~\ref{lem:lambda} 
 and Corollary~\ref{cor:kappa} to obtain
 \begin{align*}
  &I(h)^{k+1}\tau_x I(h)^l f-\tau_x I(h)^{k+l+1}f \\
  &=\big(I(h)^k I(h)\tau_x I(h)^l f-I(h)^l\tau_x I(h)^{l+1}f\big)
    +\big(I(h)^k\tau_x I(h)^{l+1}f-\tau_x I(h)^{k+l+1}f\big) \\
  &\leq\left(I(h)^k\left(\frac{I(h)\tau_x I(h)^l f-\tau_x I(h)^{l+1}f}{h}+\tau_x I(h)^{l+1}f\right)
    -I(h)^k\tau_x I(h)^{l+1}f\right)h \\
  &\quad\; +e^{(c\vee\omega)(k+l+1)}\epsilon kh \\
  &\leq\left(I(h)^k\left(\tau_x I(h)^{l+1}f+\frac{e^{clh}\epsilon}{\kappa}\right)
    -I(h)^k\tau_x I(h)^{l+1}f\right)h+e^{(c\vee\omega)(k+l+1)}\epsilon kh \\
 &\leq \frac{e^{(c\vee\omega)(k+l+1)}\epsilon (k+1)h}{\kappa}.
 \end{align*}
 Now, let $s,t\in\T$. Assumption~\ref{ass:I}(vi) and equation~\eqref{eq:cher2.IS} guarantee
 \[ I(\pi_n^s)\tau_x I(\pi_n^t)f\to S(s)\tau_x S(t)f \quad\mbox{and}\quad
    \tau_x I(\pi_n^{s+t})f\to\tau_x S(s+t)f=\tau_x S(s)S(t)f \]
 for all $x\in B_X(\delta)$ and therefore inequality~\eqref{eq:approx2} implies
 \begin{equation} \label{eq:approx3}
  S(s)\tau_x S(t)f\leq\tau_x S(s)S(t)f+\frac{e^{(c\vee\omega)(s+t)}\epsilon s}{\kappa}.
 \end{equation}
 This inequality remains valid for arbitrary $s,t\geq 0$ and $x\in B_X(\delta)$ due to 
 the strong continuity of $(S(t))_{t\geq 0}$ and Theorem~\ref{thm:I}(iii). We obtain $f\in\ASiter$
 and similarly one can show that $\ASiter\subset\AS$. 

 Let $f\in\AS$ and $s,t\geq 0$. Then, for every $\epsilon>0$, there exist $\delta, r_0>0$ with 
 \[ S(r)\tau_x S(s)S(t)f=S(r)\tau_x S(s+t)
    \leq\tau_x S(r)S(s+t)f+\frac{\epsilon r}{\kappa}
    =\tau_x S(r)S(s)S(t)+\frac{\epsilon r}{\kappa} \]
 for all $r\in [0,r_0]$ and $x\in B_X(\delta)$. We obtain $S(t)f\in\AS$.
\end{proof}

We now apply the previous result to the particular case that the approximation set contains bounded 
Lipschitz continuous functions.\ This covers all examples presented in Section~\ref{sec:examples}, 
where we can explicitly compute the generator $Af$ for smooth functions.\ 
Note that Assumption~\ref{ass:I2}(ii)
only involves the one-step operators $I(h)$ rather than the iterated operators $I(h)^l$
appearing in Definition~\ref{def:Aset}.\
In the sequel, let $X:=\Rd$ and let $\eta\colon\Rd\to\R_+$ be an infinitely differentiable 
function with $\supp(\eta)\subset B_{\Rd}(1)$ and $\int_{\Rd}\eta(x)\,\d x=1$. For every 
$n\in\N$, $x\in\Rd$ and locally integrable function $f\colon\Rd\to\R$, we define 
$\eta_n(x):=n^d\eta(nx)$ and
\[ (f*\eta_n)(x):=\int_{\Rd} f(x-y)\eta_n(y)\,\d y.\]
Let $\Lipb(r)$ be the space of all $r$-Lipschitz functions $f\colon\Rd\to\R$ with $\|f\|_\infty\leq r$
and define $\Lipb:=\bigcup_{r\geq 0}\Lipb(r)$.
The following Assumption~\ref{ass:I2} is stronger than Assumption~\ref{ass:I} and leads to a refinement of Theorem \ref{thm:I}.

\begin{assumption} \label{ass:I2}
 Suppose that the following conditions are valid:
 \begin{enumerate}
  \item $I(0)=\id_{\Ck}$.
  \item $I(h_n)$ is convex and monotone with $I(h_n)0=0$ for all $n\in\N$.
  \item There exists $\omega\geq 0$ with $\|I(h_n)f\|_\kappa\leq e^{\omega h_n}\|f\|_\kappa$
   for all $f\in\Ck$ and $n\in\N$. 
  \item For every $r\geq 0$, there exists $\omega_r\geq 0$ with
   \[ \|I(h_n)f-I(h_n)g\|_\kappa\leq e^{h_n\omega_r}\|f-g\|_\kappa
   	\quad\mbox{for all } n\in\N \mbox{ and } f,g\in B_{\Ck}(r). \]
   Moreover, the mapping $r\mapsto\omega_r$ is non-decreasing.
  \item For every $\epsilon>0$, there exist $\delta>0$ and $n_0\in\N$ with
   \[ I(h_n)(\tau_x f)\leq\tau_x I(h_n)f+\frac{r\epsilon h_n}{\kappa} \]
   for all $r\geq 0$, $f\in\Lipb(r)$, $x\in B_{\Rd}(\delta)$ and $n\geq n_0$. 
  \item The limit $I'(0)f\in\Ck$ exists for all $f\in\Cbi$. 
  \item For every $\epsilon>0$, $r,T\geq 0$ and $K\Subset X$, there exist
   $K'\Subset X$ and $c\geq 0$ with
   \[ \|I(\pi_n^t)f-I(\pi_n^t)g\|_{\infty, K}\leq c\|f-g\|_{\infty, K'}+\epsilon \]
   for all $t\in [0,T]$, $n\in\N$ and $f,g\in B_{\Ck}(r)$.
  \item There exists $c\geq 0$ with $I(h_n)\colon\Lipb(r)\to\Lipb(e^{ch_n}r)$ for all $r\geq 0$
   and $n\in\N$. 
 \end{enumerate}
\end{assumption}

\begin{remark}
Observe that Assumption~\ref{ass:I2}(iii) implies Assumption~\ref{ass:I}(iii) by choosing $\alpha(r,t) := r e^{\omega t}$. Moreover, Assumption~\ref{ass:I}(v) follows from Assumption~\ref{ass:I2}(vi) in conjunction with Remark~\ref{rem:dense}. In particular, Assumption~\ref{ass:I2} implies Assumption~\ref{ass:I}. Finally, Assumptions~\ref{ass:I2}(iii) and (iv) are guaranteed by the global Lipschitz condition: there exists \(\omega \geq 0\) with $\|I(h_n)f-I(h_n)g\|_\kappa\leq e^{\omega h_n}\|f-g\|_\kappa$ for all $f,g\in \Ck$ and $n\in\N$.
\end{remark}

\begin{theorem} \label{thm:I2}
 Suppose that Assumption~\ref{ass:I2} is satisfied and denote by 
 $(S(t))_{t\geq 0}$ the corresponding semigroup from Theorem~\ref{thm:I}. Then, the following 
 statements are valid:
 \begin{enumerate}
  \item It holds $S(t)\colon\Lipb(r)\to\Lipb(e^{ct}r)$ for all $r,t\geq 0$.
  \item For every $\epsilon>0$, there exists $\delta>0$ with
   \[ S(t)(\tau_x f)\leq\tau_x S(t)f+\frac{e^{(c\vee\omega)t}r\epsilon t}{\kappa} \]
   for all $t\geq 0$, $r\geq 0$, $f\in\Lipb(r)$ and $x\in B_{\Rd}(\delta)$.  
  \item It holds $\Cbi\subset D(A)$ with $Af=I'(0)f$ for all $f\in\Cbi$. 
  \item It holds $\LSplus\cap\Lipb\subset D(\AG)$ and, for every $f\in\LSplus\cap\Lipb$,
   \[ \AG f=\Glim_{n\to\infty}Af_n \quad\mbox{with}\quad f_n:=f*\eta_n. \]
 \end{enumerate}
\end{theorem}
\begin{proof}
 Property~(i) follows from Assumption~\ref{ass:I2}(viii), equation~\eqref{eq:cher2.IS}
 and the strong continuity of $(S(t))_{t\geq 0}$. For every $\epsilon>0$, due to 
 Assumption~\ref{ass:I2}(v) and~(viii), there exist $\delta>0$ and $n_0\in\N$ with
 \[ I(h_n)\tau_x I(h_n)^l f\leq\tau_x I(h_n)^{l+1}f+\frac{e^{clh_n}r\epsilon h_n}{\kappa} \]
 for all $n\geq n_0$, $l\in\N$, $r\geq 0$, $f\in\Lipb(r)$ and $x\in B_{\Rd}(\delta)$. 
 Assumption~\ref{ass:I2}(iii) and inequality~\eqref{eq:approx3} yield that property~(ii) 
 is satisfied. Property~(iii) follows from Assumption~\ref{ass:I2}(vi) and Theorem~\ref{thm:I}(i).
 In order to verify property~(iv), let $f\in\LSplus\cap\Lipb$. It holds $f_n:=f*\eta_n\in\Cbi\subset D(A)$ 
 for all $n\in\N$ and thus Theorem~\ref{thm:conv} implies $f\in D(\AG)$ with $\AG f=\Glim_{n\to\infty}Af_n$.
\end{proof}

As a consequence of Theorem \ref{thm:conv3} and Theorem \ref{thm:I2}, we get the following result.

\begin{theorem} \label{thm:I3}
 Let $(I(t))_{t\geq 0}$ and $(J(t))_{t\geq 0}$ be two families of operators satisfying Assumption~\ref{ass:I2} with
 \begin{equation} \label{eq:thm I3}
  I'(0)f\leq J'(0)f \quad\mbox{for all } f\in\Cbi.
 \end{equation}
 Then, it holds $S(t)f\leq T(t)f$ for all $t\geq 0$ and $f\in\Ck$, where $(S(t))_{t\geq 0}$
 and $(T(t))_{t\geq 0}$ are the semigroups from Theorem~\ref{thm:I} corresponding to $(I(t))_{t\geq 0}$ and $(J(t))_{t\geq 0}$, respectively.
\end{theorem}
\begin{proof}
 Theorem~\ref{thm:I2}(iii) and 
 inequality~\eqref{eq:thm I3} imply 
 \[ Af=I'(0)f\leq J'(0)f=Bf \quad\mbox{for all } f\in\Cbi, \] 
 and therefore the statement follows from Theorem \ref{thm:conv3} together with Theorem~\ref{thm:I2}(ii).
\end{proof}

In particular, under the stronger Assumption~\ref{ass:I2}, the limit in the Chernoff-type approximation \eqref{eq:cher2.IS} is independent of the subsequence.

\begin{corollary} \label{cor:cher}
 Let $(I(t))_{t\geq 0}$ be a family of operators satisfying 
 Assumption~\ref{ass:I2}.\ Then, the limit in equation~\eqref{eq:cher2.IS} does not depend on 
 the choice of the convergent subsequence and therefore the whole sequence converges.
\end{corollary}

\section{Examples}
\label{sec:examples}

In this section, we apply our theoretical results in illustrative examples with focus 
on the comparison principles in Subsection~\ref{sec:comp} and Section~\ref{sec:approx} 
as well as the Chernoff-type approximation in Section~\ref{sec:cher}. For additional 
applications, we refer to the limit theorems in~\cite{BK22b} and the stability results in~\cite{BKN23}.

\subsection{Stochastic control problems}
\label{sec:control}

We consider a controlled SDE of the form
\begin{equation} \label{eq:control.SDE}
 \begin{cases}
  \d X_t^{x,\alpha}=b(X_t^{x,\alpha},\alpha_t)\,\d t+\sigma(\alpha_t)\,\d W_t, \\
  X_0^{x,\alpha}=x, \end{cases}
\end{equation}
where $(W_t)_{t\geq 0}$ is a $d$-dimensional standard Brownian motion on a complete filtered 
probability space $(\Omega,\mathcal{F},(\mathcal{F}_t)_{t\geq 0},\P)$ satisfying the usual 
conditions, $\A\subset\Rq$ is an action set with $q\in\N$ and $\sigma\colon\A\to\S_+^d$ and 
$b\colon\Rd\times\A\to\Rd$ are measurable. Here, the set~$\S^d_+$ consists of all symmetric 
positive semidefinite $d\times d$-matrices.\ We endow $\Rd$ and $\Rq$ with the 
Euclidean norm and $\R^{d\times d}$ with the Frobenius norm. For every $t\geq 0$, $x\in\Rd$ 
and $f\in\Ck$, where $\kappa:=(1+|\cdot|^p)^{-1}$ for some $p\geq 1$, we consider the value function
\[
v(t,x;f):=\sup_{\alpha\in\Aad}\E\left[f(X_t^{x,\alpha})-\int_0^t L(\alpha_s)\,\d s\right], \]
where $\Aad$ consists of all predictable processes $\alpha\colon[0,\infty)\times\Omega\to\A$ with
\[ \E\left[\int_0^t |\alpha_s|\,\d s\right]<\infty \quad\mbox{for all } t\geq 0 \]
and the running cost $L\colon\A\to [0,\infty]$ is measurable. The following conditions
guarantee that equation~\eqref{eq:control.SDE} is well-posed.

\begin{assumption} \label{ass:control} \Newline
 \begin{enumerate}
  \item There exists $C\geq 0$ with $|\sigma(\alpha)|\leq C$ for all $\alpha\in\A$. Furthermore,
   \[ |b(x,\alpha)|\leq C(1+|x|) \quad\mbox{and}\quad |b(x,\alpha)-b(y,\alpha)|\leq C|x-y| \]
   for all $\alpha\in\A$ and $x,y\in\Rd$.
  \item There exists $\alpha^*\in\A$ with $L(\alpha^*)=0$. 
 \end{enumerate}
\end{assumption}

For every $t\geq 0$, $x\in\Rd$ and $f\in\Ck$, we define
\begin{equation}\label{eq:valuefunction}
(S(t)f)(x):=v(t,x;f)=\sup_{\alpha\in\Aad}\E\left[f(X_t^{x,\alpha})-\int_0^t L(\alpha_s)\,\d s\right].
\end{equation}
A standard method in optimal control is to show that the value function $v(\cdot,\cdot;f)$ 
satisfies the dynamic programming principle (DPP) for certain terminal conditions $f$. From this, 
the Hamilton-Jacobi-Bellman (HJB) equation is obtained and linked to the control problem by using 
a verification argument.\ The existence and uniqueness of the solution to the HJB equation typically 
relies on viscosity theory. In Theorem~\ref{thm:control.main} below, we show that the control problem 
can also directly be identified with a strongly continuous convex monotone semigroup $T((t))_{t\ge 0}$.\
Hence, the previously developed semigroup theory is applicable and the associated semigroup is uniquely 
determined by its generator evaluated on $\Cbi$ due to Theorem~\ref{thm:conv3}. In particular, the 
control problem can be described using the Chernoff-type approximations studied in Section~\ref{sec:cher}
which allow for approximations with piecewise constant controls and discretizations of the noise.\
We further note that Chernoff-type approximations lead to numerical schemes for stochastic control
problems whose convergence rates are given in~\cite{BJKL23}.\ In case of a sublinear value function, 
the rates obtained in~\cite{BJKL23} are consistent with the rates which have previously been obtained 
in~\cite{BJ05,BJ07,Krylov98,Krylov99} by relying on monotone schemes for viscosity solutions.
However, in the convex case, the results in~\cite{BJKL23} are apparently new. We now consider the 
following discretized version of the previous control problem.
Let $\xi\colon\Omega\to\Rd$ be a random variable with $\E[\xi]=0$, $\E[\xi\xi^T]=\one_d$ and 
$\E[|\xi|^3]<\infty$. For every $t\geq 0$, $f\in\Ck$ and $x\in\Rd$, we define $\xi_t:=\sqrt{t}\xi$ and
\begin{equation}\label{eq:one-step}
(I(t)f)(x):=\sup_{\alpha\in\A}\big(\E[f(x+b(x,\alpha)t+\sigma(\alpha)\xi_t)]-L(\alpha)t\big). 
\end{equation}
Let $(h_n)_{n\in\N}\subset (0,1]$ with $h_n\to 0$. For every $n\in\N$, $t\geq 0$ and $f\in\Ck$, 
we define 
\[ I(\pi^t_n)f:=I(h_n)^{k_n^t}f
    =\underbrace{\big(I(h_n)\circ\ldots\circ I(h_n)\big)}_{k_n^t \text{ times}}f, \]
where $k_n^t:=\max\{k\in\N_0\colon kh_n\leq t\}$ and $\pi^t_n:=\{0,h_n,\ldots, k_n^t h_n\}$. 
The following result illustrates the convergence of the discretized control problems 
$I(\pi_n^t)f$ to the original control problem $T(t)f$. Moreover, if $\xi$ is normally distributed, 
then $I(\pi^t_n)f$ corresponds to the control problem~\eqref{eq:valuefunction} with piecewise 
constant controls which are obtained by iteratively solving the one-step optimization problems
$f_{k h_n}:=I(h_n)f_{(k+1) h_n}$ of the form \eqref{eq:one-step} with $f_{k^t_n h_n}:=f$. 
By backward recursion, we obtain $I(\pi^t_n)f=f_0$.

\begin{theorem} \label{thm:control.main}
 Suppose that Assumption~\ref{ass:control} is satisfied. Then, $(S(t))_{t\geq 0}$ is a strongly 
 continuous convex monotone semigroup on $\Ck$. It holds $\Cbi\subset D(A)$ and 
 \[ (Af)(x)=\sup_{\alpha\in\A}
    \left(\frac{1}{2}\tr\big(\sigma(\alpha)^2 D^2f(x)\big)+b(x,\alpha)^T Df(x)-L(\alpha)\right) \]
 for all $f\in\Cbi$ and $x\in\Rd$, where $A$ denotes the generator of $(S(t))_{t\geq 0}$. Furthermore,
 \[ S(t)f=\lim_{n\to\infty}I(\pi_n^t)f \quad\mbox{for all } t\geq 0 \mbox{ and } f\in\Ck. \]
\end{theorem}
\begin{proof}
 We show that $(I(t))_{t\geq0}$ and $(S(t))_{t\geq 0}$ satisfy Assumption~\ref{ass:I2} with the 
 same infinitesimal behaviour on $\Cbi$. Then, the statement follows from Theorem~\ref{thm:I3},
 since $(S(t))_{t\geq 0}$ satisfies the dynamic programming principle $S(s+t)f=S(s)S(t)f$ for all 
 $s,t\geq 0$ and $f\in\Lipb$, see, e.g., Fabbri et al.~\cite[Theorem~2.24]{FGS17} or 
 Pham~\cite[Theorem~3.3.1]{Pham09}, which extends to arbitrary $f\in\Ck$ due to Lemma~\ref{lem:uniform}. 
 For every $t\in [0,1]$, $x\in\Rd$ and $\alpha\in\Aad$, Assumption~\ref{ass:control}(i) implies
 \[ 1+\E[|X_t^{x,\alpha}|^p]\leq 3^p\left(1+|x|^p+C2^p\int_0^t 1+\E[|X_s^{x,\alpha}|^p]\,\d s+C_p tC^p\right), \]
 where $C_p\geq0$ is the constant from the Burkholder--Davis--Gundy inequality.
 Hence, by Gronwall's lemma, there exists $c_p\geq0$ with
 \[ 1+\E[|X_t^{x,\alpha}|^p]\leq e^{tc_p}(1+|x|^p).\]
 For every $t\geq 0$, $f\in\Ck$, $x\in\Rd$ and $\alpha\in\Aad$, we obtain
 \[ \E\left[f(X_t^{x,\alpha})-\int_0^t L(\alpha_s)\,\d s\right] \leq\|f\|_\kappa e^{tc_p}(1+|x|^p). \]
 Assumption~\ref{ass:control}(ii) guarantees that $S(t)0=0$ for all $t\geq0$ and thus
 \[ \|S(t)f\|_\kappa\leq e^{tc_p}\|f\|_\kappa \quad\mbox{for all } t\in [0,1] \mbox{ and } f\in \Ck.\]
 Furthermore, the operator $S(t)\colon\Ck\to\Fk$ is convex and monotone. Let 
 $\tilde{\kappa}:=(1+|\cdot|^q)^{-1}$ with $q>p$. For every sequence $(f_n)_{n\in\N}\subset\Ck$ 
 with $f_n\downarrow 0$, Dini's theorem implies
 \[ \|S(t)f_n\|_{\tilde{\kappa}}\leq e^{tc_q}\|f_n\|_{\tilde{\kappa}}\to 0 \quad\mbox{as } n\to\infty \]
 which shows that $S(t)$ is continuous from above. Let $r,t\geq 0$, $f\in\Lipb(r)$ and $x,y\in\Rd$. 
 Assumption~\ref{ass:control}(ii) yields $\|S(t)f\|_\infty\leq\|f\|_\infty$ and
 \[ |X_t^{x,\alpha}-X_t^{y,\alpha}|\leq |x-y|+\int_0^t C|X_s^{x,\alpha}-X_s^{y,\alpha}|\,\d s
    \quad\mbox{for all } \alpha\in \Aad. \]
 Hence, it follows from Gronwall's lemma that
 \[ |(S(t)f)(x)-(S(t)f)(y)|\leq\sup_{\alpha\in\Aad}\E[|f(X_t^{x,\alpha})-f(X_t^{y,\alpha})|]\leq e^{Ct}r|x-y| \]
 showing that $S(t)\colon\Lipb(r)\to\Lipb(e^{Ct}r)$. Since $\Lipb\subset\Ck$ is dense, Lemma~\ref{lem:cont.app}
 implies $S(t)\colon\Ck\to\Ck$ for all $t\geq 0$. In addition, for every $\alpha\in\Aad$,
 \[ |x+X_t^{y,\alpha}-X_t^{x+y,\alpha}|\leq C|x|t+\int_0^t C|x+X_s^{y,\alpha}-X_s^{x+y,\alpha}|\,\d s. \]
 By applying Gronwall's lemma again, we obtain
 \begin{equation} \label{eq:gronwall}
  |x+X_t^{y,\alpha}-X_t^{x+y,\alpha}|\leq Cte^{Ct}
 \end{equation}
 and therefore
 \[ |(S(t)\tau_x f)(y)-(\tau_x S(t)f)(y)|
    \leq\sup_{\alpha\in\Aad}\E[|f(x+X_t^{x,\alpha})-f(X_t^{x+y,\alpha})|] \leq Ce^{Ct}r|x|t. \]
 This shows that Assumption \ref{ass:I2}(viii) is satisfied.
 
 Next, we show that $\Cbi\subset D(A)$ and 
 \[ (Af)(x)=\sup_{\alpha\in\A}
    \left(\frac{1}{2}\tr\big(\sigma(\alpha)^2 D^2f(x)\big)+b(x,\alpha)^T Df(x)-L(\alpha)\right) \]
 for all $f\in\Cbi$ and $x\in\Rd$. Let $f\in\Cbi$. It\^o's formula implies 
 \begin{align*}
  &\E\left[f(X_t^{x,\alpha})-\int_0^t L(\alpha_s)\,\d s\right]-f(x) \\
  &=\E\left[\int_0^t\frac{1}{2}\tr\big(\sigma(\alpha_s)^2 D^2 f(X_s^{x,\alpha})\big)
    +b(X_s^{x,\alpha},\alpha_s)^T Df(X_s^{x,\alpha})-L(\alpha_s)\,\d s\right] \\
 &\leq\E\left[\int_0^t g(X_s^{x,\alpha})\,\d s\right] 
 \end{align*}
 for all $t\geq 0$, $x\in\Rd$ and $\alpha\in\Aad$, where $g\colon\Rd\to\Rd$ is given by 
 \[ g(x):=\sup_{\alpha\A}\Big(\frac{1}{2}\tr\big(\sigma(\alpha)^2 D^2f(x)\big)+b(x,\alpha)^T Df(x)-L(\alpha)\Big)
    \quad\mbox{for all } x\in\Rd. \]
 By Assumption~\ref{ass:control}(i), there exists $r\geq 0$ with 
 \[ |g(x)-g(y)|\leq r|x-y| \quad\mbox{and}\quad |g(x)|\leq r(1+|x|) 
    \quad\mbox{for all } x,y\in\Rd. \] 
 Let $\epsilon>0$. For every $\delta>0$, $t\geq 0$, $x\in\Rd$ and $\alpha\in\Aad$,
 it follows from H\"older's inequality and inequality~\eqref{eq:gronwall} with $y=0$ that 
 \begin{align*}
  \E\left[\int_0^t g(X_s^{x,\alpha})\,\d s\right] 
  &=\E\left[\int_0^t g(X_s^{x,\alpha})\one_{\{|x-X_s^{x,\alpha}|\leq\delta\}}
    +g(X_s^{x,\alpha})\one_{\{|x-X_s^{x,\alpha}|>\delta\}}\,\d s\right] \\
  &\leq (g(x)+r\delta)t+r\E\left[\int_0^t (1+|X_s^{x,\alpha}|)\one_{\{|x-X_s^{x,\alpha}|>\delta\}}\,\d s\right] \\
  &\leq (g(x)+r\delta)t+2r\int_0^t \big(1+\E[|X_s^{x,\alpha}|^p]\big)^\frac{1}{p}\P(|x-X_s^{x,\alpha}|>\delta)^{\frac1q}\,\d s \\
  &\leq (g(x)+r\delta)t+\int_0^t \bigg(\frac{c(1+|x|^p)s^{p}}{\delta^{p}}\bigg)^{\frac1q}
    e^{\frac{sc_p}{p}}\big(1+|x|^p\big)^{\frac1p}\,\d s \\
  &\leq (g(x)+r\delta)t+\bigg(1+\frac{p}q\bigg)^{-1}\frac{c(1+|x|^p)t^{1+\frac{p}q}}{\delta^{p/q}}e^{\frac{tc_p}{p}},
 \end{align*}
 where $1/p+1/q=1$. By choosing $\delta:=\frac{\epsilon}{r}$ and $t_0>0$ sufficiently small, we obtain 
 \[ \E\left[\int_0^t g(X_s^{x,\alpha})\,\d s\right] \leq g(x)+(1+|x|^p)\epsilon t \]
 for all $t\in [0,t_0]$, $x\in\Rd$ and $\alpha\in\Aad$. Taking the supremum over $\alpha\in\Aad$ yields 
 \begin{equation} \label{eq:control.ub}
  \left(\frac{S(t)f-f}{t}\right)(x)-g(x)\leq (1+|x|^p)\epsilon
  \quad\mbox{for all } t\in [0,t_0] \mbox{ and } x\in\Rd. 
 \end{equation}
 In order to prove a lower bound, it is sufficient to take the supremum over controls which 
 are constant in time. For every $\alpha\in\A$ and $x\in\Rd$, we define 
 \[ g_\alpha(x):=\frac{1}{2}\tr\big(\sigma(\alpha)^2 D^2f(x)\big)+b(x,\alpha)^T Df(x)-L(\alpha). \]
 By Assumption~\ref{ass:control}(i), there exists $r\geq 0$ with 
 \[ |g_\alpha(x)-g_\alpha(y)|\leq r|x-y| \quad\mbox{and}\quad |g_\alpha(x)|\leq r(1+|x|) 
    \quad\mbox{for all } x,y\in\Rd \mbox{ and } \alpha\in\A. \] 
 In addition, for every $t\geq 0$, $x\in\Rd$ and $\alpha\in\A$,
 \[ \E[f(X_t^{x,\alpha})-L(\alpha)]-f(x)=\E[g_\alpha(X_t^{x,\alpha})]. \]
 Hence, similar to the previous estimates, there exists $t_1\in (0,t_0]$ with 
 \begin{equation} \label{eq:control.lb}
  \left(\frac{S(t)f-f}{t}\right)(x)-g(x)\geq -(1+|x|^p)\epsilon
  \quad\mbox{for all } t\in [0,t_1] \mbox{ and } x\in\Rd. 
 \end{equation}
 Combining inequality~\eqref{eq:control.ub} and inequality~\eqref{eq:control.lb} yields 
 \[ g=\lim_{h\downarrow 0}\frac{S(h)f-f}{h}. \]

 Performing similar computations for constant controls together with the assumptions $\E[\xi]=0$ 
 and $\E[|\xi|^2]=\tr(\E[\xi\xi^T])=\tr(\one_d)=d$ shows that the family $(I(t))_{t\geq 0}$ satisfies 
 Assumption \ref{ass:I2} as well with $I'(0)f=Af$ for all $f\in\Cbi$. In case that $\xi$ is not
 normally distributed, one has to apply Taylor's formula rather than It\^o's formula.
\end{proof}

Recall that Chernoff-type approximations reduce the original optimization problem to a sequence
of one-step optimization problems $f_{k h_n}:=I(h_n)f_{(k+1) h_n}$ of the form~\eqref{eq:one-step} 
with $f_{k^t_n h_n}:=f$. If the supremum in the one-step optimization problem~\eqref{eq:one-step} 
is attained by a control of the form $\alpha^\ast=g^{t,f}(x)$ for some measurable function $g^{t,f}$, 
then a maximizer of the approximated control problem $I(\pi_n^t)f$ is given by the piecewise constant 
Markovian control $\alpha^\ast_s=g^{h_n,f_{(k+1)h_n}}(X^{x,\alpha^\ast}_{k h_n})$ for all $s\in(k h_n,(k+1)h_n]$. 
As noted before, this procedure yields $\epsilon$-optimal picewise constant controls for the original 
optimization problem.  For the sake of illustration, we consider a simple example where the optimal controls 
are characterized by first-order conditions.

\begin{example}
 We consider a linear quadratic setting with $d=1$. Let $\sigma(\alpha):=\sigma>0$, 
 $b(x,\alpha):=ax+\alpha$ and $L(\alpha):=\frac{\alpha^2}2$ for all $\alpha\in\A:=\R$.
 Then, for every $f\in\Ck$, $x\in\R$ and $h_n\geq 0$, the one-step optimization problem 
 is given by
 \begin{equation} \label{eq:Iex}
  (I(h_n)f)(x):=\sup_{\alpha\in\R}\Big(\E\big[f(x+(ax+\alpha)h_n+\sigma\sqrt{h_n}\xi)\big]-\frac{\alpha^2}{2}h_n\Big),
\end{equation}
 where $\xi\sim\n(0,1)$. It follows from the proof of Theorem~\ref{thm:control.main} that 
 $I(h_n)\colon\Lipb\to\Lipb$ for all $h_n\geq 0$. Let $f\in\Lipb$ and $h_n>0$. Since 
 \[ \lim_{|\alpha|\to\infty}\sup_{x\in\R}\Big(\E\big[f(x+(ax+\alpha)h_n+\sigma\sqrt{h_n}\xi)\big]
    -\frac{\alpha^2}{2}h_n\Big)=-\infty\] 
 and the function $\alpha\mapsto \mathbb{E}[f(x+(ax+\alpha)h_n+\sigma \sqrt{h_n}\xi]$ is continuous, 
 there exists a maximizer $\alpha^\ast(h_n,x)$ of the supremum in equation~\eqref{eq:Iex} which 
 is bounded and measurable in $x$. Furthermore, the weak derivative $f'\in L^\infty$ exist 
 and the first-order condition 
 \begin{equation} \label{eq:foc}
  \alpha^\ast(h_n,x)=\E\big[f'\big(x+(ax+\alpha^\ast(h_n,x))h_n+\sigma\sqrt{h_n}\xi\big)\big]
 \end{equation}
 is satisfied. It follows from the boundedness of $f$ that
 \[ \E\big[f\big(x+(ax+\alpha^\ast(h_n,x))h_n+\sigma\sqrt{h_n}\xi\big)\big]-\frac{\alpha^\ast(h_n,x)^2}2 h_n
    \geq\E\big[f(x+axt+\sigma\xi_t)\big] \]
 and therefore $h_n\alpha^\ast(h_n,x)^2\leq 4\|f\|_\infty$. We obtain $h_n |\alpha^\ast(h_n,x)|\to 0$ as 
 $h_n\downarrow 0$. If $f'$ is even continuous, we can further conclude that
 \[ \alpha^\ast(h_n,x=\E\big[f'\big(x+(ax+\alpha^\ast(h_n))h_n+\sigma\sqrt{h_n}\xi\big)\big]\to f'(x)
    \quad\mbox{as } h_n\downarrow 0, \]
 i.e., the approximately optimal strategy $\alpha^\ast(h_n,x)$ in the state $x$ converges to $f'(x)$ 
 when the time step $h_n$ tends to zero.
\end{example}

We conclude this section by showing that the symmetric Lipschitz set can be computed explicitly
which provides a regularity result for a possibly degenerate fully nonlinear PDE. Similar results
have previously been obtained in~\cite{BK22,BK23}. 
Let $L^\infty$ be the space of all bounded measurable functions $f\colon\Rd\to\R$ and denote 
by $W^{1,\infty}$ the corresponding first order Sobolev space. For $f\in W^{1,\infty}$ and 
$\sigma\in\S^d_+$, we say that $\tr(\sigma^2 D^2f)$ exists as a regular distribution if and only if 
there exists a locally integrable function $g\colon\Rd\to\R$ with
\[ \int_{\Rd}g\phi\,\d x=-\int_{\Rd}(\sigma Df)^T\sigma D\phi\,\d x 
    \quad\mbox{for all } \phi\in\Cci. \]
In this case, we define $\tr(\sigma^2 D^2f):=g$. Since $\tr(\sigma D^2f)=\Delta f$ for $\sigma=\one_d$,
we subsequently also write $\Delta_\sigma f:=\tr(\sigma^2 D^2f)$ and $f\in D(\Delta_\sigma)$.

\begin{theorem} \label{thm:control.LSsym} 
 Suppose that Assumption~\ref{ass:control} is satisfied and that $\sup_{\alpha\in\A}L(\alpha)<\infty$.
 Furthermore, we assume that, for every $\alpha\in\A$, there exists $\bar{\alpha}\in\A$ with 
 $\sigma(\alpha)=\sigma(\bar{\alpha})$ and $b(x,\alpha)=-b(x,\bar{\alpha})$ for all $x\in\Rd$.
 Then, 
 \[ \LSsym\cap\Lipb=\left\{f\in\bigcap_{\alpha\in\A}D(\Delta_{\sigma(\alpha)})\cap W^{1,\infty}
  \colon\sup_{\alpha\in\A}\|\Delta_{\sigma(\alpha)}f\|_\kappa<\infty\right\}. \]
 In addition, if $b(x,\alpha)=b(y,\alpha)$ for all $x,y\in\Rd$ and $\alpha\in\A$, then  
 \[ S(t)\colon\LSsym\cap\Lipb\to\LSsym\cap\Lipb \quad\mbox{for all } t\geq 0. \]
\end{theorem}
\begin{proof}
 Since the semigroup $(S(t))_{t\geq 0}$ does not depend on the particular choice of $\xi$,
 we can choose $\xi\sim\mathcal{N}(0,\one_d)$. 
 It follows from Theorem~\ref{thm:I}(iv) and the inequality 
 \[ -S(t)(-f)\leq -I(t)(-f)\leq I(t)f\leq S(t)f \] 
 for all $t\geq 0$ and $f\in\Ck$ that $\LIsym=\LSsym$. Let $r\geq 0$ and $f\in\LIsym\cap\Lipb(r)$.
 Choose $c\geq 0$ and $t_0>0$ with 
 \[ \|I(t)(-f)+f\|_\kappa\leq ct \quad\mbox{and}\quad \|I(t)f-f\|_\kappa\leq ct 
    \quad\mbox{for all } t\in [0,t_0]. \]
 Then, for every $t\in [0,t_0]$, $\alpha\in\A$ and $x\in\Rd$, we obtain
 \[ \big|\E[f(x+\sigma(\alpha)\xi_t+b(x,\alpha)t)]-f(x)\big|\kappa(x)\leq (c+L(\alpha))t. \]
 Let $\eta\colon\Rd\to\R_+$ be an infinitely differentiable function satisfying 
 $\supp(\eta)\subset B_{\Rd}(1)$ and $\int_{\Rd}\eta(x)\,\d x=1$. For every $n\in\N$
 and $x\in\Rd$, we define $\eta_n(x):=n^d\eta(nx)$ and
 \[ (f*\eta_n)(x):=\int_{\Rd}f(x-y)\eta_n(y)\,\d y.\]
 For every $t\in [0,t_0]$, $\alpha\in\A$, $x\in\Rd$ and $n\in\N$, we use 
 Assumption~\ref{ass:control}(i) to estimate 
 \begin{align*}
  &\big|\E[f_n(x+\sigma(\alpha)\xi_t+b(x,\alpha)t)]-f_n(x)\big| \\
  &\leq\int_{\Rd}\big|\E\big[f(x-y+\sigma(\alpha)\xi_t+b(x,\alpha)t)
    -f(x-y+\sigma(\alpha)\xi_t+b(x-y,\alpha)t)\big]\big|\eta_n(y)\,\d y \\
  &\quad\; +\int_{\Rd}\big|\E[f(x-y+\sigma(\alpha)\xi_t+b(x-y,\alpha)t)]-f(x-y)\big|
    \frac{\kappa(x-y)}{\kappa(x-y)}\eta_n(y)\,\d y \\
  &\leq rCt\int_{\Rd}|y|\eta_n(y)\,\d y+(c+L(\alpha))t\int_{\Rd}(1+|x-y|^p)\eta_n(y)\,\d y.
 \end{align*}
 Hence, it follows from $\supp(\eta_n)\subset B_{\Rd}(1)$ that
 \begin{equation}
  \big|\E[f_n(x+\sigma(\alpha)\xi_t+b(x,\alpha)t)]-f_n(x)\big|\kappa(x)
  \leq\big(Cr+c_\kappa(c+L(\alpha))\big)t
 \end{equation}
 for all $t\in [0,t_0]$, $\alpha\in\A$, $x\in\Rd$ and $n\in\N$, where 
 \[ c_\kappa:=\sup_{x\in\Rd}\sup_{|y|\leq 1}\frac{\kappa(x)}{\kappa(x-y)}\leq 2^p. \]
It follows from It\^o's formula that 
 \[ \left|\frac{1}{2}\tr(\sigma(\alpha)^2 D^2 f_n(x))+b(x,\alpha)^T Df_n(x)\right|\kappa(x)
    \leq Cr+c_\kappa (c+L(\alpha)) \]
 for all $\alpha\in\A$, $x\in\Rd$ and $n\in\N$. Let $\alpha\in\A$.
 By assumption, there exists $\bar{\alpha}\in\A$ with $\sigma(\alpha)=\sigma(\bar{\alpha})$ and 
 $b(x,\alpha)=-b(x,\bar{\alpha})$ for all $x\in\Rd$. Then, 
 \begin{align*}
  |\tr(\sigma(\alpha)^2 D^2 f_n(x))|\kappa(x)
  &\leq\left|\frac{1}{2}\tr(\sigma(\alpha)^2 D^2 f_n(x))+b(x,\alpha)^T Df_n(x)\right|\kappa(x) \\
  &\quad\; +\left|\frac{1}{2}\tr(\sigma(\bar{\alpha})^2 D^2 f_n(x))+b(x,\bar{\alpha})^T Df_n(x)\right|\kappa(x) \\
  &\leq 2(Cr+c_\kappa (c+L(\alpha))
 \end{align*}
 for all $n\in\N$ and $x\in\Rd$. By Banach-Alaoglu's theorem, there exists $g\in L^\infty$
 with $\tr(\sigma(\alpha)^2 D^2f_n)\kappa\to g$ in the weak*-topology. Furthermore, 
 it follows from $f\in\Lipb$ that 
 \[ \sigma(\alpha)Df_n=\big(\sigma(\alpha)Df)*\eta_n\to\sigma(\alpha)Df. \]
 In particular, for every $\phi\in\Cci$, it holds $\frac{\phi}{\kappa}\in\Cci$ and thus
 \[ \int_{\Rd}\tr(\sigma(\alpha)^2 D^2f_n)\phi\,\d x
    =\int_{\Rd}\big(\tr(\sigma(\alpha)^2 D^2f_n)\kappa\big)\tfrac{\phi}{\kappa}\,\d x 
    \to\int_{\Rd}\frac{g}{\kappa}\phi\,\d x. \]
 This shows that $\tr(\sigma(\alpha)D^2 f)$ exists as a regular distribution and
 \[ \|\tr(\sigma(\alpha)D^2 f)\|_\kappa\leq 2(Cr+c_\kappa (c+L(\alpha)). \]
 Taking the supremum over $\alpha\in\A$, we obtain
 \[ f\in\bigcap_{\alpha\in\A}D(\Delta_{\sigma(\alpha)})\cap W^{1,\infty} \quad\mbox{and}\quad 
    \sup_{\alpha\in\A}\|\Delta_{\sigma(\alpha)}f\|_\kappa<\infty. \]
    
 Now, let $f\in\bigcap_{\alpha\in\A}D(\Delta_{\sigma(\alpha)})\cap W^{1,\infty}$ with 
 $\sup_{\alpha\in\A}\|\Delta_{\sigma(\alpha)}f\|_\kappa<\infty$.
 Define $f_n:=f*\eta_n$ for all $n\in\N$. Since $\xi_t\sim W_t$, it follows from It\^o's 
 formula and Assumption~\ref{ass:control}(i) that there exists $c\geq 0$ with 
 \begin{align*}
  &\big|\E[f_n(x+\sigma(\alpha)\xi_t+b(x,\alpha)t)]-f_n(x)\big| \\
  &=\big|\E[f_n(x+\sigma(\alpha)W_t+b(x,\alpha)t)]-f_n(x)\big| \\
  &=\left|\E\left[\int_0^t\Big(\frac{1}{2}\tr(\sigma(\alpha)^2 D^2f_n)+b(x,\alpha)^T Df_n\Big)
    (x+\sigma(\alpha)W_s+b(x,\alpha)s)\,\d s\right]\right| \\
  &\leq\left\|\frac{1}{2}\tr(\sigma(\alpha)^2 D^2f_n)+b(\cdot,\alpha)Df_n\right\|_\kappa
    \E\left[\int_0^t 1+|x+\sigma(\alpha)W_s+b(x,\alpha)s|^p\,\d s\right] \\
  &\leq ct\left\|\frac{1}{2}\tr(\sigma(\alpha)^2 D^2f_n)+b(\cdot,\alpha)Df_n\right\|_\kappa
 \end{align*}
 for all $t\in [0,1]$, $\alpha\in\A$, $x\in\Rd$ and $n\in\N$. Moreover, for every $\alpha\in\A$ and $n\in\N$, 
 \[ \tr(\sigma(\alpha)^2 D^2f_n)=\tr(\sigma(\alpha)^2 D^2f)*\eta_n
    \quad\mbox{and}\quad Df_n=Df*\eta_n. \] 
 For every $t\in [0,1]$ and $n\in\N$, it follows that
 \[ \|I(t)f_n-f_n\|_\kappa
    \leq \sup_{\alpha\in\A}\big(\|\Delta_{\sigma(\alpha)}f\|_\kappa+\|b(\cdot,\alpha)Df\|_\kappa\big)cc_\kappa t, \]
 where Assumption~\ref{ass:control}(i) guarantees that supremum is finite. Since $I(t)$ is continuous w.r.t. 
 the mixed topology, we obtain $f\in\LI\cap\Lipb$. Applying the previous arguments with $-f$ implies 
 $f\in\LIsym\cap\Lipb$. 
 
 The invariance of $\LSsym\cap\Lipb$ follows from the results in~\cite[Section~5]{BK23}. 
 Indeed, while the results in~\cite{BK23} are stated under slightly different conditions,
 a close inspection of the proofs reveals that it is sufficient to show $I(t)(-I(t)(-f))\leq -I(t)(-I(t)f)$ 
 for all $t\geq 0$ and $f\in\Ck$. In the present setting, this follows immediately from
 Fubini's theorem since the drifts do not depend on the current state.
\end{proof}

In the previous theorem, we do not assume that the matrices $\sigma(\alpha)$ are positive definite 
which is a common assumption for parabolic PDEs. Moreover, in the completely degenerate case 
$\sigma(\alpha)=0$ for all $\alpha\in\A$, it holds 
$\LSsym\cap\Lipb=\Lipb$ and $(S(t))_{t\geq 0}$ is a shift semigroup corresponding to a first 
order PDE. On the other hand, if $\sigma(\alpha)$ is positive definite for some $\alpha\in\A$, 
one can argue similar to the proof of~\cite[Theorem~3.1.7]{Lunardi95} to obtain
\[ \LSsym\cap\Lipb\subset D(\Delta)
	=\bigcap_{p\geq 1} \{f\in W^{2,p}_{\loc}\cap\Ck\colon (\Delta f)\kappa\in L^\infty\} \]
and Theorem~\ref{thm:dist} can be applied. 
Under an additional condition on the controls, one can further show that $\LSsym\subset\Lipb$, 
see~\cite[Theorem~6.3]{BK23}.\ In general, an explicit description of $\LSsym$ seems to be 
unknown even in the linear case $A=\Delta$, where $\LSsym$ coincides with the previously 
mentioned Favard space.

\subsection{Trace class Wiener processes with drift}
\label{sec:infdim}

Let $X$ be a real separable Hilbert space with	inner product $\langle\,\cdot,\cdot\,\rangle$ 
and norm $|\cdot|$. We denote by $\s_1(X)$ the space of all trace class operators endowed with 
the trace class norm $\|\cdot\|_{\s_1(X)}$. Throughout this subsection, we make the following 
assumptions.

\begin{assumption} \label{ass:Q}
 Let $B\times\Q\subset X\times\s_1(X)$ satisfy the following conditions:
 \begin{enumerate}
  \item $Q$ is selfadjoint and positive semidefinite for all $Q\in\Q$.
  \item $Q_1Q_2= Q_2Q_1$ for all $Q_1,Q_2\in\Q$.
  \item $B\times\Q$ is bounded. 
 \end{enumerate}
\end{assumption}

By the previous assumption and~\cite[Corollary~3.2.5]{Zimmer90}, there exists a joint orthonormal 
basis of eigenvectors $(e_k)_{k\in\N}\subset X$, i.e., for every $Q\in\Q$ and $k\in\N$, there exists 
$\mu_{Q,k}\geq 0$ with $Qe_k = \mu_{Q,k} e_k$. For every $Q\in\Q$ and $x\in X$,
\[ Qx = \sum_{k\in\N}\mu_{Q,k} \,\langle x,e_k\rangle e_k. \]
Hence, every $Q\in\Q$ can be identified with the sequence $\mu_Q=(\mu_{Q,k})_{k\in\N}\in\ell^1$ 
satisfying $\tr(Q)=\|Q\|_{\s_1(X)}=\|\mu_Q\|_{\ell^1}$. Assumption~\ref{ass:Q}(iii) implies that
$\Lambda:=B\times\mu(\Q)\subset X\times\ell^1$ is bounded, where $\mu(\Q):=\{\mu_Q\colon Q\in\Q\}$.
Let $(\xi^k)_{k\in\N}$ be a sequence of independent one-dimensional Brownian motions. 
For every $\lambda:=(b,\mu)\in\Lambda$, we define by
\[ (S_\lambda(t)f)(x)
	:=\E\left[f\left(x+tb+\sum_{k\in\N}\sqrt{\mu_k}\xi^k_t e_k\right)\right] \]
for all $t\geq 0$, $f\in\Cb$ and $x\in X$ the transition semigroup associated to the 
$Q$-Wiener process with drift $b$ and $\mu=\mu_Q$. For every $t\geq 0$ and $f\in\Cb$, let
\[ I(t)f:=\sup_{\lambda\in\Lambda}S_\lambda(t)f \]
Moreover, for every $t\geq 0$, $f\in\Cb$ and $x\in X$, we define
\[ (T(t)f)(x):=\sup_{(b,\mu)\in\A}
	\E\left[f\left(x+\int_0^t b_s\,\d s
	+\sum_{k\in\N}\Big(\int_0^t\sqrt{\mu_s^k}\,\d\xi^k_s\Big)e_k\right)\right], \]
where $\A$ consists of all predictable processes 
\[ (b,\mu)\colon\Omega\times [0,\infty)\to\Lambda,\; 
	(\omega,t)\mapsto (b_t(\omega), (\mu^k_t(\omega))_{k\in\N}). \]
Let $(h_n)_{n\in\N}\subset (0,\infty)$ be a sequence with $h_n\to 0$ and $\T_n\subset\T_{n+1}$ 
for all $n\in\N$, where $\T_n:=\{kh_n\colon k\in\N_0\}$. 
One can show that the sequence $(I(\pi_n^t)f)_{n\in\N}$ is non-decreasing for all $t\geq 0$ and $f\in\Cb$.

\begin{theorem} 
 It holds 
 \[ T(t)f=\lim_{n\to\infty}I(\pi_n^t)f=\sup_{n\in\N}I(\pi_n^t)f \]
 for all $t\geq 0$ and $f\in\Cb$ such that there exist a compact linear operator 
 $K\colon X\to X$ and a function $g\in\Cb$ with $f(x)=g(Kx)$ for all $x\in X$.
\end{theorem}
\begin{proof}
 Since the sequence $(I(\pi_n^t)f)_{n\in\N}$ is non-decreasing for all $t\geq 0$ and 
 $f\in\Cb$, there exists a family $(S(t))_{t\geq 0}$ of convex monotone operators on 
 $\Cb$ with $I(\pi_n^t)f\uparrow S(t)f$ for all $t\geq 0$ and $f\in\Cb$. For more details,
 we refer to~\cite[Section~2 and Example~7.2]{GNR22}. Dini's theorem implies 
 $I(\pi_n^t)f\to S(t)f$ for all $t\geq 0$ and $f\in\Cb$. By definition, the inequality 
 $S(t)f\leq T(t)f$ holds for all $t\geq 0$ and $f\in\Cb$. 
 
 First, we show that $S(t)f=T(t)f$ for all $t\geq 0$ and $f\in\Cb$ depending only on 
 finitely many coordinates, i.e., there exists $n\in\N$ such that 
 \[ f(x)=f\left(\sum_{k=1}^n \langle x,e_k\rangle e_k\right)
	\quad\mbox{for all } x\in X. \]
 Let $n\in\N$ and $X_n:=\Span\{e_1,\ldots, e_n\}\subset X$. For every $t\geq 0$, $f\in\Cb(X_n)$
 and $x\in X_n$, we define
 \begin{align*}
  (T_n(t)f)(x) 
:=\sup_{(b,\mu)\in\A_n}\E\Bigg[f\left(x+\int_0^t b_s\,\d s
  +\sum_{k=1}^n \Big(\int_0^t \sqrt{\mu_s^k}\,\d\xi^k_s\Big) e_k\right)\Bigg],
 \end{align*}
 where $\A_n$ denotes the set of all predictable processes
 $(b,\mu)\colon\Omega\times [0,\infty)\to\Lambda_n$ with
 \[ \Lambda_n:=\Big\{\big((\langle b,e_k\rangle)_{k=1,\dots,n},(\mu_k)_{k=1,\dots,n}\big)\colon (b,\mu)\in\Lambda\Big\}
 	\subset\R^n\times\R^n_+. \]
 In addition, for every $t\geq 0$, $f_n\in\Cb(X_n)$ and $x\in X_n$, we define
 \begin{align*}
  (S_{\lambda,n}(t)f)(x)
  &:=\E\left[f\left(x+tb+\sum_{k=1}^n \sqrt{\mu_k}\xi^k_t e_k\right)\right]
    \quad\mbox{for all } \lambda:=(b,\mu)\in\Lambda_n, \\
  (I_n(t)f)(x) &:=\sup_{\lambda\in\Lambda_n}(S_{\lambda,n}(t)f)(x).
 \end{align*}
 By Theorem \ref{thm:control.main}, the family $(S_n(t))_{t\geq 0}$ on $\Cb(X_n)$ given by
 \[ S_n(t)f:=\lim_{k\to\infty}I_n(\pi_k^t)f \quad\mbox{for all } (f,t)\in\Cb(X_n)\times\R_+, \]
 is a strongly continuous convex monotone semigroup with $S_n(t)f=T_n(t)f$ for all $t\geq 0$ 
 and $f\in\Cb(X_n)$. Let $f\in\Cb$ such that
 \[ f(x)=f\left(\sum_{k=1}^n \langle x,e_k\rangle e_k\right) \quad\mbox{for all } x\in X \]
 and define $\tilde f:=f|_{X_n}\in \Cb(X_n)$. From the construction of $(S(t))_{t\geq 0}$
 and $(S_n(t))_{t\geq 0}$ and the definition of $(T(t))_{t\geq 0}$ and $(T_n(t))_{t\geq 0}$,
 we obtain $S(t)f=S_n(t)\tilde{f}=T_n(t)\tilde{f}=T(t)f$ for all $t\geq 0$.
 
 Second, we show that $S(t)f=T(t)f$ for all $t\geq 0$ and all $f\in\Cb$ such that there exist 
 a compact linear operator $K\colon X\to X$ and a function $g\in\Cb$ with $f(x)=g(Kx)$ for all 
 $x\in X$. Since the finite rank operators are dense in the space of all compact linear operators 
 w.r.t. the operator norm $\|\cdot\|_{L(X)}$, there exists a sequence $(K_n)_{n\in\N}$ of finite 
 rank operators with $\|K-K_n\|_{L(X)}\to 0$. Let $f_n(x) := g(K_n x)$ for all $x\in X$ and $n\in\N$. 
 Then, it holds $\|f_n\|_\infty\leq\|g\|_\infty$ for all $n\in\N$. Moreover, since $K$ is compact, 
 for every $\epsilon >0$ and $r\geq0$, there exists $\delta>0$ with
 \[ |f(z)-f(Kx)|<\epsilon \]
 for all $x\in B_X(r)$ and $z\in X$ with $\|Kx-z\|<\delta$. Since $\sup_{x\in B_X(r)}\|Kx-K_nx\|\to 0$,
 it follows that
 \[ \lim_{n\to\infty}\sup_{x\in B_X(r)}|f(x)-f_n(x)|=0 \quad\mbox{for all } r\geq 0.\]
 Moreover, the boundedness of $\Lambda$ yields
 \[ c:=\sup_{(b,\mu)\in\A}\E\left[\left|\int_0^t b_s\,\d s
    +\sum_{k=1}^\infty \Big(\int_0^t \sqrt{\mu_s^k}\,\d\xi^k_s\Big) e_k\right|\right]<\infty. \]
 Hence, for every $r>0$ and $x\in X$, we use Chebyshev's inequality to estimate
 \begin{align*}
  \big(T(t)|f-f_n|\big)(x) &\leq\sup_{y\in B_X(r)}|f(y)-f_n(y)|
    +2\|g\|_\infty\sup_{(b,\mu)\in\A}\P\left(\left|X_t^{x,b,\mu}\right|>r\right) \\
  &\leq\sup_{y\in B_X(r)}|f(y)-f_n(y)|+\frac{2\|g\|_\infty}{r}\sup_{(b,\mu)\in\A} 
    \E\left[\left| X_t^{x,b,\mu}\right|\right] \\
  &\leq\sup_{y\in B_X(r)}|f(y)-f_n(y)|+\frac{2\|g\|_\infty(|x|+c)}{r},
 \end{align*}
 where 
 $X_t^{x,b,\mu}:=x+\int_0^t b_s\,\d s+\sum_{k=1}^\infty\Big(\int_0^t \sqrt{\mu_s^k}\,\d\xi^k_s\Big)e_k$ 
 for all $(b,\mu)\in \A$.
 By choosing first $r>0$ and then $n\in\N$ sufficiently large, we obtain
 \begin{align*}
  &\max\{|S(t)f-S(t)f_n|(x), |T(t)f-T(t)f_n|(x)\} \\
  &\leq\max\{(S(t)|f-f_n|)(x),(T(t)|f-f_n|)(x)\}
  =(T(t)|f-f_n|)(x)\to 0.
 \end{align*}
 Hence, we can use the first part to conclude $S(t)f=T(t)f$.
\end{proof}

Observe that the condition $f(x)=g(Kx)$ for a function $g\in\Cb$ and a bounded linear 
operator $K\colon X\to X$ implies that $f(x_n)\to f(x)$ whenever $Kx_n\to Kx$, which is also
referred to as $K$-continuity, cf.~\cite[Definition 3.4]{FGS17}. If $K\colon X\to X$ is a 
compact linear operator and  $f\colon X\to \R$ is $K$-continuous, then $f$ is weakly sequentially 
continuous. On the other hand, if $f\colon X\to \R$ is weakly sequentially continuous, 
then $f$ is $K$-continuous for every injective compact linear operator $K\colon X\to X$, 
see~\cite[Lemma 3.6]{FGS17}.

\subsection{Wasserstein perturbation of linear transition semigroups} 
\label{sec:wasserstein}

Let $p\in (1,\infty)$ and denote by $\p_p$ the set of all probability measures on the 
Borel-$\sigma$-algebra $\B(\Rd)$ with finite $p$-th moment. Let $(\mu_t)_{t\geq 0}\subset\p_p$ 
and $(\psi_t)_{t\geq 0}$ be a family of functions $\psi_t\colon\Rd\to\Rd$. Following the setting 
in~\cite{FKN23, BEK21}, for every $t\geq 0$, $f\in\Cb$ and $x\in\Rd$, we define a reference 
semigroup by
\[ (R(t)f)(x):=\int_{\Rd} f(\psi_t(x)+y)\,\d\mu_t(y) \]
and work under the following assumption.

\begin{assumption} \label{ass:R}
 Suppose that $(R(t))_{t\geq 0}$ is a semigroup. Furthermore, let 
 $(\mu_t)_{t\geq 0}$ and $(\psi_t)_{t\geq 0}$ satisfy the following conditions:
 \begin{enumerate}
  \item $\lim_{t\downarrow 0} \int_{\Rd} |y|^p\,\d\mu_t(y)=0$.
  \item There exist $r>0$ and $c\geq 0$ such that $\mu_t(B_{\Rd}(r)^c)\leq ct$ for all $t\in[0,1]$.	
  \item $\psi_t(0)=0$ for all $t\geq 0$.
  \item There exists $L\geq 0$ such that, for every $x,y\in\Rd$ and $t\in [0,1]$,
   \[ |\psi_t(x)-\psi_t(y)-(x-y)|\leq Lt|x-y|. \]
  \item For every $f\in\Cbi\cap\L_R$, the limit 
   \[ R'(0) f:=\lim_{h\downarrow 0}\frac{R(h)f-f}{h}\in\Cb \]
  exists. Furthermore, it holds $\Cci\subset\L_R$. 
 \end{enumerate}
\end{assumption}

For every $x,y\in \Rd$ and $t\in[0,1]$, the conditions~(iii) and~(iv) imply 
\begin{equation} \label{eq:psi}
 |\psi_t(x)-x|\leq Lt|x| \quad\mbox{and}\quad |\psi_t(x)-\psi_t(y)|\leq e^{Lt}|x-y|. 
\end{equation}

\begin{remark} \Newline
 \begin{enumerate}
  \item In case that $\psi_t(x):=e^{tA}x$ for a matrix $A\in\R^{d\times d}$, 
   the family $(R(t))_{t\geq0}$ satisfies the semigroup property if $(\mu_t)_{t\geq0}$ 
   is a skew-convolution, see~\cite{MR3385977,MR3324719}. In this case, $(R(t))_{t\geq0}$ 
   is a so-called (generalized) Mehler semigroup, see~\cite{MR1745332}.
  \item For every $t\geq 0$ and $x\in\Rd$, it holds
   \[|\psi_t(x)|\geq |x|-|\psi_t(x)-x|\geq (1-Lt)|x| \]
   and therefore $|\psi_t(x)|\to \infty$ as $|x|\to \infty$ for all $t< \frac1L$. 
   As a consequence, the semigroup $(R(t))_{t\geq 0}$ satisfies the Feller property, i.e.,  
   the space of all continuous functions vanishing at infinity is invariant under $R(t)$ for all $t\geq0$.
  \item One readily verifies that Assumption \ref{ass:R} is satisfied for Koopman semigroups 
  and transition semigroups of L\'evy and Ornstein--Uhlenbeck processes, see~\cite{FKN23} for the details.
 \end{enumerate}
\end{remark}

In the sequel, we consider a perturbation of the linear transition semigroup $(R(t))_{t\geq 0}$, 
where we take the supremum over all transition probabilities which are sufficiently 
close to the reference measure $\mu_t$. 
To that end, we endow $\p_p$ with the $p$-Wasserstein distance
\[ W_{p}(\mu,\nu):=\left(\inf_{\pi \in\Pi(\mu,\nu)}
	\int_{\Rd\times\Rd}|y-z|^p\,\d\pi(y, z)\right)^\frac{1}{p} \]
where $\Pi(\mu,\nu)$ consists of all probability measures on $\B(\Rd\times \Rd)$ 
with first marginal $\mu$ and second marginal $\nu$. Let $\varphi\colon\R_+\to [0,\infty]$ 
be a convex lower semicontinuous function with $\varphi(0)=0$ and $\varphi(v)>0$ 
for some $v>0$. The mapping $[0,\infty)\to [0,\infty],\; v\mapsto\varphi\big(v^{1/p}\big)$
is supposed to be convex which implies
\[ \varphi^*(w):=\sup_{v\geq 0}\big(vw-\varphi(v)\big)<\infty
	\quad\mbox{for all } w\geq 0. \]
For every $t\geq 0$, $f\in\Cb$ and $x\in\Rd$, we define
\[ (I(t)f)(x):=\sup_{\nu\in\p_p}\left(\int_{\Rd} f(\psi_t(x)+z)\, \d\nu(z)
	-\varphi_t\big(\W_p(\mu_t,\nu)\big)\right), \]
where $\varphi_t\colon [0,\infty)\to [0,\infty]$ denotes the rescaled function
\[ \varphi_t(v):=\begin{cases}
	t\varphi\big(\tfrac{v}{t}\big), & t>0,\, v\geq 0, \\
	0, & t=v=0, \\
	+\infty, & t=0, \,v\neq 0.
	\end{cases} \]
In addition to the Wasserstein perturbation, we consider a perturbation 
which is parametrized only by drifts $b\in\Rd$. For every $t\geq 0$, 
$f\in\Cb$ and $x\in\Rd$, we define
\[ (J(t)f)(x):=\sup_{b\in\Rd}
	\left(\int_{\Rd} f(\psi_t(x)+y+b)\, \d\mu_t(y)-\varphi_t(|b|t)\right). \]
We remark that $\W_p(\mu_t,\nu_b)=|b|t$ for all $b\in\Rd$ and $t\geq 0$, where 
\[ \nu_b(A):=\int_{\Rd}\one_A(b+y)\,\d\mu_t(y) 
	\quad\mbox{for all } A\in\B(\Rd). \]
Let $(h_n)_{n\in\N}\subset (0,\infty)$ be a sequence with $h_n\to 0$ and
$\T_n\subset\T_{n+1}$ for all $n\in\N$, where $\T_n:=\{kh_n\colon k\in\N_0\}$.

\begin{theorem} \label{thm:WS}
 There exists a strongly continuous convex monotone semigroup $(S(t))_{t\geq 0}$
 on $\Cb$ given by
 \[ S(t)f=\lim_{n\to\infty}I(\pi_n^t)f=\lim_{n\to \infty}J(\pi_n^t)f
 	\quad\mbox{for all } (f,t)\in\Cb\times\R_+ \]
 such that $\Cbi\cap\L_R\subset D(A)$ and 
 \[ Af=R'(0)f+\varphi^*(|\nabla f|) \quad\mbox{for all } f\in\Cbi\cap\L_R. \]
\end{theorem}
\begin{proof}
 In a first step, we show that both $(I(t))_{t\geq0}$ and $(J(t))_{t\geq0}$ satisfy 
 Assumption~\ref{ass:I} by showing that they satisfy Assumption~\ref{ass:I2}, where 
 Assumption~\ref{ass:I2}(ii) is satisfied with $\Cbi\cap\L_R$ instead of $\Cbi$ and 
 $I'(0)f=J'(0)f=R'(0)f+\phi^*(|\nabla f|)$ for $f\in \Cbi\cap \L_R$. By definition of 
 $(I(t))_{t\geq0}$ and $(J(t))_{t\geq0}$, Assumption~\ref{ass:I2}(i)-(iv) is
 satisfied with $\omega=\omega_r=0$ for all $r\geq0$. It follows from Assumption~\ref{ass:R}(iv)
 that
 \begin{align*}
  (\tau_x(I(t)f))(y)-I(t)(\tau_x f)(y)
  &\leq\sup_{\nu\in\p_p}\int_{\Rd}|f(\psi_t(x+y)+z)-f(\psi_t(y)+x+z)|\,\d\nu(z) \\
  &\leq r|\psi_t(x+y)-\psi_t(y)-x|\leq Lrt|x|
 \end{align*}
 for all $r\geq 0$ and $f\in\Lipb(r)$ showing that Assumption \ref{ass:I2}(v) is valid. 
 Similarly, one can verify Assumption~\ref{ass:I2}(v) for the family $(J(t))_{t\geq0}$.
 Moreover, for all $r,t\geq 0$, $f\in\Lipb(r)$ and $x,y\in\Rd$, inequality~\eqref{eq:psi} implies
 \[  |(I(t)f)(x)-(I(t)f)(y)| 
 	\leq\sup_{\nu\in\p_p}\int_{\Rd}|f(\psi_t(x)+z)-f(\psi_t(y)+z)|\,\d\nu(z)
  	\leq e^{Lt}r|x-y|. \]
 This shows that $(I(t))_{t\geq0}$ and similarly $(J(t))_{t\geq0}$ satisfy Assumption~\ref{ass:I2}(viii).
 Since~\cite[Lemma 3.10]{FKN23} implies $I(s)I(t)f\leq I(s+t)f$ for all $s,t\geq 0$ and $f\in\Cb$,
 the continuity from above of $I(t)$ for all $t\geq0$ guarantees Assumption~\ref{ass:I2}(vii). 
 Hence, by Theorem~\ref{thm:I}, there exists a strongly continuous convex monotone semigroup 
 $(S(t))_{t\geq0}$ on $\Cb$ which is given by 
 \[ S(t)f=\lim_{n\to\infty}I(\pi_n^t)f=\sup_{n\in\N}I(\pi_n^t)f \quad\mbox{for all } (f,t)\in \Cb\times \R_+. \]
 For every $T\geq 0$, $K\Subset\Rd$ and $(f_n)_{n\in\N}\subset\Cb$ with $f_n\downarrow 0$, 
 we use Dini's theorem and the inequality $J(t)f\leq I(t)f$ to obtain
 \[ 0\leq\sup_{(t,x)\in [0,T]\times K}\sup_{n\in\N}\big(J(\pi_n^t)f_k)(x)
 	\leq \sup_{(t,x)\in [0,T]\times K}(S(t)f_k)(x)\downarrow 0 \quad\,\mbox{as } k\to\infty. \]
 Next, we prove that $I'(0)f=J'(0)f=R'(0)f+\phi^*(|\nabla f|)$ for all $f\in\Cbi\cap\L_R$. 
 For every $r,t\geq 0$ and $f\in\Lipb(r)$, it follows from~\cite[Equation~(3.7)]{FKN23} that
 \begin{equation} \label{eq:IJR}
  0\leq J(t)f-R(t)f\leq I(t)f-R(t)f\leq\varphi^*(r)t
 \end{equation}
 and therefore $\Cbi\cap\L_R\subset\L_J$. Let $f\in\Cbi\cap\L_R$. By~\cite[Lemma 3.11]{FKN23},
 it holds
 \[ \frac{J(h)f-f}{h}\leq\frac{I(h)f-f}{h}\to R'(0)f+\varphi^*(|\nabla f|) \quad\mbox{as } h\downarrow 0. \]
 Moreover, for every $h>0$ and $b,x,y\in\Rd$, Taylor's formula implies
 \[ |f(\psi_h(x)+y+bh)-f(\psi_h(x)+y)-\langle\nabla f(\psi_h(x)+y),bh\rangle|
 	\leq\|D^2 f\|_\infty|b|^2h^2. \]
 As seen in the proof of~\cite[Lemma~3.6]{FKN23}, it holds $R(h)g\to g$ as $h\downarrow 0$ 
 for all $g\in\Lipb$ and therefore
 \begin{align*}
  &\frac{J(h)f-f}{h}=\frac{J(h)f-R(h)f}{h}+\frac{R(h)f-f}{h} \\
  &\geq\frac{R(h)f-f}{h}+\int_{\Rd}\langle\nabla f(\psi_h(\,\cdot\,)+y,b\rangle\d\mu_h(y)-\varphi(|b|)
  	-\|D^2 f\|_\infty |b|^2h \\
  &=\frac{R(h)f-f}{h}+R(h)\big(\langle\nabla f,b\rangle\big)-\varphi(|b|)-\|D^2 f\|_\infty |b|^2h \\
  &\to R'(0)f+\langle\nabla f,b\rangle-\varphi(|b|)
 \end{align*}
 as $h\downarrow 0$ for all $b\in\Rd$. Taking the supremum over $b\in\Rd$ yields
 \[ \lim_{h\downarrow 0}\frac{J(h)f-f}{h}=R'(0)f+\varphi^*(|\nabla f|). \]
 By Theorem~\ref{thm:I}, there exist a subsequence $(n_l)_{l\in\N}$ and a strongly continuous 
 convex monotone semigroup $(T(t))_{t\geq 0}$ on $\Cb$ with $T(t)0=0$ given by
 \begin{equation} \label{eq:JTb}
  T(t)f=\lim_{l\to\infty}J(\pi_{n_l}^t)f \quad\mbox{for all } (f,t)\in\Cb\times\T,
 \end{equation}
 where $\T\subset\R_+$ is a countable dense set with $0\in\T$. It holds $\Cbi\cap\L_R\subset D(B)$ 
 and $Bf=R'(0)f+\varphi^*(|\nabla f|)$ for all $f\in\Cbi\cap\L_R$. 
 Furthermore, Theorem~\ref{thm:I2} guarantees that condition~\eqref{eq:trans} 
 is valid for all $f\in\Lipb$ and that $T(t)\colon\Lipb\to\Lipb$ for all $t\geq 0$. 
 
 Second, inequality~\eqref{eq:IJR} implies $0\leq T(t)f-R(t)f\leq\varphi^*(r)t$ for all 
 $r,t\geq 0$ and $f\in\Lipb(r)$ and therefore the set $\D:=\L_R\cap\Lipb=\L_T\cap\Lipb$
 does not depend on the choice of the convergence subsequence in equation~\eqref{eq:JTb}
 and satisfies $T(t)\colon\D\to\D$ for all $t\geq 0$. We show that, for every $f\in\D$, 
 there exists a sequence $(f_n)_{n\in\N}\subset\Cbi\cap\L_R$ with $f_n\to f$ and 
 $\BG f=\Glim_{n\to\infty}Bf_n$. Let $f\in\D$ and $\eta\colon\Rd\to\R_+$ be infinitely 
 differentiable with $\supp(\eta)\subset B_{\Rd}(1)$, and $\int_{\Rd}\eta(x)\,\d x=1$. 
 For every $t\geq 0$, $n\in\N$ and $x\in\Rd$, Fubini's theorem and Assumption~\ref{ass:R}(iv) 
 imply 
 \begin{align*}
  &|R(t)f_n-(R(t)f)*\eta_n|(x) \\
  &\leq\int_{B(1)}\left(\int_{\Rd}|f(\psi_t(x)+y-z)-f(\psi_t(x-z)+y)|\,\d\mu_t(y)\right)\eta_n(z)\,\d z
  \leq Lrt,
 \end{align*}
 where $\eta_n(x):=n^d\eta(nx)$ and $f_n:=f*\eta_n\in\Cbi$. We obtain 
 \[ \|R(t)f_n-f_n\|_\infty\leq\|R(t)f-f\|_\infty+Lrt \]
 and thus $f_n\in\L_R$. 
 Furthermore, inequality~\eqref{eq:IJR} yields
 \[ 0\leq T(t)f-R(t)f\leq S(t)f-R(t)f\leq\varphi^*(r)t \]
 for all $r,t\geq 0$ and $f\in\Lipb(r)$ and therefore
 \[ \D:=\L_R\cap\Lipb=\L_S\cap\Lipb=\L_T\cap\Lipb. \]
 Since $S(t)\colon\D\to\D$ and $T(t)\colon\D\to\D$ for all $t\geq 0$, we can
 use Theorem~\ref{thm:conv} and Theorem~\ref{thm:comp2} to conclude that both
 semigroups coincide. 
 \end{proof}

In the particular case that $(S(t))_{t\geq 0}$ can be represented by the entropic risk measure, 
as a byproduct of Theorem~\ref{thm:WS}, we recover that $\mu_t$ satisfies the Talagrand $T_2$ 
inequality, see~\cite[Chapter~22]{Villani09} and~\cite{Talagrand96}. The proof uses the fact 
that the sequence $(I(\pi_n^t)f)_{n\in\N}$ is non-increasing for all $t\geq 0$ and $f\in\Cb$, 
see~\cite[Lemma 3.10]{FKN23}.

\begin{corollary}
 It holds $\W_2(\nu,\mu_t)\leq\sqrt{2t H(\nu|\mu_t)}$ for all $t\geq 0$ 
 and $\nu\in\P_2$, where $H(\nu|\mu_t)$ denotes the relative entropy of 
 $\nu$ w.r.t. $\mu_t:=\n(0,t\one_d)$. 
\end{corollary}
\begin{proof}
 Choose $\psi_t:=\id_{\Rd}$, $\mu_t:=\n(0,t\one_d)$ and $\varphi(v):=\frac{v^2}{2}$ 
 for all $t,v\geq 0$. Let $(W_t)_{t\geq 0}$ be a $d$-dimensional Brownian motion. 
 We show that 
 \[(S(t)f)(x)=(\tilde{S}(t)f)(x)
 	:=\frac{1}{2}\log\big(\E[\exp(2f(x+W_t))]\big) \]
 for all $t\geq 0$, $f\in\Cb$ and $x\in\Rd$.\ 
 One readily verifies that the family $(\tilde{S}(t))_{t\geq 0}$ is a strongly continuous 
 convex monotone semigroup that satisfies inequality~\eqref{eq:trans}. Furthermore, it holds 
 $\tilde S(t)\colon \Lipb\to \Lipb$ for all $t\geq0$ and $\Cbi\subset D(\tilde{A})$ with 
 \[ \tilde{A}f=\tfrac{1}{2}\big(\Delta f+|\nabla f|^2\big)
 	 \quad\mbox{for all } f\in\Cbi. \]
 Theorem \ref{thm:conv3} implies $\tilde S(t)f=S(t)f\leq I(t)f$ for all $t\geq0$ and $f\in \Cb$.
 Hence, it follows from Fenchel--Moreaus's theorem that
 \begin{align*}
  \frac{\W_2(\nu,\mu_t)^2}{2t}
  &=\sup_{f\in\Cb}\left(\int_{\Rd} f\,\d\nu-(I(t)f)(0)\right) \\
  &\leq\sup_{f\in\Cb}\left(\int_{\Rd} f\,\d\nu-(\tilde{S}(t)f)(0)\right)
  =H(\nu|\mu_t). \qedhere
 \end{align*}
\end{proof}

\appendix

\section{$\Gamma$-convergence}
\label{app:gamma}

Following the works of Beer~\cite{Beer78}, Dal Maso~\cite{DalMaso93} and Rockafellar 
and Wets~\cite{RW98}, we gather some basics about $\Gamma$-convergence. 
In~\cite{DalMaso93} and~\cite{RW98}, all results are formulated for extended 
real-valued lower semicontinuous functions but a function $f\colon X\to\Rbar$ is
upper semicontinuous if and only if $-f\colon X\to (-\infty,\infty]$ is lower semicontinuous
and all results immediately transfer to our setting.

\begin{definition}
 For every sequence $(f_n)_{n\in\N}\subset \Uk$, which is bounded above, we define
 \[ \Big(\Glimsup_{n\to\infty}f_n\Big)(x)
 	:=\sup\Big\{\limsup_{n\to\infty}f_n(x_n)\colon
 	(x_n)_{n\in\N}\subset X\mbox{ with } x_n\to x\Big\}\in\Rbar \] 
 for all $x\in X$. 
 Moreover, we say that $f=\Glim_{n\to\infty}f_n$ with $f\in\Uk$ if, for every $x\in X$,
 \begin{itemize}
  \item $f(x)\geq\limsup_{n\to\infty}f_n(x_n)$ for every sequence $(x_n)_{n\in\N}\subset X$ with $x_n\to x$,
  \item $f(x)=\lim_{n\to\infty}f_n(x_n)$ for some sequence $(x_n)_{n\in\N}\subset X$ with $x_n\to x$.
 \end{itemize}
 For every $t\geq 0$ and $(f_s)_{s\geq 0}\subset\Uk$ being bounded above, we define 
 \[ \Glimsup_{s\to t}f_s
 	:=\sup\Big\{\Glimsup_{n\to\infty}f_{s_n}\colon 0\leq s_n\to t\Big\}\in\Uk. \]
 In addition, we write $f=\Glim_{s\to t}f_s$ with $f\in\Uk$ if $f=\Glim_{n\to\infty}f_{s_n}$ 
 for all sequences $(s_n)_{n\in\N}\subset [0,\infty)$ with $s_n\to t$. 
\end{definition}

\begin{lemma} \label{lem:gamma}
 Let $(f_n)_{n\in\N}\subset\Uk$ and $(g_n)_{n\in\N}\subset\Uk$ be bounded above
 and $f,g\in\Uk$.
 \begin{enumerate}
  \item It holds $\Glimsup_{n\to\infty}f_n\in\Uk$. Furthermore, $(f_n)_{n\in\N}\subset\Uk$ 
   has a $\Gamma$-convergent subsequence, i.e., $\Glim_{k\to\infty}f_{n_{k}}\in\Uk$ exists
   for a subsequence $(n_k)_{k\in\N}$.
  \item We have $f=\Glim_{n\to\infty}f_n$ if and only if every subsequence $(n_k)_{k\in\N}$
   has another subsequence $(n_{k_l})_{l\in\N}$ with $f=\Glim_{l\to\infty}f_{n_{k_l}}$.
  \item If $f_n\downarrow f$, then $f=\Glim_{n\to\infty}f_n$.
  \item It holds 
   $\Glimsup_{n\to \infty} (f_n+g_n)\leq\Glimsup_{n\to\infty}f_n+\Glimsup_{n\to\infty}g_n$.
   Moreover, we have $\Glimsup_{n\to\infty}f_n\leq\Glimsup_{n\to\infty}g_n$ if $f_n\leq g_n$
   for all $n\in\N$.
  \item Assume that $f\in \Ck$, $f_n\to f$ uniformly on compacts and $g=\Glim_{n\to\infty}g_n$. 
   Then, it holds $f+g=\Glim_{n\to\infty}(f_n+g_n)$.
  \item If $f=\Glimsup_{n\to\infty}f_n$ and $g\in\Ck$, then
   $f\vee g=\Glimsup_{n\to\infty}(f_n\vee g)$.
  \item It holds 
   $(\Glimsup_{n\to\infty}f_n)(x)\geq\limsup_{n\to\infty}\sup_{y\in B(x,\delta_n)}f_n(y)$
   for all $x\in X$ and $(\delta_n)_{n\in\N}\subset (0,\infty)$ with $\delta_n\to 0$.
 \end{enumerate}
\end{lemma}
\begin{proof}
 Part~(i) follows immediately from~\cite[Remark~4.11]{DalMaso93},~\cite[Theorem~4.16]{DalMaso93}
 and~\cite[Theorem~8.4]{DalMaso93}. Regarding part~(ii) and~(iii), we refer 
 to~\cite[Proposition~8.3]{DalMaso93} and~\cite[Proposition~5.4]{DalMaso93}. 
 art~(iv) and~(vii) are direct consequences of the definition of the $\Gamma$-limit superior. 
 In order to show part~(v), let $x\in X$ and $(x_n)_{n\in\N}\subset X$ be a sequence
 with $x_n\to x$ and $g(x)=\lim_{n\to \infty}g_n(x_n)$. For every $n\in\N$,
 \[ \big|(f+g)(x)- (f_n+g_n)(x_n)\big|
 	\leq|f(x)-f(x_n)|+\sup_{y\in K}|f(y)-f_n(y)|+|g(x)-g_n(x_n)|, \]
 where $K:=\{x_n\colon n\in\N\}\cup\{x\}$ is compact. Since $f_n\to f$ uniformly 
 on compacts and $f$ is continuous, the right-hand side converges to zero. We 
 obtain $f+g\leq\Glim_{n\to\infty}(f_n+g_n)$ and the reverse inequality follows
 from part~(iv).
 It remains to show part~(vi). The inequality $f\vee g\leq\Glimsup_{n\to\infty}(f_n\vee g)$
 follows from part~(iv). Let $(x_n)_{n\in\N}\subset X$ and $x\in X$ with 
 $x_n\to x$. Continuity of $g$ implies
 \[ \limsup_{n\to\infty}\,(f_n\vee g)(x)=\lim_{k\to\infty}(f_{n_k}\vee g)(x_k)
 	=\Big(\limsup_{k\to\infty}f_{n_k}(x_{n_k})\Big)\vee g(x)\leq (f\vee g)(x), \]
 where $(x_{n_k})_{k\in\N}$ is a suitable subsequence approximating the limit superior.
\end{proof}

The following geometric characterization of the $\Gamma$-limit superior is based on 
the work of Beer~\cite{Beer78} and is very useful to link $\Gamma$-upper 
semicontinuity with continuity from above. Let $f\in\Uk$ and $\epsilon>0$. 
Following~\cite{Beer78}, the upper $\epsilon$-parallel function to $f$ is defined by
\[ f^\epsilon\colon X\to\R,\; x\mapsto\frac{1}{\kappa(x)}\Big(\sup_{y\in B(x,\epsilon)}
	\max\big\{ f(y)\kappa(y),-\tfrac{1}{\epsilon}\big\}+\epsilon\Big), \]
where $B(x,\epsilon):=\{y\in X\colon d(x,y)\leq\epsilon\}$. 
If closed bounded sets in $X$ are compact 
one can show that $f^\epsilon$ is upper semicontinuous, see the proof of~\cite[Lemma~1.3]{Beer78}, 
but for a general metric space $(X,d)$ the upper semicontinuity of $f^\epsilon$ cannot 
be guaranteed. For this reason, we consider the upper semicontinuous envelope of $f^\epsilon$
defined by
\[ \fe\colon X\to\R,\; x\mapsto\limsup_{y\to x}f^\epsilon(y)
	=\inf_{\delta>0}\sup_{y\in B(x,\delta)}f^\epsilon(y).\]
Denoting by $f^\epsilon$ the constant sequence $(f^\epsilon)_{n\in\N}$, 
for every $x\in X$,
\[ \fe(x)=\sup\Big\{\limsup_{n\to \infty}f^\epsilon(x_n)\colon (x_n)_{n\in \N}\subset X 
	\mbox{ with }x_n\to x\Big\} =\Glimsup_{n\to \infty}f^\epsilon.\]

\begin{lemma} \label{lem:fe} \Newline
 \begin{enumerate}
  \item For every $f\in \Uk$ and $\epsilon>0$, 
   \begin{equation} \label{eq:fe1}
    f^\epsilon\leq\fe\leq\inf_{\epsilon'>\epsilon}f^{\epsilon'}
	\quad\mbox{and}\quad
	-\tfrac{1}{\epsilon}\leq \fe\kappa\leq\|f^+\|_\kappa+\epsilon. 
   \end{equation}
   Furthermore, it holds $\fe\in\Uk$ for all $\epsilon>0$ and $\fe\downarrow f$ as 
   $\epsilon\downarrow 0$. 
  \item Let $(f_n)_{n\in\N}\subset\Uk$ be bounded above and $f\in\Uk$. Then, 
  $\Glimsup_{n\to\infty}f_n\leq f$ if and only if, for every $\epsilon>0$ and $K\Subset X$, 
  there exists $n_0\in\N$ with
  \begin{equation} \label{eq:fe2}
    f_n(x)\leq\fe(x) \quad\mbox{for all } x\in K \mbox{ and } n\geq n_0.
   \end{equation}
 \end{enumerate}
\end{lemma}
\begin{proof}
 First, we show inequality~\eqref{eq:fe1}. Let $f\in \Uk$, $\epsilon>0$ and $x\in X$. For 
 every $\epsilon'>\epsilon$,
 \begin{align*}
  f^\epsilon(x)
  &\leq\fe(x)\leq\sup_{y\in B(x,\epsilon'-\epsilon)}f^\epsilon(y) \\
  &=\sup_{y\in B(x,\epsilon'-\epsilon)}\frac{1}{\kappa(y)}
 	 \Big(\sup_{z\in B(y,\epsilon)}\max\big\{f(z)\kappa(z),-\tfrac{1}{\epsilon}\big\}+\epsilon\Big) \\
  &\leq \frac{c_{\epsilon'}(x)}{\kappa(x)}\Big(\sup_{y\in B(x,\epsilon'-\epsilon)}
 	 \sup_{z\in B(y,\epsilon)}\max\big\{f(z)\kappa(z),-\tfrac{1}{\epsilon}\big\}+\epsilon\Big)\\
  &\leq\frac{c_{\epsilon'}(x)}{\kappa(x)}\Big(\sup_{z\in B(x,\epsilon')}
  	\max\big\{f(z)\kappa(z), -\tfrac{1}{\epsilon'}\big\}+\epsilon'\Big)
  	=c_{\epsilon'}(x) f^{\epsilon'}(x),
 \end{align*}
 where $c_{\epsilon'}(x):=\sup_{y\in B(x,\epsilon'-\epsilon)}\frac{\kappa(x)}{\kappa(y)}$. 
 Continuity of $\kappa$ implies $c_{\epsilon'}(x)\downarrow 1$ as $\epsilon'\downarrow\epsilon$. 
 Hence, taking the infimum over $\epsilon'>\epsilon$ in the previous estimate yields the
 first part of inequality~\eqref{eq:fe1}. Furthermore, we can estimate
 \[ \big(\fe\kappa\big)^+(x)\leq\inf_{\epsilon'>\epsilon}\big(f^{\epsilon'}\kappa\big)^+(x)
 	\leq\inf_{\epsilon'>\epsilon}\Big(\sup_{y\in X}\,(f\kappa)^+(y)+\epsilon'\Big)
 	=\|f^+\|_\kappa+\epsilon. \]
 In particular, we obtain $\fe\in\Uk$, because $\fe$ is upper semicontinuous by definition.
	
 Second, we show that $\overline{f}^\delta\downarrow f$ as $\delta\downarrow 0$. 
 Due to inequality~\eqref{eq:fe1} it is sufficient to prove $f^\delta\downarrow f$ as 
 $\delta\downarrow 0$. Since $f\kappa$ is upper semicontinuous, for every $x\in X$ 
 and  $\epsilon>0$, there exists $\delta>0$ such that $(f\kappa)(y)\leq (f\kappa)(x)+\epsilon$ 
 for all $y\in B(x,\delta)$. We obtain 
 \[ f(x)\leq f^\delta(x)
 	=\frac{1}{\kappa(x)}\sup_{y\in B(x,\delta)}\max\big\{(f\kappa)(y),-\tfrac{1}{\delta}\big\}+\delta
 	\leq\max\big\{f(x),-\tfrac{1}{\delta}\big\}+\delta+\epsilon. \]
 This implies $f(x)\leq\inf_{\delta>0} f^\delta(x)\leq f(x)+\epsilon\downarrow f(x)$
 as $\epsilon\downarrow 0$.
	
 Third, let $(f_n)_{n\in\N}\subset\Uk$ be bounded above and $f\in\Uk$ with
 $\Glimsup_{n\to\infty}f_n\leq f$. We follow the proof of~\cite[Lemma~1.5]{Beer78}
 to verify inequality~\eqref{eq:fe2}. Let $K\Subset X$ and $\epsilon>0$.
 Since $\Glimsup_{n\to\infty}f_n\leq f$ and $\kappa>0$ is continuous, we obtain
 \[ \limsup_{n\to\infty}(f_n\kappa)(x_n)\leq (f\kappa)(x)
 	\quad\mbox{for all } x\in K \mbox{ and } (x_n)_{n\in\N}\subset X \mbox{ with }x_n\to x.\]
 Hence, for every $x\in K$, there exist $n_x\in\N$ and $r_x\in (0,\epsilon)$ such that
 \[ (f_n\kappa)(y)\leq\max\big\{(f\kappa)(x),-\tfrac{1}{\epsilon}\big\}+\epsilon
 	\quad\mbox{for all } n\geq n_x \mbox{ and } y\in B(x,r_x). \]
 By compactness of $K$, we can choose $x_1,\ldots,x_k\in K$ with 
 $K\subset\bigcup_{i=1}^k B(x_i,r_{x_i})$. Define $n_0:=n_{x_1}\vee\ldots\vee n_{x_k}$. 
 Let $x\in K$ and $i\in\{1,\ldots,k\}$ with $d(x,x_i)<r_{x_i}<\epsilon$. We obtain
 \[ f_n(x)\leq\frac{1}{\kappa(x)}\Big(\max\big\{(f\kappa)(x_i),-\tfrac{1}{\epsilon}\big\}+\epsilon\Big)
 	\leq\fe(x) \quad\mbox{for all } n\geq n_0.\]
 
 Fourth, let $(f_n)_{n\in\N}\subset\Uk$ be bounded above and $f\in\Uk$ such that 
 inequality~\eqref{eq:fe2} is valid. Let $x\in X$ and $(x_n)_{n\in\N}\subset X$ with 
 $x_n\to x$. Since $K:=\{x_n\colon n\in\N\}\cup\{x\}$ is compact, for every $\epsilon>0$,
 there exists $n_0\in\N$ such that $f_n(x_n)\leq\fe(x_n)$ for all $n\geq n_0$. We obtain
 $\limsup_{n\to\infty}f_n(x_n)\leq\fe(x)$ for all $\epsilon>0$, because $\fe$ is upper 
 semicontinuous. Hence, part~(i) implies 
 $\Glimsup_{n\to\infty}f_n\leq\inf_{\epsilon>0}\fe=f$.
\end{proof}

The definition of $\fe$ simplifies if $(X,d)$ satisfies an additional geometric property.

\begin{lemma}
 Assume that $(X,d)$ has midpoints, i.e., for every $x,z\in X$ and $\lambda\in [0,1]$, 
 there exists $y_\lambda\in X$ with $d(x,y_\lambda)=\lambda d(x,z)$ and
 $d(y_\lambda,z)=(1-\lambda)d(x,z)$.
 Then, it holds $\fe=\inf_{\epsilon'>\epsilon} f^{\epsilon'}$ for all $f\in\Uk$ and 
 $\epsilon>0$.
\end{lemma}
\begin{proof}
 Let $x\in X$ and $\epsilon'>\epsilon$. Since $(X,d)$ has midpoints, for every 
 $z\in B(x,\epsilon')$ there exists $y\in X$ with $d(x,y)\leq\epsilon'-\epsilon$ and 
 $d(y,z)\leq\epsilon$. Hence, we can estimate
 \begin{align*}
  f^{\epsilon'}(x) 
  &=\frac{1}{\kappa(x)}\sup_{z\in B(x,\epsilon')} 
  	\Big(\max\big\{f(z)\kappa(z),-\tfrac{1}{{\epsilon'}}\big\}+\epsilon'\Big) \\
  &=\frac{1}{\kappa(x)}\sup_{y\in B(x,\epsilon'-\epsilon)}\sup_{z\in B(y,\epsilon)}
  	\Big(\max\big\{f(z)\kappa(z),-\tfrac{1}{\epsilon'}\big\}+\epsilon'\Big) \\
  &\leq c_{\epsilon'}(x)\sup_{y\in B(x,\epsilon'-\epsilon)}\frac{1}{\kappa(y)}\sup_{z\in B(y,\epsilon)}
  	\Big(\max\big\{f(z)\kappa(z),-\tfrac{1}{\epsilon'}\big\}+\epsilon'\Big),
 \end{align*}
 where $c_{\epsilon'}(x):=\sup_{y\in B(x,\epsilon'-\epsilon)}\frac{\kappa(y)}{\kappa(x)}$. 
 Continuity of $\kappa$ implies $c_{\epsilon'}\downarrow 1$ as $\epsilon'\downarrow\epsilon$.
 We obtain
 \[ \inf_{\epsilon'>\epsilon} f^{\epsilon'}(x)
 	\leq\inf_{\epsilon'>\epsilon}\sup_{y\in B(x,\epsilon'-\epsilon)}f^\epsilon(y)
	= \fe(x).\]
 The reverse estimate follows from inequality~\eqref{eq:fe1}.
\end{proof}

\section{Basic convexity estimates}\label{sec:convestim}

\begin{lemma} \label{lem:lambda}
 Let $\X$ be a vector space and $\phi\colon\mathcal{X}\to\R$ be a convex functional. 
 Then, 
 \[ \phi(x)-\phi(y)\leq\lambda\left(\phi\left(\frac{x-y}{\lambda}+y\right)-\phi(y)\right)
 	\quad\mbox{for all } x,y\in\X \mbox{ and } \lambda\in (0,1]. \]
\end{lemma}
\begin{proof}
 For every $x,y\in\X$ and $\lambda\in (0,1]$,
 \begin{align*}
  \phi(x)-\phi(y)
  &=\phi\left(\lambda\left(\frac{x-y}{\lambda}+y\right)+(1-\lambda)y\right)-\phi(y) \\
  &\leq\lambda\phi\left(\frac{x-y}{\lambda}+y\right)+(1-\lambda)\phi(y)-\phi(y) \\
  &=\lambda\left(\phi\left(\frac{x-y}{\lambda}+y\right)-\phi(y)\right). \qedhere
 \end{align*}
\end{proof}

\begin{corollary} \label{cor:kappa}
 Let $\Phi$ be a convex operator $\Phi\colon\Ck\to\Ck$ such that there exists $c\geq 0$ with
 $\|\Phi(f)\|_\kappa\leq c\|f\|_\kappa$ for all $f\in\Ck$. Then,
 \[ \Phi\left(f+\frac{a}{\kappa}\right)\leq\Phi(f)+\frac{c|a|}{\kappa}
 	\quad\mbox{for all } f\in\Ck \mbox{ and } a\in\R. \]
 Furthermore, if $\Phi$ is additionally monotone, then
 \[ \|\Phi f-\Phi g\|_\kappa\leq c\|f-g\|_\kappa \quad\mbox{for all } f,g\in\Ck. \]
\end{corollary}
\begin{proof}
 Let $f\in\Ck$ and $a\in\R$. For every $\lambda\in (0,1)$,
 \begin{align*}
  \Phi\left(f+\frac{a}{\kappa}\right)
  &\leq\lambda\Phi\left(\frac{1}{\lambda}f\right)+(1-\lambda)\Phi\left(\frac{a}{(1-\lambda)\kappa}\right) \\
  &\leq\lambda\Phi\left(\frac{1}{\lambda}f\right)+(1-\lambda)\frac{c|a|}{(1-\lambda)\kappa}
  =\lambda\Phi\left(\frac{1}{\lambda}f\right)+\frac{c|a|}{\kappa}.
 \end{align*}
 Lemma~\ref{lem:lambda} implies $\lambda\Phi\left(\frac{1}{\lambda}f\right)\to\Phi(f)$ as $\lambda\to 1$
 and the first part of the claim follows. Furthermore, if $\Phi$ is additionally monotone, one can 
 apply the previous estimate with $a:=\|f-g\|_\kappa$ to obtain the second part of the claim. 
\end{proof}

Let  $\Fk$ be the space of all functions $f\colon X\to\Rbar$ with $\|f^+\|_\kappa<\infty$.

\begin{lemma} \label{lem:lip}
 Let $\Phi\colon\Ck\to\Fk$ be a convex monotone operator with $\Phi 0=0$.
 Then, the following statements are valid:
 \begin{enumerate}
  \item For every $r\geq 0$, there exists $c\geq 0$ with
   \[ \|\Phi f\|_\kappa\leq c\|f\|_\kappa \quad\mbox{for all } f\in B_{\Ck}(r). \]
   One can choose $c=0$ for $r=0$ and $c:=\frac{1}{r}\big\|\Phi\frac{r}{\kappa}\big\|_\kappa$ for $r>0$.\
   In particular, the function $\Phi f$ is real-valued, i.e., $\Phi f\colon X\to\R$ for all $f\in\Ck$.
  \item For every $r\geq 0$, there exists $c\geq 0$ with
   \[ \|\Phi f-\Phi g\|_\kappa\leq c\|f-g\|_\kappa \quad\mbox{for all } f, g\in B_{\Ck}(r). \]
   One can choose $c:=0$ for $r=0$ and $c:=\frac{1}{r}\sup_{f'\in B_{\Ck}(3r)}\|\Phi f'\|_\kappa<\infty$ 
   for $r>0$.
 \end{enumerate}
 The previous statements remain valid if we replace $\Ck$ by $\Bk$. 
\end{lemma}
\begin{proof}
 First, let $r>0$, $f\in B_{\Ck}(r)$ and $\lambda:=\frac{\|f\|_\kappa}{r}$. We
 use the fact that $\Phi$ is convex and monotone with $\Phi 0=0$ to estimate
 \[ \Phi f=\Phi\big(\lambda\tfrac{1}{\lambda}f+(1-\lambda)0\big)
 	\leq\lambda\Phi\big(\tfrac{1}{\lambda}f\big)
 	\leq\lambda\Phi\tfrac{r}{\kappa}
 	=\tfrac{\|f\|_\kappa}{r}\Phi\tfrac{r}{\kappa}. \]
 Moreover, it follows from the convexity of $\Phi$ and $\Phi 0=0$ that
 \[ 0=\Phi 0=\Phi\big(\tfrac{1}{2}f+\tfrac{1}{2}(-f)\big)
 	\leq\tfrac{1}{2}\Phi f+\tfrac{1}{2}\Phi(-f). \]
 We conclude $(\Phi(\pm f))(x)>-\infty$ for all $x\in X$ and $-\Phi(-f)\leq\Phi f$. 
 Combining the previous estimates yields
 \[ -\tfrac{\|f\|_\kappa}{r}\Phi\tfrac{r}{\kappa}\leq -\Phi(-f)\leq\Phi f
 	\leq\tfrac{\|f\|_\kappa}{r}\Phi\tfrac{r}{\kappa}. \]
 Hence, it holds $\|\Phi f\|_\kappa\leq c\|f\|_\kappa$ with
 $c:=\frac{1}{r}\big\|\Phi\frac{r}{\kappa}\big\|_\kappa<\infty$. 
 
 Second, let $r\geq 0$ and $f,g\in B_{\Ck}(r)$. We define
 \[ \Phi_f\colon\Ck\to\Fk,\; f'\mapsto\Phi(f+f')-\Phi f \quad\mbox{for all } f'\in\Ck. \]
 Note, that $\|\Phi f\|_\kappa<\infty$ by the first part and therefore
 $\Phi f'\in\Fk$ for all $f'\in\Ck$. Furthermore, it follows from the first part that
 \[ \|\Phi f-\Phi g\|_\kappa=\big\|\Phi_f(f-g)\big\|_\kappa
 	\leq\frac{1}{2r}\big\|\big(\Phi_f(\tfrac{2r}{\kappa})\big)^+\big\|_\kappa\|f-g\|_\kappa
 	\leq c\|f-g\|_\kappa, \]
 where $c:=\frac{1}{r}\sup_{f'\in B_{\Ck}(3r)}\|\Phi f'\|_\kappa<\infty$.
\end{proof}

\section{Extension of convex monotone functionals}
\label{app:ext}

Denote by $\ca$ the set of all Borel measures $\mu\colon\B(X)\to [0,\infty]$ with 
$\int_X \frac{1}{\kappa}\,\d\mu<\infty$. Let $\phi\colon\Ck\to\R$ be a convex 
monotone functional with $\phi(0)=0$ and define 
\begin{equation}\label{def:convconjugate}
\phi^*\colon\ca\to [0,\infty],\; \mu\mapsto\sup_{f\in\Ck}\big(\mu f-\phi(f)\big), 
	\quad\mbox{where}\quad \mu f:=\int_X f\,\d\mu. 
 \end{equation}
Based on the results from~\cite{BCK19}, we obtain the following extension and 
dual representation result.

\begin{theorem} \label{thm:dual}
 Let $\phi\colon\Ck\to\R$ be a convex monotone functional with $\phi(0)=0$
 which is continuous from above. Then, the following statements are valid: 
 \begin{enumerate}
  \item For every $r\geq 0$, there exists a $\sigma(\ca,\Ck)$-compact convex set 
  $M_r\subset\ca$ with
  \[ \phi(f)=\max_{\mu\in M_r}\big(\mu f-\phi^*(\mu)\big) 
  	\quad\mbox{for all } f\in B_{\Ck}(r). \]
  Moreover, one can choose 
  $M_r:=\{\mu\in\ca\colon\phi^*(\mu)\leq\phi(\nicefrac{2r}{\kappa})
  	-2\phi(\nicefrac{-r}{\kappa})\}$.
  \item Define 
   $\bar{\phi}\colon\Uk\to\Rbar,\;f\mapsto\inf\{\phi(g)\colon g\in\Ck,\, g\geq f\}$.
   Then, the functional $\bar{\phi}$ is convex, monotone and the unique extension of $\phi$ 
   which is continuous from above. In addition, $\bar{\phi}$ admits the dual representation
   \[ \bar{\phi}(f)=\max_{\mu\in M_r}\big(\mu f-\phi^*(\mu)\big) 
   	\quad\mbox{for all } r\geq 0 \mbox{ and } f\in B_{\Uk}(r). \]
   \item Define 
    $\hat{\phi}\colon\Bk\to\Rbar,\; f\mapsto\lim_{c\to\infty}\sup_{\mu\in\ca}
    \big(\mu\big(\max\big\{f,-\tfrac{c}{\kappa}\big\}\big)-\phi^*(\mu)\big)$.
    Then, the functional $\hat{\phi}$ is convex, monotone and an extension of $\phi$.
    In addition, $\hat{\phi}$ admits the dual representation
    \[ \hat{\phi}(f)=\sup_{\mu\in M_r}\big(\mu f-\phi^*(\mu)\big) 
    	\quad\mbox{for all } r\geq 0 \mbox{ and } f\in B_{\Bk}(r), \]
    In particular, for every $\epsilon>0$ and $r\geq 0$, there exists $K\Subset X$ with 
    $\hat{\phi}\big(\frac{r}{\kappa}\one_{K^c}\big)<\epsilon$.
 \end{enumerate}
\end{theorem}
\begin{proof}
 First, we apply~\cite[Theorem~2.2]{BCK19} to obtain
 \[ \phi(f)=\max_{\mu\in\ca}\big(\mu f-\phi^*(\mu)\big) \quad\mbox{for all } f\in\Ck. \]
 Let $r\geq 0$ and $f\in B_{\Ck}(r)$. Choose $\mu\in\ca$ with $\phi(f)=\mu f-\phi^*(\mu)$.
 It follows from the definition of $\phi^*$ and the monotonicity of $\phi$ that
 \[ \mu\tfrac{2r}{\kappa}-\phi\big(\tfrac{2r}{\kappa}\big) \leq\phi^*(\mu)
 	=\mu f-\phi(f)\leq\mu\tfrac{r}{\kappa}-\phi\big(-\tfrac{r}{\kappa}\big). \]
 We obtain $\mu \tfrac{r}{\kappa}\leq\phi(\tfrac{2r}{\kappa})-\phi(-\tfrac{r}{\kappa})$
 and therefore $\phi^*(\mu)\leq\phi\big(\tfrac{2r}{\kappa})-2\phi\big(-\tfrac{r}{\kappa}\big)$.
 Hence, 
 \[ \phi(f)=\max_{\mu\in M_r}\big(\mu f-\phi^*(\mu)\big)
 	\quad\mbox{for all } f\in B_{\Ck}(r), \]
 where 
 $M_r:=\big\{\mu\in\ca\colon\phi^*(\mu)\leq\phi\big(\nicefrac{2r}{\kappa})
 	-2\phi\big(-\nicefrac{r}{\kappa}\big)\big\}$.
 Moreover, the set $M_r$ is convex and $\sigma(\ca,\Ck)$-compact, 
 see~\cite[Theorem~2.2]{BCK19}.

 Second, by monotonicity of $\phi$, the functional $\bar{\phi}$ is monotone and 
 an extension of $\phi$. We show that $\bar{\phi}(f)=\lim_{n\to\infty}\phi(f_n)$ 
 for all sequences $(f_n)_{n\in\N}\subset\Ck$ and $f\in\Uk$ with $f_n\downarrow f$. By 
 definition of the infimum, there exists a sequence $(g_k)_{k\in\N}\subset\Ck$ 
 such that $\phi(g_k)\to\bar{\phi}(f)$ as $k\to\infty$. Let $g_n^k:=f_n\vee g_k$ 
 for all $k,n\in\N$. Since $g_n^k\downarrow g_k$ as $n\to\infty$,  and $\phi$
 is monotone and continuous from above, we obtain
 \[ \lim_{n\to\infty}\phi(f_n)\leq\lim_{n\to\infty}\phi(g_n^k)=\phi(g_k)
 	\quad\mbox{for all } k\in\N. \] 
 The monotonicity of $\bar{\phi}$ implies 
 $\bar{\phi}(f)\leq\lim_{n\to\infty}\phi(f_n)\leq\lim_{k\to\infty}\phi(g_k)=\bar{\phi}(f)$.
 In particular, it follows that $\bar{\phi}$ is convex. Indeed, let $f,g\in\Uk$ and 
 $\lambda\in [0,1]$. Since $\Uk=(\Ck)_\delta$ and $\Ck$ is directed
 downwards, there exist sequences $(f_n)_{n\in\N}$ and $(g_n)_{n\in\N}$ 
 in $\Ck$ with $f_n\downarrow f$ and $g_n\downarrow g$. We obtain
 \begin{align*}
  \bar{\phi}(\lambda f+(1-\lambda)g)
  &=\lim_{n\to\infty}\phi(\lambda f_n+(1-\lambda)g_n) 
  \leq\lim_{n\to\infty}\big(\lambda\phi(f_n)+(1-\lambda)g_n\big) \\
  &=\lambda\bar{\phi}(f)+(1-\lambda)\bar{\phi}(g).
 \end{align*}

 Third, we show that $\bar{\phi}$ is continuous from above.  Let $(f_n)_{n\in\N}\subset\Uk$ 
 and $f\in\Uk$ with $f_n\downarrow f$. Since $\Uk=(\Ck)_\delta$, for every $k\in\N$, 
 there exists a sequence $(f_k^n)_{n\in\N}\subset\Ck$ with $f_k^n\downarrow f_k$
 as $n\to\infty$. Define $\tilde f_n:=\min\{f_1^n,\ldots,f_n^n\}\in\Ck$ for all 
 $n\in\N$. It holds
 \begin{align*}
  \tilde{f}_{n+1} &=\min\{f_1^{n+1},\ldots,f_n^{n+1},f_{n+1}^{n+1}\}
  	\leq\min\{f_1^n,\ldots,f_n^n\}=\tilde{f}_n \quad\mbox{for all } n\in\N, \\
  f_n &=\min\{f_1,\ldots,f_n\}\leq\min\{f_1^n,\ldots,f_n^n\}=\tilde{f}_n 
  	\quad\mbox{for all } n\in\N, \\
  \tilde{f}_n &=\min\{f_1^n,\ldots,f_k^n,\ldots,f_n^n\}\leq f_k^n
  	\quad\mbox{for all } k,n\in\N \mbox{ with } k\leq n.
 \end{align*}
 We obtain
 $f=\lim_{n\to\infty}f_n\leq\lim_{n\to\infty}\tilde f_n\leq\lim_{n\to\infty}f_k^n=f_k$
 for all $k\in\N$. Hence, it follows from the monotonicity of $\bar{\phi}$ and the second 
 part of the proof that
 \[ \bar{\phi}(f)\leq\lim_{n\to\infty}\bar{\phi}(f_n)\leq\lim_{n\to\infty}\phi(\tilde f_n)=\bar{\phi}(f). \]
 We have shown that $\bar{\phi}$ is continuous from above and thus~\cite[Theorem~2.2]{BCK19} 
 implies 
 \[ \bar{\phi}(f)=\max_{\mu\in\ca}\big(\mu f-\phi^*(\mu)\big) \]
 for all bounded $f\in\Uk$. 
 By the same arguments as in the first step, the maximum in the previous equation 
 can be taken over the set $M_r$ for all $r\geq 0$ and $f\in B_{\Uk}(r)$. 
 The uniqueness of $\bar{\phi}$ as extension, which is continuous from above,
 follows from $\Uk=(\Ck)_\delta$.

 Fourth, the functional $\hat{\phi}$ is clearly convex and monotone. Since
 $\bar{\phi}$ is continuous from above, we obtain
 \[ \bar{\phi}(f)=\lim_{c\to\infty}\bar{\phi}\big(\max\big\{f,-\tfrac{c}{\kappa}\big\}\big)
 	=\hat{\phi}(f) \quad\mbox{for all } f\in\Uk. \]
 Similar to the first part of this proof, it follows that the supremum in the definition
 of $\hat{\phi}$ can be taken over $M_r'$ for all $r\geq 0$ and $f\in B_{\Bk}(r)$. 
 By~\cite[Theorem~2.2]{BCK19}, the set $M_r'$ is $\sigma(\ca,\Ck)$-compact 
 and convex. The last statement follows from $\phi^*\geq 0$ and the fact that, by 
 Prokhorov's theorem, the set $\{\mu_\kappa\colon\mu\in M_r'\}$ is tight for all $r\geq 0$, 
 where $\mu_\kappa(A):=\int_A\frac{1}{\kappa}\,\d\mu$ for all $A\in\B(X)$.
\end{proof}

 Let $\phi\colon\Ck\to\R$ be a convex monotone functional with $\phi(0)=0$
 which is continuous from above at zero, i.e., $\phi(f_n)\downarrow 0$ for all 
 $(f_n)_{n\in\N}\subset\Ck$ with $f_n\downarrow 0$. Then, it follows from the 
 proof of~\cite[Theorem~2.2]{BCK19} that $\phi$ is continuous from above, 
 i.e., $\phi(f_n)\downarrow \phi(f)$ for all $(f_n)_{n\in\N}\subset\Ck$ and $f\in\Ck$
 with $f_n\downarrow f$. However, if we replace $\Ck$ by $\Uk$, this statement
 does not remain valid, because $\Uk$ is not a vector space.

\begin{lemma} \label{lem:cont.app}
	Let $(\phi_i)_{i\in I}$ be a family of convex monotone functionals $\phi_i\colon\Ck\to\R$ 
	with $\phi_i(0)=0$ and $\sup_{i\in I}\sup_{f\in B_{\Ck}(r)}|\phi_i(f)|<\infty \quad\mbox{for all } r\geq 0$.
	Then, the following two statements are equivalent:
	\begin{enumerate}
		\item[(i)] It holds $\sup_{i\in I}\phi_i(f_n)\downarrow 0$ for all sequences $(f_n)_{n\in\N}\subset\Ck$
		with $f_n\downarrow 0$. 
		\item[(ii)] For every $\epsilon>0$ and $r\geq 0$, there exist $c\geq 0$ and $K\Subset\Rd$ with 
		\[ \sup_{i\in I}|\phi_i(f)-\phi_i(g)|\leq c\|f-g\|_{\infty, K}+\epsilon \quad
		\mbox{for all }i\in I\mbox{ and }f,g\in B_{\Ck}(r).\] 
	\end{enumerate}
\end{lemma}
\begin{proof}
	Suppose that condition~(i) is satisfied. Let $\epsilon>0$ and $r\geq 0$. Since 
	$\phi_i\colon\Ck\to\R$ is continuous from above, we can apply Theorem~\ref{thm:dual}
	to obtain
	\begin{equation} \label{eq:dual}
		\phi_i(f)=\max_{\mu\in M_i}\big(\mu f-\phi_i^*(\mu)\big) 
		\quad\mbox{for all } i\in I \mbox{ and } f\in B_{\Ck}(r), 
	\end{equation}
	where 
	$M_i:=\{\mu\in\ca\colon\phi_i^*(\mu)\leq\phi_i(\tfrac{2r}{\kappa})-2\phi_i(-\tfrac{r}{\kappa})\}$.
    Since the convex monotone functional
	$\phi\colon\Ck\to\R$, $f\mapsto\sup_{i\in I}\phi_i(f)$ is continuous from above, \cite[Theorem~2.2]{BCK19} implies that
	$M:=\{\mu\in\ca\colon\phi^*(\mu)\leq\sup_{i\in I}
	(\phi_i(\tfrac{2r}{\kappa})-2\phi_i(-\tfrac{r}{\kappa}))\}$
	is $\sigma(\ca, \Ck)$-relatively compact. Hence, Prokhorov's theorem 
	yields $K\Subset\Rd$ with 
	$\sup_{\mu\in M}\int_{K^c}\frac{1}{\kappa}\,\d\mu\leq\frac{\epsilon}{2r}$. 
	We use equation~\eqref{eq:dual} and 
	$M_i\subset M$ to obtain
	\begin{align*}
		&|\phi_i(f)-\phi_i(g)| 
		\leq\sup_{\mu\in M_i}|\mu f-\mu g| 
		\leq\sup_{\mu\in M_i}\left(\int_K |f-g|\,\d\mu+\int_{K^c}|f-g|\,\d\mu\right) \\
		&\leq\sup_{\mu\in M_i}\left(\mu(K)\|f-g\|_{\infty, K}+\int_{K^c}\frac{2r}{\kappa}\,\d\mu\right) 
		\leq c\|f-g\|_{\infty, K}+\epsilon
	\end{align*}
	for all $i\in I$, $f,g\in B_{\Ck}(r)$ and $c:=\sup_{\mu\in M}\mu(K)\leq\phi(1)+\sup_{\mu\in M}\phi^*(\mu)<\infty$. 
	
	Suppose that condition~(ii) is satisfied. Let $f_n\downarrow 0$ and $r:=\|f_1\|_\kappa$. 
	For every $\epsilon>0$, there exist $c\geq 0$ and $K\Subset\Rd$ with 
	$\sup_{i\in I}\phi_i(f_n)\leq c\|f_n\|_{\infty, K}+\nicefrac{\epsilon}{2}$ for all $n\in\N$.
	Hence, by Dini's theorem, there exists $n_0\in\N$ with $\sup_{i\in I}\phi_i(f_n)\leq\epsilon$ 
	for all $n\geq n_0$. 
\end{proof}

\begin{theorem} \label{thm:limsup} 
 Let $(\phi_n)_{n\in\N}$ be a sequence of functionals $\phi_n\colon\Ck\to\R$ which satisfy 
 the following conditions:
 \begin{itemize}
  \item $\phi_n$ is convex and monotone with $\phi_n(0)=0$ for all $n\in\N$, 
  \item $\sup_{n\in\N}\sup_{f\in B_{\Ck}(r)}|\phi_n(f)|<\infty$ for all $r\geq 0$,
  \item $\sup_{n\in\N}\phi_n(f_k)\downarrow 0$ as $k\to\infty$ for all 
   $(f_k)_{k\in\N}\subset\Ck$ with $f_k\downarrow 0$.
 \end{itemize}
 For every $n\in\N$, Theorem~\ref{thm:dual} yields that the functional $\phi_n\colon\Ck\to\R$
 has a unique extension $\bar{\phi}_n\colon\Uk\to\Rbar$ which is continuous from above. 
 Let $(f_n)_{n\in\N}$ and $(g_n)_{n\in\N}$ be bounded sequences in $\Uk$. Then, 
 \[ \limsup_{n\to\infty}\bar{\phi}_n(f_n+g_n)\leq\limsup_{n\to\infty}\bar{\phi}_n(f+g_n), 
 	\quad\mbox{where}\quad f:=\Glimsup_{n\to\infty}f_n. \]
\end{theorem}
\begin{proof}
 First, we show that
 $\limsup_{n\to\infty}\bar{\phi}_n(f_n+g_n)\leq\limsup_{n\to\infty}\bar{\phi}_n(\fe+g_n)$
 for all $\epsilon\in (0,1]$. Since the functionals $(\phi_n)_{n\in\N}$ are uniformly 
 bounded, the functional
 \[ \phi\colon\Ck\to\R,\; f\mapsto\sup_{n\in\N}\phi_n(f) \]
 is well-defined, convex, monotone and satisfies $\phi(0)=0$. Furthermore, 
 it follows from Theorem~\ref{thm:dual}(ii) and $\phi^*\leq\phi_n^*$ that 
 \[ \bar{\phi}_n(f)=\max_{\mu\in M_r}\big(\mu f-\phi_n^*(\mu)\big)
 	\quad\mbox{for all } n\in\N,\, r\geq 0 \mbox{ and } f\in B_{\Uk}(r), \]
 where 
 $M_r:=\{\mu\in\ca\colon\phi^*(\mu)\leq\sup_{n\in\N}|\phi_n(\nicefrac{2r}{\kappa})
 	-2\phi_n(\nicefrac{-r}{\kappa})|\}$. 
 In the sequel, we fix $r:=\sup_{n\in\N}(\|f_n\|_\kappa+\|g_n\|_\kappa+1)$, 
 $M:=M_r$ and $\epsilon>0$. Since $\phi$ is continuous from above,
 by~\cite[Theorem~2.2]{BCK19} and Prokhorov's theorem, the set
 $\{\mu_\kappa\colon\mu\in M\}$ is tight, where 
 $\mu_\kappa(A):=\int_A\frac{1}{\kappa}\,\d\mu$. 
 Hence, there exists $K\Subset X$ with 
 $\sup_{\mu\in M}\mu_\kappa(K^c)(r+\nicefrac{1}{\epsilon})<\epsilon$.
 Moreover, by Lemma~\ref{lem:fe}, we can choose $n_0\in\N$ with
 $f_n(x)\leq\fe(x)$ for all $x\in K$ and $n\geq n_0$. 
 For every $n\geq n_0$, it follows from $\fe\geq -\frac{1}{\epsilon\kappa}$ that
 \begin{align*}
  \bar{\phi}_n(f_n+g_n)
  &=\max_{\mu\in M}\big(\mu(f_n+g_n)-\phi_n^*(\mu)\big) \\
  &\leq\max_{\mu\in M}\Big(\mu\Big(\fe+g_n+\big(\tfrac{r+\nicefrac{1}{\epsilon}}{\kappa}\big)
  	\one_{K^c}\big)\Big)-\phi_n^*(\mu)\Big) \\
  &\leq\max_{\mu\in M}\big(\mu(\fe+g_n)-\phi_n^*(\mu)\big)+\epsilon
    \leq\bar{\phi}_n(\fe+g_n)+\epsilon. 
 \end{align*}
 We obtain 
 $\limsup_{n\to\infty}\bar{\phi}_n(f_n+g_n)\leq\limsup_{n\to\infty}\bar{\phi}_n(\fe+g_n)+\epsilon$
 for all $\epsilon>0$.
 
 Second, we show 
 $\inf_{\epsilon\in (0,1]}\limsup_{n\to\infty}\bar{\phi}_n(\fe+g_n)
 	\leq\limsup_{n\to\infty}\bar{\phi}_n(f+g_n)$.
 Using the dual representation from the first part, we obtain
 \begin{align*}
  \inf_{\epsilon\in (0,1]}\limsup_{n\to\infty}\bar{\phi}_n(\fe+g_n)
  &=\inf_{n\in\N}\inf_{\epsilon\in (0,1]}\max_{\mu\in M}\sup_{k\geq n}
 	\big(\mu(\fe+g_k)-\phi_k^*(\mu)\big) \\
 &=\inf_{n\in\N}\inf_{\epsilon\in (0,1]}\max_{\mu\in M}
 	\big(\mu\fe-\alpha_n(\mu)\big),
 \end{align*}
 where $\alpha_n(\mu):=\inf_{k\geq n}(\phi_k^*(\mu)-\mu g_k)$. To
 interchange the maximum over $\mu\in M$ with the infimum over $\epsilon\in (0,1]$
 by using~\cite[Theorem~2]{Fan53}, we have to replace $\alpha_n$ by a convex
 lower semicontinuous function. It holds $\inf_{\mu\in M}\alpha_n(\mu)>-\infty$, 
 since $\phi_k^*\geq 0$ and $(g_k)_{k\in\N}$ is bounded. Hence, we 
 can define $\overline{\alpha}_n\colon M\to\R$ as the lower semicontinuous 
 convex hull of $\alpha_n$, i.e., the supremum over all lower semicontinuous 
 convex functions which are dominated by $\alpha_n$. Fenchel--Moreau's 
 theorem,~\cite[Theorem~2]{Fan53} and Lemma~\ref{lem:fe} imply
 \begin{align*}
  &\inf_{n\in\N}\inf_{\epsilon\in (0,1]}\max_{\mu\in M}
 	\big(\mu\fe-\alpha_n(\mu)\big)
  =\inf_{n\in\N}\inf_{\epsilon\in (0,1]}\max_{\mu\in M}
 	\big(\mu\fe-\overline{\alpha}_n(\mu)\big) \\
 &=\inf_{n\in\N}\max_{\mu\in M}\big(\mu f-\overline{\alpha}_n(\mu)\big) 
 =\inf_{n\in\N}\max_{\mu\in M}\big(\mu f-\alpha_n(\mu)\big)
 =\limsup_{n\to\infty}\bar{\phi}_n(f+g_n). \qedhere
 \end{align*}
\end{proof}

\section{Proof of Theorem~\ref{thm:I}}
\label{app:I}

We need the following version of Arz\'ela--Ascoli's theorem. A sequence $(f_n)_{n\in\N}$ 
of functions $f_n\colon X\to\R$ is called uniformly equicontinuous if and only if, for every 
$\epsilon>0$, there exists $\delta>0$ with $|f_n(x)-f_n(y)|<\epsilon$ for all $n\in\N$
and $x,y\in X$ with $d(x,y)<\delta$.

\begin{lemma} \label{lem:AA}
 Let $(f_n)_{n\in\N}\subset\Ck$ be bounded and uniformly equicontinuous. Then, there exist a 
 function $f\in\Ck$ and a subsequence $(n_l)_{l\in \N}$ with $f_{n_l}\to f$ uniformly on compacts.
\end{lemma}
\begin{proof}
 Let $D\subset X$ be countable and dense. By the Bolzano--Weierstrass theorem and a diagonalization 
 argument, there exists a subsequence $(n_l)_{l\in \N}$ such that the limit 
 \[ f(x):=\lim_{l\to\infty}f_{n_l}(x)\in\R \]
 exists for all $x\in D$. Since the mapping $D\to\R,\; x\mapsto f(x)$ is uniformly continuous 
 and satisfies $\sup_{x\in D}|f(x)|\kappa(x)<\infty$, there exists a unique extension $f\in\Ck$. 
 Since $(f_n)_{n\in\N}$ is uniformly equicontinuous, it holds $f_{n_l}(x_{n_l})\to f(x)$ for all 
 $x\in X$ and $x_{n_l}\to x$. In addition, we have $f_{n_l}\to f$ uniformly on compacts. Indeed, 
 assume by contradiction that there exist $\epsilon>0$, $K\Subset X$, a subsequence $(n_{l,1})_{l\in\N}$ 
 of $(n_l)_{l\in\N}$ and $x_{n_{l,1}}\in K$ with
 \[ |f_{n_{l,1}}(x_{n_{l,1}})-f(x_{n_{l,1}})|\geq\epsilon \quad\mbox{for all } l\in\N. \]
 Since $K$ is compact, we can choose a further subsequence $(n_{l,2})_{l\in\N}$ of 
 $(n_{l,1})_{l\in\N}$ and $x\in X$ with $x_{n_{l,2}}\to x$. It holds $f(x_{n_{l,2}})\to f(x)$ 
 and $f_{n_{l,2}}(x_{n_{l,2}})\to f(x)$ which leads to a contradiction. We obtain 
 $\lim_{l\to\infty}\|f-f_{n_l}\|_{\infty, K_{n_l}}=0$. 
\end{proof}

\begin{proof}[Proof of Theorem~\ref{thm:I}]
 First, we show that there exists a family $(S(t))_{t\geq 0}$ of convex monotone operators 
 $S(t)\colon\Ck\to\Ck$ with $S(t)0=0$ and $\D\subset\LI$ satisfying property~(ii) for all 
 $(f,t)\in\D\times\R_+$ and a subsequence $(n_l)_{l\in\N}\subset\N$ with
 \begin{equation} \label{eq:cher2.IS2}
  S(t)f=\lim_{l\to\infty}I(\pi_{n_l}^t)f \quad\mbox{for all } (f,t)\in\D\times\T. 
 \end{equation}
 Note that~\cite[Lemma~2.7-2.9]{BK23} which are applied in the sequel do not rely on the 
 relative compactness w.r.t. the norm topology required in~\cite[Assumption~2.4]{BK23}.
 Assumption~\ref{ass:I}(iii) and~(iv) and~\cite[Lemma~2.7]{BK23} imply
 \begin{equation} \label{eq:cher2.bound}
  I(\pi_n^t)f\in B_{\Ck}(\alpha(r,t)) \quad\mbox{and}\quad
  \|I(\pi_n^t)f-I(\pi_n^t)g\|_\kappa\leq e^{t\omega_{\alpha(r,t)}}\|f-g\|_\kappa
 \end{equation}
 for all $n\in\N$, $r,t\geq 0$ and $f,g\in B_{\Ck}(r)$. We use Assumption~\ref{ass:I}(v), 
 Lemma~\ref{lem:AA} and a diagonalization argument to choose a subsequence $(n_l)_{l\in\N}\subset\N$ 
 such that the limit 
 \[ S(t)f:=\lim_{l\to\infty}I(\pi_{n_l}^t)f\in\Ck \]
 exists for all $(f,t)\in\D\times\T$. Moreover, for every $f\in\D\subset\LI$, there exist
 $c\geq 0$ and $t_0>0$ such that~\cite[Lemma~2.8]{BK23} implies
 \begin{equation} \label{eq:Lset I}
  \|I(\pi_n^s)f-I(\pi_n^t)f\|_\kappa\leq ce^{T\omega_{\alpha(r,T)}}(|s-t|+h_n)
 \end{equation}
 for all $T\geq 0$, $s,t\in [0,T]$ and $n\in\N$ with $h_n\leq t_0$. For every $f\in\D$, 
 by~\cite[Lemma~2.9]{BK23} the mapping $\T\to\Ck,\; t\mapsto S(t)f$ has an extension to $\R_+$ 
 satisfying
 \begin{equation} \label{eq:Lset}
  \|S(s)f-S(t)f\|_\kappa\leq ce^{T\omega_{\alpha(r,T)}}|s-t|
  \quad\mbox{for all } T\geq 0 \mbox{ and } s,t\in [0,T]. 
 \end{equation}
 By construction, the operators $S(t)\colon\D\to\Ck$ are convex and monotone with $S(t)0=0$ 
 for all $t\geq 0$. Moreover, inequality~\eqref{eq:Lset I} implies $\D\subset\L^S$ and property~(ii) 
 is satisfied for all $(f,t)\in\D\times\R_+$.
 
 Second, we extend $(S(t))_{t\geq 0}$ from $\D$ to $\Ck$ and show that property~(ii) and~(iii)
 are valid as well as the inclusion $\LI\subset\LS$ and equation~\eqref{eq:cher2.IS}. We also 
 show that the mapping $t\mapsto S(t)f$ is continuous for all $f\in\Ck$. For every $\epsilon>0$, 
 $r,T\geq 0$ and $K\Subset X$, due to Assumption~\ref{ass:I}(vi) and equation~\eqref{eq:cher2.IS2}, 
 there exist $K'\Subset X$ and $c\geq 0$ with
 \begin{equation} \label{eq:cont2}
  \|S(t)f-S(t)g\|_{\infty,K}\leq c\|f-g\|_{\infty,K'}+\epsilon
 \end{equation}
 for all $t\in [0,T]$ and $f,g\in B_{\Ck}(r)\cap\D$. For every $f\in B_{\Ck}(r)$, 
 Assumption~\ref{ass:I}(v) yields a sequence $(f_n)_{n\in\N}\subset B_{\Ck}(r)\cap\D$ with and 
 $f_n\to f$. Inequality~\eqref{eq:cher2.bound} and inequality~\eqref{eq:cont2} guarantee that the limit
 \[ S(t)f:=\lim_{n\to\infty}S(t)f_n\in\Ck \]
 exists and is independent of the choice of the sequence $(f_n)_{n\in\N}$. The properties~(ii) 
 and~(iii) are satisfied for arbitrary functions $f,g\in\Ck$ and the inclusion $\LI\subset\LS$ 
 follows from the fact that inequality~\eqref{eq:Lset I} is valid for all $f\in\LI$. Next, we 
 verify equation~\eqref{eq:cher2.IS}. Let $(f,t)\in\Ck\times\T$, $\epsilon>0$ and $K\Subset X$. 
 Define $r:=\|f\|_\kappa$ and choose $K'\Subset X$ and $c\geq 0$ such that Assumption~\ref{ass:I}(vi) 
 and inequality~\eqref{eq:cont2} are valid for arbitrary $f,g\in B_{\Ck}(r)$. Since 
 Assumption~\ref{ass:I}(v) yields $g\in B_{\Ck}(r)\cap\D$ with $\|f-g\|_{\infty,K'}<\epsilon$,
 we obtain
 \begin{align*}
  \|S(t)f-I(\pi^t_{n_l})f\|_{\infty,K} 
  &\le \|S(t)f-S(t)g\|_{\infty,K}+\|S(t)g-I(\pi^t_{n_l})g\|_{\infty,K} \\
  &\quad\; +\|I(\pi^t_{n_l})g-I(\pi^t_{n_l})f)\|_{\infty,K} \\
  &\leq 2c\|f-g\|_{\infty,K'}+2\epsilon+\|S(t)g-I(\pi^t_{n_l})g\|_{\infty,K} \\
  &\leq 2(c+1)\epsilon+\|S(t)g-I(\pi^t_{n_l})g\|_{\infty,K}.
 \end{align*}
 Hence, it follows from equation~\eqref{eq:cher2.IS2} that 
 \begin{equation} \label{eq:cher2.IS3}
  \lim_{l\to\infty}\|S(t)f-I(\pi^t_{n_l})f\|_{\infty,K}=0 \quad\mbox{for all } (f,t)\in\Ck\times\T.
 \end{equation}
 In addition, for every $t\geq 0$, $f\in\Ck$, $\epsilon>0$ and $K\Subset X$, Assumption~\ref{ass:I}(v)
 and the previously shown property~(iii) guarantee that there exists $g\in\D$ with
 \[ \|S(s)f-S(t)f\|_{\infty,K}\leq\|S(s)g-S(t)g\|_{\infty,K}+\epsilon
 	\quad\mbox{for all } s\in [0,t+1]. \]
 Hence, it follows from inequality~\eqref{eq:Lset} that $\lim_{s\to t}\|S(s)f-S(t)f\|_{\infty,K}=0$. 
 
 Third, we show that $S(0)f=f$ and $S(s+t)f=S(s)S(t)f$ for all $s,t\geq 0$ and $f\in\Ck$
 and conclude that $S(t)\colon\LS\to\LS$ for all $t\geq 0$. It follows from $I(0)f=f$ that
 $S(0)f=f$ for all $f\in\Ck$. Let $r\geq 0$, $s,t\in\T$ and $f\in B_{\Ck}(r)\cap\D$. 
 For every $n\in\N$, 
 \begin{align*}
  S(s+t)f-S(s)S(t)f 
  &=\big(S(s+t)f-I(\pi_n^{s+t})f\big)+\big(I(\pi_n^{s+t})f-I(\pi_n^s)I(\pi_n^t)f\big) \\
  &\quad\; +\big(I(\pi_n^s)I(\pi_n^t)f-I(\pi_n^s)S(t)f\big)+\big(I(\pi_n^s)S(t)f-S(s)S(t)f\big).
 \end{align*}
 It follows from equation~\eqref{eq:cher2.IS3} that the first and last term on the right-hand
 side convergence to zero for the subsequence $(n_l)_{l\in\N}$. Furthermore, we use 
 inequality~\eqref{eq:cher2.bound}, $k_n^{s+t}-k_n^s-k_n^t\in\{0,1\}$ and $f\in\D\subset\LI$
 to obtain
 \[ \|I(\pi_n^{s+t})f-I(\pi_n^s)I(\pi_n^t)f\|_\kappa
 	\leq e^{(s+t)\omega_{\alpha(r,s+t)}}\|I(h_n)f-f\|_\kappa\to 0 \quad\mbox{as } n\to\infty. \]
 Since $I(\pi_n^t)f, S(t)f\in B_{\Ck}(\alpha(r,t))$ for all $n\in\N$, for every $\epsilon>0$ and 
 $K\Subset X$, Assumption~\ref{ass:I}(vi) yields $K'\Subset X$ and $c\geq 0$ with
 \[ \|I(\pi_n^s)I(\pi_n^t)f-I(\pi_n^s)S(t)f\|_{\infty,K}
 	\leq c\|I(\pi_n^t)f-S(t)f\|_{\infty,K'}+\epsilon \quad\mbox{for all } n\in\N. \]
 Equation~\eqref{eq:cher2.IS3} guarantees that the previous term converges to zero for the
 subsequence $(n_l)_{l\in\N}$ and we obtain $S(s+t)f-S(s)S(t)f=0$ for all $s,t\in\T$ and $f\in\D$.
 Furthermore, Assumption~\ref{ass:I}(v) and the previously shown property~(iii) imply
  \[ S(s+t)f=S(s)S(t)f \quad\mbox{for all } s,t\in\T \mbox{ and } f\in\Ck. \]
 In order to extend the previous equation to arbitrary times $s,t\geq 0$, we choose 
 sequences $(s_n)_{n\in\N}\subset [0,s]\cap\T$ and $(t_n)_{n\in\N}\subset [0,t]\cap\T$
 with $s_n\to s$ and $t_n\to t$. For every $n\in\N$, 
 \begin{align*}
  S(s+t)f-S(s)S(t)f
  &=\big(S(s+t)f-S(s_n+t_n)f\big)+\big(S(s_n)S(t_n)f-S(s_n)S(t)f\big) \\
  &\quad\; +\big(S(s_n)S(t)f-S(s)S(t)f\big).
 \end{align*}
 Since we have already verified the properties~(ii) and~(iii) and shown that $(S(t))_{t\geq 0}$ 
 is strongly continuous, the terms on the right-hand side converge to zero as $n\to\infty$. 
 In addition, for every $f\in\LS$ and $t\geq 0$, there exist $c\geq 0$ and $h_0>0$ with 
 \begin{align*} 
  \|S(h)S(t)f-S(t)f\|_\kappa &=\|S(t)S(h)f-S(t)f\|_\kappa \\
  &\leq e^{t\omega_{\alpha(r,t+h)}}\|S(h)f-f\|_\kappa\leq ce^{t\omega_{\alpha(r,t+h)}}h
 \end{align*}
 for $r:=\|f\|_\kappa$ and all $h\in [0,h_0]$ showing that $S(t)f\in\LS$. 
  
 Fourth, we show that $f\in D(A)$ and $Af=I'(0)f$ for all $f\in\Ck$ such that 
 \[ I'(0)f=\lim_{h\downarrow 0}\frac{I(h)f-f}{h}\in\Ck \]
 exists. Let $\epsilon>0$ and $K\Subset X$. Define $g:=I'(0)f$ and 
 \[ r:=\sup_{n\in\N}\max\left\{\|I(h_n)f\|_\kappa, \|f+h_n g\|_\kappa, 
    \left\|\frac{I(h_n)f-f}{h_n}-g\right\|_\kappa\right\}<\infty. \]
 By Assumption~\ref{ass:I}(vi), there exist $K'\Subset X$ and $c\geq 0$ with
 \begin{equation} \label{eq:cont3}
  \|I(h_n)^k f_1-I(h_n)^k f_2\|_{\infty,K}\leq c\|f_1-f_2\|_{\infty, K'}+\frac{\epsilon}{4}
 \end{equation}
 for all $k,n\in\N$ with $kh_n\leq 1$ and $f_1,f_2\in B_{\Ck}(2r)$. W.l.o.g, we assume that 
 $K\subset K'$. Since Assumption~\ref{ass:I}(v) yields that $\LI\subset\Ck$ is dense, 
 we can use Assumption~\ref{ass:I}(vi) and argue similar to the proof of~\cite[Lemma~4.4]{BK23} 
 to choose $t_0\in (0,1]$ with
 \begin{equation} \label{eq:cond (4.2)}
  \left\|\frac{I(h_n)^k(f+h_n g)-I(h_n)^k f}{h_n}-g\right\|_{\infty,K}
  \leq\frac{\epsilon}{2}
 \end{equation}
 for all $k,n\in\N$ with $kh_n\leq t_0$. By definition of $g$, we can further suppose that
 \begin{equation} \label{eq:cher2.gen1}
  \left\|\frac{I(h)f-f}{h}-g\right\|_{\infty,K'}\leq\frac{\epsilon}{4c}
  \quad\mbox{for all } h\in (0,t_0]. 
 \end{equation}
 By induction, we show that, for every $k,n\in\N$ with $kh_n\leq t_0$, 
 \begin{equation} \label{eq:cher2.gen2}
  \left\|\frac{I(h_n)^k f-f}{kh_n}-g\right\|_{\infty,K}\leq\epsilon. 
 \end{equation}
 For $k=1$, the previous inequality holds due to inequality~\eqref{eq:cher2.gen1}. 
 Moreover, for every $k\in\N$, Lemma~\ref{lem:lambda} implies
 \begin{align*}
  &-\left(I(h_n)^k\left(g-\frac{I(h_n)f-f}{h_n}+I(h_n)f\right)-I(h_n)^k I(h_n)f\right) \\
  &\leq\frac{I(h_n)^k I(h_n)f-I(h_n)^k(f+h_n g)}{h_n} \\
  &\leq I(h_n)^k\left(\frac{I(h_n)f-f}{h_n}-g+f+h_n g\right)-I(h_n)^k(f+h_n g).
 \end{align*}
 It follows from inequality~\eqref{eq:cont3} and inequality~\eqref{eq:cher2.gen1} that
 \[ \left\|\frac{I(h_n)^k I(h_n)f-I(h_n)^k(f+h_n g)}{h_n}\right\|_{\infty,K}
 	\leq c\left\|\frac{I(h_n)f-f}{h_n}-g\right\|_{\infty,K'}+\frac{\epsilon}{4}
 	\leq\frac{\epsilon}{2}. \]
 If inequality~\eqref{eq:cher2.gen2} is valid for a fixed $k\in\N$, we can use 
 inequality~\eqref{eq:cond (4.2)} to conclude
 \begin{align*}
  \left\|\frac{I(h_n)^{k+1}f-f}{(k+1)h_n}-g\right\|_{\infty,K}
  &\leq\frac{1}{k+1}\left\|\frac{I(h_n)^k I(h_n)f-I(h_n)^k(f+h_n g)}{h_n}\right\|_{\infty,K} \\
  &\quad\; +\frac{1}{k+1}\left\|\frac{I(h_n)^k(f+h_n g)-I(h_n)^k f}{h_n}-g\right\|_{\infty,K} \\
  &\quad\; +\frac{k}{k+1}\left\|\frac{I(h_n)^k f-f}{kh_n}-g\right\|_{\infty,K} \\
  &\leq\frac{1}{k+1}\cdot\frac{\epsilon}{2}+\frac{1}{k+1}\cdot\frac{\epsilon}{2}
  +\frac{k}{k+1}\epsilon =\epsilon.
 \end{align*}
 For every $t\in (0,t_0]\cap\T$, equation~\eqref{eq:cher2.IS3} and inequality~\eqref{eq:cher2.gen2} 
 yield
 \[ \left\|\frac{S(t)f-f}{t}-g\right\|_{\infty,K}
 	=\lim_{l\to\infty}\left\|\frac{I(\pi_{n_l}^t)f-f}{k_{n_l}^t h_{n_l}}-g\right\|_{\infty,K}\leq\epsilon. \]
 Since $(S(t))_{t\geq 0}$ is strongly continuous, the previous estimate remains valid for arbitrary 
 times $t\in (0,t_0]$ showing that
 \[ \lim_{h\downarrow 0}\left\|\frac{S(h)f-f}{h}-g\right\|_{\infty,K}=0. \qedhere \]
\end{proof}

\bibliographystyle{abbrv}
\bibliography{BibThesis}

\end{document}